\documentclass[11pt]{amsart}  
\usepackage{amssymb}
\usepackage{amscd}
\usepackage{amsmath}

\newtheorem{thm}{Theorem}[section]
\newtheorem{cor}[thm]{Corollary}
\newtheorem{corollary}[thm]{Corollary}
\newtheorem{prop}[thm]{Proposition}
\newtheorem{proposition}[thm]{Proposition}
\newtheorem{lemma}[thm]{Lemma}

\theoremstyle{remark}
\newtheorem{remark}[thm]{Remark} 
\newtheorem{example}[thm]{Example}

\theoremstyle{definition}
\newtheorem{definition}[thm]{Definition}

\numberwithin{equation}{subsection} 

\newcommand{\bfC}{\mathtt{\mathbf C}} 
\newcommand{\bfM}{\mathtt{\mathbf M}}
\newcommand{\bfN}{\mathtt{\mathbf N}}
\newcommand{\bfx}{\mathtt{\mathbf x}}
\newcommand{\bfy}{\mathtt{\mathbf y}}
\newcommand{\T}{\mathtt{\mathbf T}}
\newcommand{\bfV}{\mathtt{\mathbf V}}
\newcommand{\bfW}{\mathtt{\mathbf W}}
\newcommand{\bfY}{\mathtt{\mathbf Y}}
\newcommand{\hor}{\operatorname{hor}}
\newcommand{\dgmod}{\operatorname{dgmod}}
\newcommand{\Modinf}{\operatorname{Mod}_\infty}
\newcommand{\Conv}{\operatorname{Conv}}
\newcommand{\oConv}{\overline{\Conv}}
\newcommand{\bConv}{\mathbf{Conv}}
\newcommand{\bB}{\mathbb B}
\newcommand{\bof}{\mathbf f}
\newcommand{\bog}{\mathbf g}
\newcommand{\boh}{\mathbf h}
\newcommand{\Jets}{\operatorname{Jets}}
\newcommand{\Coker}{\operatorname{Coker}}
\newcommand{\ML}{\operatorname{ML}}

\newcommand{\Op}{\operatorname{Op}}
\newcommand{\GL}{\operatorname{GL}}

\newcommand{\Sp}{\operatorname{Sp}}
\newcommand{\tSp}{\widetilde{\operatorname{Sp}}}

\newcommand{\tM}{\widetilde M}

\newcommand{\g}{\mathfrak{g}}
\newcommand{\A}{{\widehat{\mathbb A}}}

\newcommand{\caH}{{\mathcal H}}

\newcommand{\E}{{\mathcal E}}
\newcommand{\caP}{{\mathcal P}}
\newcommand{\cB}{{\mathcal B}}

\newcommand{\bV}{{\mathbb V}}
\newcommand{\C}{{\mathbb C}}
\newcommand{\fC}{\widehat{\mathbb C}}

\newcommand{\Z}{{\mathbb Z}}
\newcommand{\R}{{\mathbb R}}

\newcommand{\oih}{\frac{1}{i\hbar}}

\newcommand{\tn}{\widetilde{\nabla}}
\newcommand{\tnabla}{\widetilde{\nabla}}
\newcommand{\tiP}{{\widetilde{P}}}

\newcommand{\cont}{\operatorname{cont}}

\newcommand{\Ad}{\operatorname{Ad}}
\newcommand{\tIsom}{{\widetilde{\operatorname{Isom}}}}

\newcommand{\Rees}{\operatorname{Rees}}

\newcommand{\End}{\operatorname{End}}
\newcommand{\Hom}{\operatorname{Hom}}

\newcommand{\uRHOM}{{\underline{\bR\operatorname{HOM}}}}

\newcommand{\uHom}{\underline{\operatorname{Hom}}}

\newcommand{\tAd}{\widetilde{\operatorname{Ad}}}
\newcommand{\Der}{{\mathtt{Der}}}

\newcommand{\proj}{\operatorname{proj}}

\newcommand{\geB}{\mathfrak B}

\newcommand{\geX}{\mathfrak X}

\newcommand{\geU}{\mathfrak U}

\newcommand{\cucup}{\cup}
\newcommand{\ddx}{\frac{\partial}{\partial x}}
\newcommand{\ddfx}{\frac{\partial}{\partial \fx}}
\newcommand{\ddxi}{\frac{\partial}{\partial \xi}}
\newcommand{\ddfxi}{\frac{\partial}{\partial \fxi}}
\newcommand{\ddz}{\frac{\partial}{\partial z}}
\newcommand{\ddfz}{\frac{\partial}{\partial \fz}}

\newcommand{\ddzeta}{\frac{\partial}{\partial \zeta}}
\newcommand{\ddfzeta}{\frac{\partial}{\partial \fzeta}}
\newcommand{\ddfxj}{\frac{\partial}{\partial \fx_j}}
\newcommand{\ddfxij}{\frac{\partial}{\partial \fxi_j}}

\newcommand{\uC}{\underline{ C}}
\newcommand{\uG}{\underline{ G}}

\newcommand{\ucH}{\underline{ \mathcal H}}
\newcommand{\ucG}{\underline{ \mathcal G}}

\newcommand{\ucQ}{\underline{ \mathcal Q}}
\newcommand{\bz}{\overline{ z}}
\newcommand{\tg}{\widetilde{ g}}
\newcommand{\tgg}{\widetilde{ \mathfrak g}}

\newcommand{\id}{{\mathtt{Id}}}

\newcommand{\HOM}{\operatorname{HOM}}

\newcommand{\wbA}{\widehat{\widehat \bA}}
\newcommand{\bfA}{\widehat{\bA}}

\newcommand{\gX}{\mathfrak X}

\newcommand{\bC}{{\mathtt{\mathbb C}}}

\newcommand{\bP}{{\mathtt{\mathbb P}}}
\newcommand{\bQ}{{\mathtt{\mathbb Q}}}
\newcommand{\bR}{{\mathtt{\mathbb R}}}
\newcommand{\bT}{{\mathtt{\mathbb T}}}

\newcommand{\bZ}{{\mathtt{\mathbb Z}}}
\newcommand{\cA}{{\mathtt{\mathcal A}}}
\newcommand{\cL}{{\mathtt{\mathcal L}}}
\newcommand{\cC}{{\mathtt{\mathcal C}}}
\newcommand{\cD}{{\mathtt{\mathcal D}}}
\newcommand{\cE}{{\mathtt{\mathcal E}}}
\newcommand{\cF}{{\mathtt{\mathcal F}}}
\newcommand{\cG}{{\mathtt{\mathcal G}}}
\newcommand{\cH}{{\mathtt{\mathcal H}}}
\newcommand{\cM}{{\mathtt{\mathcal M}}}
\newcommand{\cN}{{\mathtt{\mathcal N}}}
\newcommand{\cO}{{\mathtt{\mathcal O}}}
\newcommand{\cP}{{\mathtt{\mathcal P}}}
\newcommand{\cQ}{{\mathtt{\mathcal Q}}}
\newcommand{\cR}{{\mathtt{\mathcal R}}}
\newcommand{\cS}{{\mathtt{\mathcal S}}}
\newcommand{\cT}{{\mathtt{\mathcal T}}}

\newcommand{\cV}{{\mathtt{\mathcal V}}}
\newcommand{\cW}{{\mathtt{\mathcal W}}}
\newcommand{\muS}{{\mu S}}

\newcommand{\SSS}{{\mathtt{\operatorname{SS}}}}

\newcommand{\fx}{\widehat{x}}
\newcommand{\fxi}{\widehat{\xi}}
\newcommand{\DR}{\mathtt{DR}}
\newcommand{\isomoto}{\stackrel{\sim}{\rightarrow}}
\newcommand{\otomosi}{\stackrel{\sim}{\leftarrow}}
\newcommand{\Ker}{\operatorname{Ker}}

\newcommand{\Ob}{\operatorname{Ob}}

\newcommand{\Ext}{\bf \operatorname{Ext}}

\newcommand{\Lie}{\operatorname{Lie}}

\newcommand{\MP}{\operatorname{MPar}}
\newcommand{\Mp}{\operatorname{Sp^4}}
\newcommand{\Mpp}{\operatorname{Mp}}

\newcommand{\ffAbM}{{\widehat{\widehat{\bA^\bullet }}}_M}
\DeclareMathOperator{\ad}{\mathtt{ad}}

\newcommand{\bA}{\mathtt{\mathbb A}}

\newcommand{\bK}{{\mathbb K}}
\newcommand{\fy}{\widehat{y}}
\newcommand{\fz}{\widehat{z}}

\newcommand{\feta}{\widehat{\eta}}
\newcommand{\fzeta}{\widehat{\zeta}}
\newcommand{\fe}{\widehat{e}}
\newcommand{\RHom}{\operatorname{\R Hom}}

\newcommand{\tG}{{\widetilde{\mathbf G}}}
\newcommand{\tP}{{\widetilde{\mathbf P}}}
\newcommand{\tiG}{{\widetilde{ G}}}

\newcommand{\Omhalf}{{\Omega^{\half}}}
\newcommand{\half}{\frac{1}{2}}
\newcommand{\fbA}{{\widehat{\bA}}}
\newcommand{\fbV}{{\widehat{\bV}}}
\newcommand{\fcV}{{\widehat{\cV}}}
\newcommand{\fotimes}{{\widehat{\otimes}}}

\newcommand{\ffbA}{\widehat{\widehat{\bA}}}
\newcommand{\ffbV}{\widehat{\widehat{\bV}}}
\newcommand{\wbV}{\widehat{\bV}}
\newcommand{\bg}{\overline{g}}

\newcommand{\Br}{\operatorname{Bar}}
\newcommand{\Cbr}{\operatorname{Cobar}}

\newcommand{\bfD}{\mathbf D}
\newcommand{\bfP}{\mathbf P}
\newcommand{\bfT}{\mathbf T}
\newcommand{\bfR}{\mathbf R}

\newcommand{\cAb}{\cA ^\bullet}
\newcommand{\cBb}{\cB ^\bullet}
\newcommand{\cCb}{\cC ^\bullet}
\newcommand{\cOb}{\cO ^\bullet}

\newcommand{\cVb}{\cV ^\bullet}
\newcommand{\cWb}{\cW ^\bullet}

\makeatletter
\renewcommand{\subsubsection}{\@startsection
{subsubsection}%
{2}%
{0mm}%
{-\baselineskip}%
{-0.5\baselineskip}%
{\normalfont\normalsize\bfseries }}%
\makeatother

\begin{document}

\bigskip

\centerline{\bf A microlocal category associated to a symplectic manifold}

\bigskip

{\centerline{Boris Tsygan}}

\bigskip

{\centerline{\em In memory of Boris Vasilievich Fedosov and Mosh\'{e} Flato}

\bigskip

{\bf Abstract} For a symplectic manifod subject to certain topological conditions, a category enriched in $A_\infty$ modules over the Novikov ring is constructed. The construction is based on the category of modules over Fedosov's deformation quantization algebra that have an additional structure, namely, an action of the fundamental groupoid up to inner automorphisms. Based in large part on the ideas of Bressler-Soibelman, Feigin, and Karabegov, it is motivated by the theory of Lagrangian distributions and and is related to other microlocal constructions of a category starting from a symplectic manifold, such as those due to Nadler-Zaslow and Tamarkin. In the case when our manifold is a flat two-torus, the answer is very close to both the microlocal category of Tamarkin and the Fukaya category as computed by Polishchuk and Zaslow.
\section{Introduction}\label{s:intro}
There are several ways to construct a category which is an invariant of a symplectic manifold. One is due to Fukaya and is based on Floer cohomology \cite{FOOO}, \cite{FOOO1}.  A connection between the Fukaya theory and theory of constructible sheaves was established by Nadler and Zaslow \cite{NZ}, \cite{N}. Another construction of a category starting from a symplectic manifold was carried out by Tamarkin \cite{T} and \cite{T1}. It is based on microlocal theory of sheaves on manifolds developed by Kashiwara and Schapira in \cite{KS}. 

In this paper we describe yet another construction. It is based on microlocal objects, as \cite{T} and \cite{T1} are. But instead of microlocal theory of sheaves we use asymptotics of pseudodifferential operators and Lagrangian distributions \cite{GS}, \cite{H}, or rather their algebraic version described by deformation quantization \cite{BFFLS}, \cite{F}, \cite{NT}, \cite{NT3}. 
\subsection{Motivation from Morse theory}\label{ss:MotMorse}
\subsubsection{The classical Morse filtration}\label{sss:classical Morse}
First recall that, given a function $f$ on a $C^\infty$ manifold $X$, one can study De Rham cohomology of $X$ using a  filtration of the sheaf $\bC_X$ by subsheaves $\bC_{X,t}=\bC_{\{f(x)\geq t\}}$ for any real $t$. If $f$ is a Morse function, the cohomology $H^\bullet (X, \bC_{X,t}/\bC_{X,t'})$ is described in terms of critical points of $f$.
\subsubsection{The filtered local system of $\bK$-modules}\label{sss:filteredLocSys}
The above can be interpreted as follows. Let 
$$\Lambda=\{\sum _{k=0}^\infty a_k \exp(\frac{1}{i\hbar}c_k )|a_k\in \bC;c_k\geq 0; c_k\to\infty\}$$
be the Novikov ring. Let $\bK$ be its field of quotients which is defined the same way as $\Lambda$, with the condition $c_k\geq 0$ replaced by $c_k\in \bR.$ Consider the trivial $\bK$-module of rank one and the corresponding constant sheaf $\bK_X$ on $X$. Given a function $f$, consider the action of the fundamental groupoid $\pi_1(X)$ on $\bK_X$ such that any class of a path $x_0\to x_1$ acts by multiplication by $\exp(\frac{1}{i\hbar}(f(x_0)-f(x_1))).$

For any real number $t$, denote by $C^\infty _{\Lambda^t,X}$ the sheaf associated to the presheaf of formal expressions
\begin{equation}\label{eq:C infty Lambda intro}
\{\sum _{k=0}^\infty a_k \exp(\frac{1}{i\hbar}\varphi_k )|a_k\in C_X^\infty((\hbar));\varphi_k\in C_X^\infty; \varphi_k\geq t; \varphi_k\to\infty\}
\end{equation}
Define $C^\infty_{\bK,X}$ the same way but without the condition $\varphi_k\geq t.$ When $t=0,$ we denote $C^\infty_{\Lambda^t,X}$ by $C^\infty_{\Lambda,X}.$

The fundamental groupoid $\pi_1(X)$ acts on $C^\infty_{\bK,X}$ (the simple exact meaning of this statement is explained in Definition \ref{dfn:A infty Lie groupoid}). Horizontal sections are of the form 
\begin{equation}\label{eq:hor sects}
\sum _k a_k \exp(\frac{1}{i\hbar}(c_k+f(x)))
\end{equation} 
where $a_k\in \bC((\hbar))$, $c_k\in \bR,$ and $c_k\to\infty.$ Now consider the sheaf $\cF^t(f)$ of horizontal sections that are in $C_{\Lambda^t,X}^\infty$. Note that $\exp(\frac{1}{i\hbar}(c+f))$ is in $\cF^t(f)$ on an open set if and only if $c\geq t-f$ on this open set. Therefore 
\begin{equation}\label{eq:cohomo of F t of f}
H^\bullet (X, \cF^t(f))={\widehat{\oplus}}_c H^\bullet (U_{c,t})((\hbar))
\end{equation}
where $U_{c,t}$ is the biggest open subset on which $c\geq t-f.$ We see that this cohomology essentially contains all the information about the cohomology of $X_t$ for various $t.$ The symbol ${\widehat{\oplus}}$ denotes the completed direct sum, {\em i.e.} the space of infinite sums 
\begin{equation}\label{eq:dirsum}
\sum_{k=1}^\infty A_k, \; A_k\in H^\bullet (U_{{c_k},t})((\hbar)),\; c_k\to\infty
\end{equation}
\subsubsection{The twisted De Rham complex}\label{sss:tw dr com intro}
The language of local systems and of actions of the fundamental groupoid makes it natural to look at flat connections. 
\begin{definition}\label{dfn:K forms}
Denote by $\Omega^\bullet_{\bK,X}$,  resp. $\Omega^\bullet_{\Lambda^t,X}$, resp. $\Omega^\bullet_{\Lambda,X}$, the sheaf of differential forms with coefficients in $C^\infty_{\bK,X}$, resp. $C^\infty_{\Lambda^t,X},$ resp. $C^\infty_{\Lambda,X}.$
\end{definition}
 Consider the twisted De Rham complex 
\begin{equation}\label{eq:tw DR intro}
(\Omega^\bullet_{\bK,X}, \,i\hbar d_{\DR}+df\wedge)
\end{equation}
This complex is filtered by subcomplexes $\Omega^\bullet_{\Lambda^t,X}.$ The fundamental groupoid acts on it preserving the differential (again, see Definition \ref{dfn:A infty Lie groupoid} for the exact meaning of this).

Now, for traditional local systems of finite dimensional vector spaces, locally, the cohomology of the De Rham complex is the same as the space of horizontal sections. The latter is (again, locally) the same as the derived space of horizontal sections, which is by definition the cohomology of the fundamental groupoid with coefficients in functions. In the context of $C^\infty_{\bK,M}$-valued forms, the first of these statements is false. In fact, the cohomology of the complex \eqref{eq:tw DR intro} is huge: regardless of $f$, it is the sum of cohomologies of $d_{\DR}+d\varphi\wedge$ for {\em all} $\varphi.$ But if we consider the local double complex of cochains of the fundamental groupoid with coefficients in \eqref{eq:tw DR intro}, we get the cohomology isomorphic to $\bK.$ This is easy to see. In fact, we can replace $f$ by $0$ in \eqref{eq:tw DR intro}, since the two complexes are isomorphic by means of multiplication by $\exp(\frac{1}{i\hbar}f)$. The value of the local double complex on a coordinate chart $U$ becomes 
$$\cC^{p,q}=\Omega^p_{\bK} (U^{q+1})$$
for $p,q\geq 0.$ There are two differentials: one is $d_{\DR}: \cC^{p,q}\to \cC^{p+1,q};$ the other is $\delta: \cC^{p,q}\to \cC^{p,q+1}$ where for $\omega\in \cC^{p,q}$
\begin{equation}\label{eq:delta}
\delta\omega=\sum_{j=0}^q (-1)^j p_j^* \omega
\end{equation}
Here $p_j$ is the projection $X^{q+1}\to X^q$ along the $j$th factor. But the differential $\delta$ admits a contracting homotopy
\begin{equation}\label{eq:delta 1}
h\omega=i_0^* \omega
\end{equation}
where $i_0(x_0,\ldots,x_{q-1})=(0,x_0,\ldots,x_{q-1}).$ More precisely, $[\delta,h]=\id-r_0$ where $r_0=0$ for $q>0$ or $p>0,$ and $r_0 a=a(0)$ for $p=q=0.$

The sheaf associated to the presheaf of local complexes $\cC^{\bullet,\bullet}$ inherits the action of the fundamental groupoid. The easiest way to express this is to say that, if 
\begin{equation}\label{eq:lim ind loc coh intro}
\cC^{p,q}_x=\varinjlim _{x\in U} \cC^{p,q}(U),
\end{equation}
then there are operators 
\begin{equation}\label{eq:lim ind loc coh intro 1}
\pi_1(x,y) \times \cC_y^{p,q}\to \cC^{p,q}_x
\end{equation}
that define an action. In a more general situation, when we start with a differential graded module $\cE^\bullet $ over $\Omega^\bullet_{\bK,X}$ with a compatible action of $\pi_1(X),$ they define an $A_\infty$ action. This is more or less the same for all practical purposes ({\em cf.} \ref{ss:A infty acciones}).

We summarize the above as follows. Starting from a function $f$ we constructed a filtered differential graded module $\cE^\bullet $ over $\Omega^\bullet_{\bK,X}$ with a compatible action of $\pi_1(X),$ namely the twisted De Rham complex \eqref{eq:tw DR intro}. From that we passed to a filtered $\bK_X$-module with an ({\em a priori} $A_\infty$) action of $\pi_1(X).$ It is natural to call such an object a filtered infinity local syslem of $\bK$-modules. (Note that the complex is filtered but $\pi_1$ does not preserve the filtration). The goal of this paper is to generalize large parts of the above in the way that we explain next.
\subsection{Lagrangian submanifolds}\label{ss:lag su} 
\subsubsection{Review of the results}\label{sss:Review of Results}
Let $M$ be a symplectic manifold and $L_0,$ $L_1$ its Lagrangian submanifolds. Under some topological assumptions that we will list below, we will construct an infinity-local system of $\bK$-modules $\cC^\bullet (L_0,L_1)$ on $M.$ In examples, this infinity local system  is often filtered. The precise topological conditions that guarantee it being filtered are yet to be determined. Complexes $\cC^\bullet(L_0,L_1)$  have a structure of an $A_\infty$ -category enriched in $A_\infty$ local systems of $\bK$-modules (we will develop this in detail in a subsequent work). When $M=T^*X$, $L_0={\rm{graph}}(0),$ and $L_1={\rm{graph}}(df),$ we recover the construction we discussed above (with some modification).

The topological conditions, most probably much too conservative for large parts of the construction, are as follows.

1) The manifold $M$ has an $\Sp^4$-structure ({\em cf.} \ref{ss:SpN}). In other words, for an almost complex structure compatible with $\omega,$ consider the first Chern class $c_1(M)$ of the tangent bundle viewed as a complex vector bundle. Then $2c_1(M)$ must be trivial in $H^2(M,\bZ/4\bZ).$ An $\Sp^4$ structure is a trivialization of $2c_1(M).$

2) The image of the pairing of the class of the symplectic form with the image of the Hurewicz morphism is zero: $\langle\pi_2(M), [\omega]\rangle=0.$

(The properties of Lagrangian submanifolds that are usually considered in Fukaya theory, such as exactness, grading, and existence of a ${\rm{Spin}}$ structure, all make their appearance in our considerations, as well as in \cite{T1}. Their exact role will be discussed in a subsequent work).

The infinity local system will be constructed in several steps indicated below. The meaning of all the terms used will be explained later in the introduction and/or in the rest of the article. All steps are possible under some additional conditions.

a) We will introduce a sheaf of algebras $\cA_M$ with a flat connection on $M.$ On this sheaf, the fundamental groupoid $\pi_1(M)$ will act up to inner automorphisms. Denote by $\cA_M^\bullet$ the differential graded algebra of $\cA_M$-valued forms, with the differential given by the connection.

b) Consider two modules $\cV$ and $\cW$ over $\cA_M$ with a compatible action of $\pi_1(M)$ and a compatible connection. Denote by $\cV^\bullet,$ $\cW^\bullet$ the differential graded modules of forms with values in $\cV$ or $\cW.$ Then the standard complex computing their ${\rm{Ext}}$ over $\cA_M^\bullet$ has a structure of a $\Omega^\bullet_{\bK,M}$-module with a (twisted) $A_\infty$ action of $\pi_1(M)$.

c) Given an $\Omega^\bullet_{\bK,M}$-module with a (twisted) $A_\infty$ action of $\pi_1(M),$ we will construct an infinity local system as in \eqref{eq:lim ind loc coh intro}.

d) To construct modules $\cV$ as in b), note that we can start with an $\cA_M$-module with a compatible connection and a compatible action of a bigger groupoid $\tG_M$ that maps onto $\pi_1(M)$ in such a way that the kernel of this map acts by inner automorphisms. 

e) Given a Lagrangian submanifold $L,$ we notice that there exists a subgroupoid of $\tG_M|L$ on $L,$ as well as an $\cA_M|L$-module with a compatible connection and a compatible action of this subgroupoid. Now we can get an object as in d) by an induction procedure.

We will now outline the steps a)-e) in more detail.
\subsection{Deformation quantization}\label{ss:deq}
\subsubsection{The twisted De Rham complex, deformation quantization, and ${\rm{Ext}}$ functors}\label{sss:Tw DR, def quant, Ext} The fact that the twisted De Rham complex can be interpreted in terms of homological algebra had been known for a long time. Namely, let $\cD_{\hbar}(X)$ be the ring of $C^\infty$ $\hbar$-differential operators, {\em i.e.} the subalgebra of all differential operators which is generated, in any local coordinate system, by $F(x_1,\ldots,x_n)$ for all functions $F$ and by $i\hbar \frac{\partial}{\partial x_j}$ for all $j$. Here $\hbar$ can be any nonzero number, but it is easy to modify this construction to make $\hbar$ a formal parameter (in which case $\cD_\hbar(X)$ is the Rees ring \cite{Borel}). The algebra $\cD_\hbar(X)$ acts on the space of functions on $X$. Denote the corresponding module by $V_0.$ Now note that a function $f$ defines an automorphism of $\cD_\hbar(X)$, namely the conjugation with $\exp(\frac{1}{i\hbar}f).$ When $\hbar$ is not a number but a formal parameter, it is not clear how to define $\exp(\frac{1}{i\hbar}f)$ but conjugation by it makes perfect sense. Namely, in any coordinate system it sends $F(x_1,\ldots, x_n)$ to itself for all $F$ and $i\hbar\frac{\partial}{\partial x_j}$ to $i\hbar\frac{\partial}{\partial x_j}+\frac{\partial f}{\partial x_j}$ for all $j$. It can be easily shown that ${\rm{Ext}}^\bullet_{\cD_\hbar}(V_0,V_f)$ can be computed by the twisted De Rham complex. When $\hbar$ is a nonzero number, this complex is of course isomorphic to the standard De Rham complex. When $\hbar$ is a formal parameter, this complex is 
\begin{equation}\label{eq:tw De Rham intro}
(\Omega^\bullet(X)[\hbar], i\hbar d_{\DR}+df\wedge)
\end{equation}
When we formally invert $\hbar$ the cohomology of this differential becomes easier to compute because we can use the spectral sequence associated to the filtration by powers of $\hbar$. The first differential in this spectral sequence is $df\wedge.$ When $f$ has isolated nondegenerate critical points, the cohomology of this differential, and therefore the cohomology of the twisted De Rham complex, is concentrated in the top degree $n$ and its dimension over the field $\bC((\hbar))$ of Laurent series is equal to the number of critical points.

Now let $\bA_M$ be a deformation quantization of $C^\infty (M)$ ({\em cf.} \cite{BFFLS}; we recall the definitions in \ref{sss:defquadef}). When $M=T^*X,$ there is the canonical deformation quantization that is a certain completion of $\cD_\hbar(X)$. (Another, arguably more correct, deformation is a completion of the algebra of $\hbar$-differential operators on half-forms). The algebra $\bA_M$ is a reasonable replacement of $\cD_\hbar(X)$, although it is no longer an algebra over $\bC[\hbar]$ but only over $\bC[[\hbar]].$ In particular it does not allow any specialization at a nonzero number $\hbar.$ 

In mid-eighties, Feigin suggested an idea based on the intuition from algebraic theory of $\cD$-modules \cite{Borel}. According to this idea, and to a subsequent work \cite{BS} of Bressler and Soibelman, one should associate to a Lagrangian submanifold $L$ a sheaf of $\bA_M$-modules $\bV_L$ supported on $L.$ Then ${\rm{Ext}}^\bullet (\bV_{L_0}, \bV_{L_1})$ should somehow be a first approximation for a more interesting theory, namely the Floer cohomology. The latter also sees intersection points of transversal Lagrangian submanifolds, but in a much subtler way. Those intersection points define cochains (not necessarily cocycles) of the Floer complex that are not of the same but of different degrees (given by the Maslov index). Furthermore, the differential in the Floer complex may send one such cochain to a linear combination of other points  (in other words, there may be instanton corrections). The standard homological algebra seems to be unable to catch these effects.

Below we will outline several tools that, combined, seem to allow to construct a category some (but not all) of whose objects come from Lagrangian submanifolds  and which is much closer to the Fukaya category than the bare category of $\bA_M$-modules. 
\subsubsection{The Fedosov construction}\label{sss:Fedosov intro} The work of Fedosov \cite{F} provided a simple and very efficient tool for working with deformation quantization of symplectic manifolds. Recall that a local model for deformation quantization is the Weyl algebra $C^\infty(M)[[\hbar]]$ with the Moyal-Weyl product $*$. The key properties of this product are that it is $\Sp(2n,\bR)$-invariant and that 
$$[\xi_j,x_k]=i\hbar \delta_{jk};[x_j,x_k]=[\xi_j,\xi_k]=0.$$
The local model for the Fedosov construction is as follows. Start with the space $\fbA$ of power series in formal variables $\fx_j,$ $\fxi_j,$ and $\hbar,$ $1\leq j\leq n.$ Turn it into an algebra by introducing the Moyal-Weyl product. Now consider the algebra of $\fbA$-valued differential forms on the Darboux chart with coordinates $x_j, \xi_j.$ This algebra is equipped with the differential given by formula 
\begin{equation}\label{eq:Fedosov flat intro}
\nabla_{\mathbb A}=\sum_{j=1}^n ((\frac{\partial}{\partial x_j}-\frac{\partial}{\partial \fx_j})dx_j+   (\frac{\partial}{\partial \xi_j}-\frac{\partial}{\partial \fxi_j})d\xi_j)
\end{equation}
({\em cf.} also \eqref{eq:Fedosov flat}). The cohomology algebra of this differential is the usual deformation quantization. 

For a general symplectic manifold $M$, one replaces a deformation $\bA_M$ with the algebra $\Omega^\bullet (M, \fbA_M)$ of $\fbA_M$-valued differential forms on $M.$ Here $\fbA_M$ is the bundle of algebras with fiber $\fbA$. The differential on the algebra $\Omega^\bullet (M, \fbA_M)$ is a chosen Fedosov connection. On any local Darboux chart, this algebra is isomorphic to the one discussed in the previous paragraph.

Note that the usual intuition about flat connections does not work here. Namely, there is no action of the fundamental groupoid (monodromy) preserving this flat connection. In fact, even locally, the algebra of horizontal sections is not at all isomorphic to the fiber. This feature will change rather radically after a modification that we introduce next.
Much of what follows is based on the idea suggested to the author by Alexander Karabegov: extend the work of Fedosov so that it will describe an asymptotic version of Maslov's theory of canonical operators and of H\"{o}rmander's theory of Lagrangian distributions ({\em cf.} \cite{GS}, \cite{H}, \cite{Leray}). Actually, the constructions below require nothing but a systematic introduction into deformation quantization of quantities of the form \eqref{eq:oscill intro} below. They do however have very strong connections to \cite{GS}, \cite{H}, \cite{Leray}. We discuss these connections in Appendices (Sections \ref{s:meta}, \ref{s:metaa}, \ref{s:LagrDistr}, and \ref{s:app d}). Note that exponentials \eqref{eq:oscill intro} were considered in deformation quantization since the introduction of the subject, in particular in \cite{BFFLS}, in \cite{DS} in \cite{Fedosov}.
\subsubsection{The extended Fedosov construction}\label{sss:extended Fedosov intro} Let us start with a remark about what happens when one tries systematically to introduce into deformation quantization quantities of the form 
\begin{equation}\label{eq:oscill intro}
\exp(\frac{1}{i\hbar}\varphi).
\end{equation} 
Let us do this at the level of the algebra of formal series $\fbA.$ All such quantities where $\varphi$ are power series starting with cubic terms become elements of a new algebra automatically as soon as one replaces $\fbA$ by a completion ${\widehat{\fbA}}$ ({\em cf.} \ref{ss:The alg curvA}). We interpret quantities \eqref{eq:oscill intro} where $\varphi$ are quadratic as elements of the $4$-fold covering group $\Mp(2n,\bR)$ (see the remark below). To add elements \eqref{eq:oscill intro} where $\varphi$ is constant, we tensor our algebra by the Novikov field $\bK$ (as in \ref{sss:filteredLocSys}).
\begin{remark}\label{rmk:why Sp4} Here is an explanation of the presence of $\Sp^4$ ({\em cf.} section \ref{s:meta} for definitions). The Lie algebra of derivations of the algebra $\bfA$ has a subalgebra consisting of elements $\frac{1}{i\hbar} \ad(q(\fx,\fxi))$ where $q$ is a quadratic function. This Lie subalgebra is isomorphic to ${\mathfrak{sp}}(2n),$ and its action is the standard action by linear coordinate changes. Consider the $\bfA$-module $\bC[[\fx, \hbar]][\hbar^{-1}]$ on which $\fx$ acts by multiplication and $\fxi$ by $i\hbar\ddfx.$ On it, $\frac{1}{i\hbar}\fx_j \fxi_k$ acts by $\fx_j \frac{\partial}{\partial \fx_k}+\frac{1}{2}\delta_{jk}.$ Note that $\ad(\frac{1}{i\hbar}\fx_j \fxi_k)$ form a basis of the subalgebra ${\mathfrak{gl}}(n)$ inside ${\mathfrak{sp}}(2n).$ We see that one can integrate the action of this Lie subalgebra on the module to an action of the group, put the most natural way to do this is to pass to the two-fold cover $\ML(n,\bR)$ consisting of pairs $\{(g,\zeta)|\det(g)=\zeta^2\}.$ One cannot extend this group action to the full symplectic group. To achieve that, we will have to extend the module considerably. But the group containing $\ML(n)$ is not $\Sp(2n)$ but its universal two-fold cover ${\rm{Mp}}(2n).$ The group $\Sp^4$ contains ${\rm{Mp}}(2n)$  as a normal subgroup with quotient $\bZ/2\bZ.$ We pass to this bigger group because it behaves better with respect to Lagrangian subspaces. For example, if a symplectic manifold $M$ has a real polarization, then $M$ has an $\Sp^4(2n)$-structure but not necessarily an  ${\rm{Mp}}(2n)$-structure. On a more basic level, the pre-image of $\GL(n,\bR)$ in $\Sp^4(2n, \bR)$ splits, {\em i.e.} is isomorphic to $\GL(n,\bR)\times \bZ/4\bZ.$
\end{remark} 

Finally, we do not add elements \eqref{eq:oscill intro} where $\varphi$ are linear, for the following reason. Note that $\ad(\frac{1}{i\hbar} \fxi_j)=\ddfxj$ and $\ad(\frac{1}{i\hbar} \fx_j)=-\ddfxij$. Exponentials of these operators should be shifts in formal variables $\fx_j$ and $\fxi_j.$ But such shifts do not act on power series. Instead, they should correspond to shifts acting from one fiber of the associated bundle of algebras to another. These shifts will be discussed in \ref{ss:shifts} below. One does not need to add them, they will act automatically as long as topological conditions 1), 2) from \ref{sss:Review of Results} are satisfied.

We get an algebra $\cA$ containing ${\widehat{\fbA}},$ $\bC[\Mp(2n)],$ and $\bK$ as subalgebras. The associated bundle of algebras $\cA_M$ carries a Fedosov connection $\nabla_\cA$ that extends the one on $\fbA_M.$ For all we know, the cohomology of the De Rham complex of this connection is huge. But the bundle of algebras $\cA_M$ carries another structure that we are going to discuss next.
\subsection{The action of $\pi_1$ up to inner automorphisms}\label{ss:shifts} It turns out that, if conditions 1) and 2) from \ref{sss:Review of Results} are satisfied, the fundamental groupoid $\pi_1(M)$ {\em acts on the bundle of algebras $\cA_M$ up to inner automorphisms}. The notion of such an action is defined in section \ref{s:action up to inner}. Moreover, the Fedosov connection $\nabla_\cA$ extends to a flat connection up to inner derivations compatible with this action ({\em cf.} \ref{sss:Conns up to inn}).

All the requisite notions are well-known and go back to Grothendieck. The version that suits our purposes is developed here in section \ref{s:action up to inner}. For the readers convenience we introduce these notions gradually, starting with the case of a group acting on an algebra, though the generality we need is that of a Lie groupoid acting on a sheaf of algebras. The Lie groupoid in question will be the fundamental groupoid or its extension by a bundle of Lie groups.
\subsection{From an action up to inner automorphisms to an $A_\infty$ local system}\label{ss:From act up to inn to loc sys intro}
In section \ref{s:From up to in to A infty} we explain that, given an action of a groupoid $\cG$ on a sheaf of algebras $\cA$ up to inner automorphisms and given two $\cA$-modules $\cV$ and $\cW$ with a compatible action of the groupoid, the standard complex $\cC^\bullet (\cV, \cA, \cW)$ that computes ${\rm{Ext}}^\bullet _{\cA}(\cV,\cW)$ carries a ({\em twisted}) $A_{\infty}$ action of $\cG.$ We make a similar argument when $\cA$ carries a flat connection up to inner derivations. (Twisted $A_\infty$ actions are discussed in \ref{appendix:infty and twisted}. They are needed because the action in \ref{ss:shifts} is continuous only locally).

Let $\cA_M^\bullet $ be the sheaf of $\cA_M$-valued forms on $M.$ The above procedure starts with two differential graded $\cA$-modules $\cV^\bullet,$ $\cW^\bullet$ with compatible actions of $\pi_1(M)$ and produces the standard complex $\cC^\bullet(\cV^\bullet, \cA^\bullet, \cW^\bullet)$ which is a sheaf of $\Omega_{\bK, M}^\bullet$-modules with a compatible {\em twisted} $A_\infty$ action of $\pi_1(M).$ Finally, for an open chart $U$ in $M,$ consider the double complex $\cC^{\bullet, \bullet}(\cV^\bullet, \cW^\bullet)(U)$ where $\cC^{p,q}(U)$ is the space of $q$-cochains of $\pi_1(U)$ with coefficients in the graded component $\cC^p(\cV^\bullet, \cA^\bullet, \cW^\bullet),$ as in the second part of \ref{sss:tw dr com intro}. Let 
$$\cC^{\bullet, \bullet}_x=\varinjlim_{x\in U} \cC^{\bullet, \bullet}(\cV^\bullet, \cW^\bullet)(U)$$
be the stalk at a point $x.$ As we indicated in \ref{sss:tw dr com intro} (after \eqref{eq:lim ind loc coh intro}), these complexes form an $A_\infty$ local system of $\bK$-modules. 
We  denote this local system by $\uRHOM(\cV^\bullet, \cW^\bullet).$

We sum up the construction up to this point in section \ref{s:Resume}.
\subsection{Objects constructed from Lagrangian submanifolds}. We proceed to construct a differential graded module $\cV_L$ as in \ref{ss:From act up to inn to loc sys intro} starting from a Lagrangian submanifold $L$. This is done using an induction procedure that is explained in Section \ref{s:Objects from Lagrangians}, in particular in \ref{ss:Objects from Lagrangians}. In Section \ref{s:Rhom int points Maslov}, we prove that the general construction, when applied to $M=\bR^{2n},$ $L_0={\rm{graph}}(0),$ and $L_1={\rm{graph}}(df),$ reproduces the one in \ref{sss:filteredLocSys}, with the one important distinction. Namely, the filtered $A_\infty$ local system $\uRHOM(\cV_{L_0}^\bullet, \cV_{L_1}^\bullet)$ whose construction is outlined above is a module over a trivial local system of differential graded algebras whose fiber is the algebra
\begin{equation}\label{eq:alg of cochains of MPar}
\cS^\bullet=C^\bullet (\MP(n), \bK)
\end{equation}
of cochains of the group $\MP(n)$ with coefficients in the Novikov field $\bK.$ Here $\MP(n)$ is the parabolic subgroup of the group $\Mp(2n)$ which is the pre-image of the stabilizer of the Lagrangian submanifold $\xi_1=\ldots=\xi_n=0$ in $\Sp(2n).$ 
We prove that the general construction outlined in \ref{ss:From act up to inn to loc sys intro} is the tensor product of $\cS^\bullet$ by the filtered local system described in  \ref{sss:filteredLocSys}.
\begin{remark} There probably exists a correct way of factoring out the maximal ideal of $\cS^\bullet$ and in particular recovering the exact answer as in \ref{sss:filteredLocSys}. Note that the algebra $\cS^\bullet$ plays a vital role in the computation in section \ref{s:Rhom int points Maslov}. Namely, the vanishing of the cohomology of $\MP(n)$ with coefficients in a certain class of modules leads to a vanishing result for all components involving a factor $\exp(\frac{1}{i\hbar} \varphi(x,\fx))$ where the quadratic part of $\varphi$ with respect to $\fx$ is nonzero. {\em Cf.} Lemma \ref{lemma:statfaza}, Corollary \ref{cor:statfaza 2} (which we interpret as stationary phase statements of some sort). 
\end{remark}
\subsubsection{The example of a two-dimensional torus}\label{ss:Ex torus intro} In \ref{s:Rhom theta}, we compute $\uRHOM(\cV^\bullet _{L_0}, \cV^\bullet _{L_m})$ where $M=\bR^2/\bZ^2,$ $L_0=\{\xi=0\},$ and $L_m=\{\xi=mx\}.$ The answer is the trivial bundle whose fiber is the space of matrices indexed by $k,\ell\in \bZ$ with coefficients in $\cS^\bullet.$ If $\gamma_1,\,\gamma_2$ are the two generators of the fundamental group $\pi_1(M)\isomoto \bZ^2,$ then the action of $\pi_1(M)$ on the matrix units ${\mathbf{E}}_{k\ell}$ is given by
$$\gamma_1^q \gamma_2^p: {\mathbf{E}}_{k\ell}\mapsto \exp(\frac{1}{i\hbar} ( \frac{mq^2}{2}+q(\ell-k))) {\mathbf{E}}_{k+p, \ell+p-mq}$$ 
As a consequence (Corollary \ref{cor:hor theta}), horizontal sections of this local system have the same algebraic expression as theta functions. This agrees with the computation of the Fukaya category of $M$ given by Polishchuk and Zaslow in \cite{PZ}.
\subsection{Microlocal category of sheaves}\label{ss:Sheaves intro}
\subsubsection{The microlocal category of Tamarkin} \label{ss:Tamarkin}
In \cite{T}, Tamarkin defined the category $D(T^*X)$ for a manifold $X$. This is a full subcategory of the differential graded category of complexes of sheaves on $X\times { \R}$. Below are the key properties of the differential graded category $D(T^*X).$

(1) For $c\geq 0,$ there is a natural transformation $\tau_c: \id \to (T_c)_*$ where, for $(x,t)\in X\times \R,$ $T_c(x,t)=(x,t+c).$ One has $\tau_c\tau_{c'}=\tau_{c+c'}.$ Define
$$\HOM(\cF, \cG)={\prod _{c\geq 0} } '\RHom (\cF, (T_c)_*\cG)$$
where $\prod '$ is the subset of the direct product consisting of all elements $(v_c)$ such that $v_c=0$ for all but countably many $c_k, k=1,2,\ldots,$ satisfying $c_k\to \infty.$
Then $\HOM(\cF, \cG)$ is a complex of modules over the Novikov ring $\Lambda_\bZ=\{\sum _{k=0}^{\infty} a_k e^{-\frac{c_k}{i\hbar}}\}$ where $a_k \in \Z,$ $c_k\in \R,$ $c_k\geq 0,$ and $c_k\to \infty.$

\begin{remark}\label{rmk:Tauc}
For a general sheaf $\cF$ there is no relation between its behavior on an open subset $U$ and on the shift of $U$ by $c$ in the $t$ direction. But Tamarkin's subcategory has a remarkable property that the natural transformation $\tau_c$ exists. A key example is provided by sheaves $\cF_f$ defined in the paragraph below.
\end{remark}

(2) For every object $\cF$ of $D(T^*X)$, a closed subset $\muS(\cF)$ is defined, called the microsupport of ${\mathcal F}.$
Let $f$ be a smooth function on $X.$ Denote $\cF_f=\Z_{\{t+f(x)\geq 0\}}.$ Then $\muS(\cF_f)={\operatorname{graph}}(df).$
(Observe that $T_c^*\cF_f=\cF_{f-c};$ the morphism $\tau_c:\cF\to T_{c*}\cF$ is the restriction to the subset $\{t-f-c\geq 0\}$ of $\bZ_{\{t-f\geq 0\}}$).

(3) For a Morse function $f$, the complex $HOM(\cF_0, \cF_f)$ is quasi-isomorphic to the Morse complex of $f.$

(4) Let ${\mathtt{\mathbf T}}^2$ be the standard 2-torus with the flat symplectic structure. One defines the category $D({\mathtt{\mathbf T}}^2)$ of objects of $D(T^*\R^1)$ equivariant under certain projective action of $\Z^2.$ For every Lagrangian submanifold of ${\mathtt{\mathbf T}}^2$ of the form $a\xi+bx=c,$ $a, b , c$ being integers, one constructs an object $\cF_{a,b,c}$ of $D({\mathtt{\mathbf T}}^2)$. The full subcategory generated by these objects is isomorphic to the full subcategory of the Fukaya category generated by Lagrangian submanifolds $a\xi+bx=c$ as computed by Polishchuk-Zaslow in \cite{PZ}.
\begin{remark}\label{rmk:what is DT2} The category $D({\mathtt{\mathbf T}}^2)$ can be defined either as a partial case of the general construction \cite{T1} or by an explicit procedure that we recall in \ref{ss:sheaves see theta}.
\end{remark}

(5) {\em Theorem B}. Let $\Phi$ be a Hamiltonian symplectomorphism of $T^*X$ which is equal to identity outside a compact subset. There exists a functor $T_\Phi: D(T^*X)\to D(T^*X)$ such that, if $\muS(\cF)$ is compact, $\muS(T_\Phi(\cF))\subset \Phi(\muS(\cF)).$ For every $\cF$ and $\cG,$ $\HOM(\cF, \cG)$ and $\HOM(\cF, T_\Phi(\cG))$ are isomorphic modulo $\Lambda_\bZ$-torsion. Similarly for $\HOM(\cF, \cG)$ and $\HOM(T_\Phi(\cF), \cG)$.

(6) {\em Theorem A}. Let $\cF$ and $\cG$ be objects of $D(T^*X)$ such that $\muS(\cF)$ and $\muS(\cG)$ are compact and do not intersect. Then $\HOM(\cF,\cG)=0$ modulo $\Lambda_\bZ$-torsion.

For the sake of completeness, let us indicate how some of the above constructions are carried out. For a sheaf ${\mathcal F}$ on $X\times { \R},$ let $\SSS(\cF)$ be its singular support as defined in \cite{KS}. Let $D(T^*X)$ be the left orthogonal complement to the subcategory of sheaves ${\mathcal G}$ such that $\SSS({\mathcal G})$ is contained in $\{\tau\leq 0\},$ where $\tau$ is the variable dual to the coordinate $t$ on $\R.$ The microsupport of an object $\cF$ is defined by $\muS(\cF)=\{(x,\xi)\in T^*X| (x,\xi, t, 1)\in \SSS(\cF) \operatorname{for}\;\operatorname{some} t\in \R\}.$

Tamarkin's current work \cite{T1} generalizes the construction of $D(T^*X)$ to any symplectic manifold $M$.
\subsubsection{Comparisons between the categories}\label{sss:comrapaisons}
As we can see, many properties of the category $D(T^*X)$ are parallel to those of categories such as $\cA_M^\bullet$-modules with an $A_\infty$ action of $\pi_1(M).$ These include (1) (the second half), (3), and (4). Property (5) is very likely to hold. Properties (2) and (6) need further study (see next remark).

The following idea probably allows to construct a functor from ($\cA_M,$ $\pi_1(M)$)-modules on $T^*X$ satisfying some conditions to sheaves on $X\times \bR.$ For such a module $\cV^\bullet,$ assume that $\uRHOM (\cV_0^\bullet, \cV^\bullet)$ is a {\em filtered} infinity local system as, for example, in Conjecture \ref{prop:filt on RHOM} if the latter is true. Denote the filtration by ${\rm{Filt}}_a,\,a\in \bR.$ Then the stalk at $(x,t)$ of the sheaf corresponding to $\cV^\bullet$ should be the ${\rm{Filt}}_t$ part of the complex that computes local cohomology of this infinity local system at $x.$ 
\begin{remark}\label{rmk:muS}
Our source of defining ($\cA_M,$ $\pi_1(M)$)-modules are {\em oscillatory modules}. (Their original version was defined in \cite{OM}). Oscillatory modules as defined here in \ref{ss:Oscillatory mods} are actually complexes of sheaves. It is possible to relax the definition somewhat and only require them to carry a differential $\nabla_\cV$ satisfying $\nabla_\cV^2=\frac{1}{i\hbar}\omega$ where $\omega$ is the symplectic form. (In other words, we can use the groupoid $\tG_M$ as defined in  \ref{sss:not quite flat con} and not in \ref{sss:quite flat con}). If we allow this, we seem to gain much more generality. For example, it will be much easier to construct an oscillatory module not only from a Lagrangian but from a coisotropic submanifold (as discussed in \cite{KW}) and maybe for more general submanifolds. On the other hand, it seems that the condition $\nabla_\cV^2=0$ may be what is needed to define the microlocal support $\mu S(\cV^\bullet)$ (the latter should be some version of the support of the differential $\nabla_\cV$). {\em Cf.}, for example, an explicit formula for $\nabla_\cV$ given by \eqref{eq:nabla V osc}.
\end{remark}
\begin{remark}\label{rmk:Witten} Much of the motivation behind our approach came from \cite{W}. We do not know any rigorous link between the two works. It would be very interesting to relate our methods to the study of asymptotics of eigenvalues of the Schr\"{o}dinger operator.
\end{remark}
\subsubsection{Acknowledgements} I am grateful to Dima Tamarkin for fruitful discussions and for many explanations of his works. As already indicated above, much of the present paper originated from earlier ideas of Boris Feigin and Sasha Karabegov.
\section{$\RHom$ and the twisted De Rham complex}\label{s:Rhom int points}
\subsection{Deformation quantization algebra}\label{ss:Deformation quantization of a formal neighborhood} 
Put
$${\mathbb A}=C^\infty (\bR^{2n})[[\hbar]]$$ 
with the Moyal-Weil product
$$
(f*g)(x,\xi)=\exp(\frac{i\hbar}{2}(\frac{\partial}{\partial \xi}\frac{\partial}{\partial y}-\frac{\partial}{\partial x}\frac{\partial}{\partial \eta})) (f(x,\xi)g(y,\eta))|_{x=y,\xi=\eta}
$$
For a function $f(x)$ denote
$$
\bV_f=\bA / \sum_j \bA (\xi_j-\frac{\partial f}{\partial x_j})
$$
or, in a simplified notation, 
$$\bV_f=\bA/\bA(\xi-f'(x))$$
\begin{lemma}\label{lemma:Vf as funs}
As a $\bC[[\hbar]]$-module, $\bV_f$ is isomorphic to $C^\infty (\bR^n)[[\hbar]]$ on which $x_j$ acts by multiplication and $\xi_j$ by  $i\hbar \frac{\partial}{\partial x_j}+\frac{\partial f}{\partial x_j}.$
\end{lemma}
\subsection{The complex computing $\RHom (\bV_0, \bV_f)$}\label{Rhom 0 f}
\begin{lemma}\label{lemma:koszul ext comput}
The complex $(\Omega^\bullet (\bR^n)[[\hbar]], i\hbar d_{\rm{DR}}+df\wedge)$ computes ${\operatorname {Ext}}^\bullet_{\bA}(\bV_0, \bV_f)$ 
\end{lemma}
\begin{proof} Fix a basis $e_1, \ldots, e_n$ of $\bC^n.$ Let $e^1,\ldots,e^n$ be the dual basis of $(\bC^n)^*.$ Let $\cR_k = \bA \otimes \wedge ^k ({\mathbb C}^n)$. If $e_1, \ldots, e_n$ is a basis of ${\mathbb C}^n,$ define the differential
\begin{equation}\label{eq: diff koszul} 
\partial (a\otimes e_{j_1}\wedge \ldots \wedge e_{j_k})=\sum_{p=1}^k (-1)^p a\xi_{j_p} \otimes e_{j_1}\wedge \ldots\wedge {\widehat {e_{j_p}}} \wedge \ldots \wedge e_{j_k}
\end{equation}
The complex $(\cR_\bullet, \partial)$ is a free resolution of the module $\bV_0.$ The complex $\Hom_\bA (\cR_\bullet, \bV_f))$ becomes 
\begin{equation}\label{eq:kos cplx} 
C^k =\wedge ^k ( \bC^n)^*\otimes \bV_f ;
\end{equation}
\begin{equation}\label{eq:kos diff 2}
d(e^{j_1} \wedge \ldots \wedge e^{j_k}\otimes v)=\sum _{p=1}^k e^{j_1} \wedge \ldots \wedge e^{j_k}\wedge e_p \otimes \xi_pv
\end{equation}
which is isomorphic to $(\Omega^\bullet (\bR^n)[[\hbar]], i\hbar d_{\rm{DR}}+df\wedge)$ because of Lemma \ref{lemma:Vf as funs}.
\end{proof}
\section{The Weyl algebra and the  Fedosov connection}\label{ss:Weyl and Fedosov} 
\subsection{The case of ${\mathbb R}^{2n}$}\label{ss:Fedosov R2n} Set
$${{{\A}}}={\mathbb C}[[\fx_1, \ldots, \fx_n, \fxi_1, \ldots, \fxi_n, \hbar]]$$
with the Moyal-Weyl product
$$
(f*g)(\fx,\fxi)=\exp(\frac{i\hbar}{2}(\frac{\partial}{\partial \fxi}\frac{\partial}{\partial \fy}-\frac{\partial}{\partial \fx}\frac{\partial}{\partial \feta})) (f(\fx,\fxi)g(\fy,\feta))|_{\fx=\fy,\fxi=\feta}
$$

Define the operator on $\A$-valued forms by 
\begin{equation}\label{eq:Fedosov flat}
\nabla_{\mathbb A}= (\frac{\partial}{\partial x}-\frac{\partial}{\partial \fx})dx+   (\frac{\partial}{\partial \xi}-\frac{\partial}{\partial \fxi})d\xi
\end{equation}
This is the Fedosov connection (in the partial case of a flat space). One has $\nabla_{\mathbb A}^2=0;$ the complex $(\Omega^\bullet (\bR^{2n}, \A), \nabla_{\mathbb A})$ is quasi-isomorphic to $C^\infty (\bR^{2n})[[\hbar]].$ The latter embeds quasi-isomorphically to the former by means of 
\begin{equation}\label{eq: aaassjs}
f\mapsto f(x+\fx,\xi+\fxi).
\end{equation}
\subsubsection{Infinitesimal symmetries of the deformation quantization algebra on a formal neighborhood}\label{ss:infin symms} Let
${{{\A}}}={\mathbb C}[[\fx_1, \ldots, \fx_n, \fxi_1, \ldots, \fxi_n, \hbar]]$
with the Moyal-Weyl product as in \ref{ss:Deformation quantization of a formal neighborhood}.
Put
$${\mathfrak g}=\Der_{\cont}(\A)=\frac{1}{i\hbar}\A/\frac{1}{i\hbar}{{\mathbb{C}}[[\hbar]]};\;\;{\widetilde{\mathfrak g}}=\frac{1}{i\hbar}\A$$
viewed as Lie algebras with the bracket $a*b-b*a.$ 

Introduce the grading
\begin{equation}\label{eq:grading}
|\fx_i|=|\fxi_i|=1;\; |\hbar|=2.
 \end{equation}
One has a central extension
\begin{equation}\label{eq:extension g}
0\to\oih\C[[\hbar]]\to\tgg\to\g\to 0,
\end{equation}
as well as
\begin{equation}\label{eq:grading g}
\g=\prod_{i=-1}^\infty \g_i;\;\tgg=\prod_{i=-2}^\infty \tgg_i.
\end{equation}
We will use the notation
\begin{equation}\label{eq:grading g 1}
\g_{\geq 0}=\prod_{i=0}^\infty \g_i;\;\tg_{\geq 0}=\prod_{i=0}^\infty \tgg_i.
\end{equation}
Note that
\begin{equation}\label{eq:grading g 2}
\g_{0}\isomoto {\mathfrak{sp}}(2n)
\end{equation}
and the action of this Lie algebra on $\A$ is the standard action of $ {\mathfrak{sp}}$ by infinitesimal linear coordinate changes.
\subsubsection{DG model for ${\RHom}(\bV_0,\bV_f)$}\label{sss:DG model for RHom} Though this is not needed for the sequel, let us explain how modules $\bV_f$ can be replaced by their DG analogs. Define 
\begin{equation}\label{eq:V f DG model}
\Omega^\bullet (\bR^{2n} , \wbV _f)=\Omega^\bullet (\bR^n)[[\fx, \hbar]]
\end{equation}
with the differential
\begin{equation}\label{eq: nabla V}
\nabla_{\bV}= (\frac{\partial}{\partial x}-\frac{\partial}{\partial \fx})dx
\end{equation}
and the action of $\Omega ^\bullet (\bR^{2n}, \A)$ defined as follows: $x$ and $\fx$ act by multiplication; $\xi$ acts by multiplication by $f'(x);$ $\fxi$ acts by $i\hbar \frac{\partial}{\partial \fx}+f'(x+\fx)-f'(x);$ $d\xi$ acts by $df'(x)=f''(x)dx.$ 

It is easy to see that $\Omega ^\bullet (\bR^{2n}, \A)$ is the space of global sections of a sheaf of differential graded algebras, and $\Omega^\bullet (\bR^{2n} , \wbV _f)$ is the space of global sections of a sheaf of differential graded modules supported on the Lagrangian submanifold $L_f = \{\xi=f'(x)\}.$ The formula $v\mapsto v(x+\fx)$ defines a quasi-isomorphic embedding 
$$\bV_f \to \Omega (\bR^n, \wbV_f)$$
compatible with the embedding of algebras $C^\infty(\bR^{2n})[[\hbar ]]\to  \Omega ^\bullet (\bR^{2n}, \A)$ defined in \eqref{eq: aaassjs}.

\begin{lemma}\label{lemma:RHom DG case} Let $e^*$, $\fe^*$ and $a^*$ be three free graded commutative variables of degrees $1$, $1,$ and $0$ respectively.
The cohomology 
$$\RHom _{\Omega ^\bullet (\bR^{2n}, \A)}  (\Omega^\bullet (\bR^{2n} , \wbV _0), \Omega^\bullet (\bR^{2n} , \wbV _f))$$
is computed by the complex  
$$\Omega^\bullet (\bR^{n} , \wbV _f)[e^*, \fe ^*][[a^*]], \nabla_{\bV}+e^*  \xi  + \fe^* \fxi  + a^* d\xi  +  (e^*-\fe^*)\frac{\partial}{\partial a^*} $$
which is isomorphic to $\Omega^\bullet (\bR^{n})[[\fx, \hbar]][e^*, \fe ^*][[a^*]]$ with the differential
$$(\frac{\partial}{\partial x}-\frac{\partial}{\partial \fx})dx +  e^* f'(x)   + a^* f''(x)dx  + \fe^* (i\hbar \frac{\partial}{\partial \fx} +f'(x+\fx)-f'(x))+(e^*-\fe^*) \frac{\partial}{\partial a^*}  $$
The latter complex is quasi-isomorphic to the one in Lemma \ref{lemma:koszul ext comput}.
\end{lemma}
\begin{proof} The DG module $\Omega^\bullet (\bR^n, \wbV_0)$ is the quotient of the free DG module $\Omega^\bullet (\bR^{2n}, \A)$ by the differential graded submodule generated by $\xi, $ $ d\xi,$ and $\fxi.$ A Koszul complex $\cP = \Omega^\bullet (\bR^{2n}, \A)[e,\fe,a]$ is a semi-free resolution of this quotient. The differential extends $\nabla_{\mathbb A},$ sends $e v$ to $\xi v+ av,$ $\fe v $ to $-\fxi v+av,$ $av$ to $d\xi\cdot v,$ and is a coderivation with respect to the action of $\bC[e,\fe,a].$ The complex $\Hom _{\Omega^\bullet (\bR^{2n}, \A)} (\cR, \Omega^\bullet (\bR^n, \wbV_f))$ is isomorphic to both complexes above. It remains to show that the latter of those complexes is quasi-isomorphic to ($\Omega^\bullet (\bR^{2n})[[\hbar]], i\hbar d_{\rm{DR}}+df\wedge)$. To this end, consider the second complex in the statement of the lemma. Change the odd variables to $e^*$ and $e^*-\fe^*;$ note that we can factor out all positive powers of $a^*$ and $e^*-\fe^*$. This is because the differential $(e^*-\fe^*) \frac{\partial}{\partial a^*}$ is acyclic. We are left with the complex
$\Omega^\bullet (\bR^{n})[[\fx, \hbar]][e^*]$ with differential
$$(\frac{\partial}{\partial x}-\frac{\partial}{\partial \fx})dx +  e^*  (i\hbar \frac{\partial}{\partial \fx} +f'(x+\fx))  $$
Now change the even variables. Put $y=x+\fx$ and keep $\fx$ as the second variable. As for the odd variables, put $Dx=dx-i\hbar e^*$ and keep $e^*$ as the second variable. The differential becomes
$$(i\hbar \frac{\partial}{\partial y}+f'(y))e^* - \frac{\partial}{\partial \fx} Dx.$$
We can factor out all positive powers of $\fx$ and of $Dx$ because the differential $\frac{\partial}{\partial \fx} Dx$ is acyclic.
\end{proof}
\subsection{Deformation quantization of symplectic manifolds}\label{sss:defquadef}
We recall from \cite{BFFLS} that a deformation quantization of a symplectic manifold $M$ is a formal product 
$$f*g=fg+\sum_{k=1}^\infty (i\hbar)^k P_k(f,g)$$
where $P: C^\infty(M)\times C^\infty(M)\to C^\infty(M)$ are bilinear bidifferential operators, $f*(g*h)=(f*g)*h$ in $C^\infty(M)[[\hbar]],$ $1*f=f*1=f,$ and 
$$P_1(f,g)-P_1(g,f)=\{f,g\}.$$ 
An isomorphism between two deformation quantizations is a formal series
$$T(f)=f+\sum_{k=1}^\infty (i\hbar)^k T_k(f)$$
where $T(f)*T(g)=T(f*'g)$ and $T_k: C^\infty(M)\to C^\infty(M)$ are linear differential operators. Below we review how to classify deformation quantizations up to isomorphism using Fedosov connections.
\subsection{The bundle $\fbA_M$}\label{ss:the bdle AM} By $\fbA_M$ we denote the bundle of algebras associated to the action of $\Sp(2n)$ on $\fbA$.
\subsection{The Fedosov connection}\label{ss:fedcon} \begin{definition}\label{def:Fedosov connection}
A Fedosov connection $\nabla$ is a connection in the bundle of algebras $\cA_M$ satisfying the following properties.
\begin{enumerate}
\item\label{Fedcon 1}
$$\nabla(fg)=\nabla(f)g+f\nabla(g)$$
for any local sections $f$ and $g$ of $\bA_M.$
\item\label{Fedcon 2}
$\nabla^2=0$
\item\label{Fedcon 3}
In any local Darboux coordinates $x,\xi$ on $M$ and any formal Darboux coordinates $\fx,\fxi$ of $\bA,$ 
$$\nabla=d_{\rm{DR}}-(\frac{\partial}{\partial \fx}dx-\frac{\partial}{\partial \fxi}d\xi)+A_{\geq 0}$$
where $A_{\geq 0}$ is a one-form with coefficients in ${\mathfrak g}_{\geq 0}$ (we use the notation of \eqref{eq:grading g 1}).
\end{enumerate} 
\end{definition}
Note that ${\mathfrak{sp}}(2n)$ embeds into $\tgg$ as the space of $\frac{1}{i\hbar}q(\fx,\fxi)$ where $q$ is a quadratic polynomial.
\begin{definition}\label{def:lifted Fedosov connection}
A lifted Fedosov connection $\tn$ is a collection of ${\widetilde{\mathfrak g}}$-valued one-forms $A_j$ on local Darboux charts $U_j$ such that 
\begin{enumerate}
\item\label{Fedcon 0} 
$$A_j=-dg_{jk}g_{jk}^{-1}+\Ad({g_{jk}})A_k$$
for any $j$ and $k.$
\item\label{Fedcon 21}
$\nabla^2$ is central.
\item\label{Fedcon 31}
In any local Darboux coordinates $x,\xi$ on $M$ and any formal Darboux coordinates $\fx,\fxi$ of $\bA,$ 
$$\nabla=d_{\rm{DR}}-\frac{1}{i\hbar}{\fxi} dx+\frac{1}{i\hbar}\fx d\xi+A_{\geq 0}$$
where $A_{\geq 0}$ is a one-form with coefficients in ${\widetilde{\mathfrak g}}_{\geq 0}$ (we use the notation of \eqref{eq:grading g 1}).
\end{enumerate} 
\end{definition}
Any lifted Fedosov connection $\tn$ defines a Fedosov connection $\nabla$ via the projection ${\widetilde{\mathfrak g}}\to {\mathfrak g}.$ In this case we call $\tn$ a lifting of $\nabla$.

Let
\begin{equation}\label{eq:G semi dir 0}
G=\Sp(2n,\R)\ltimes \exp (\g_{\geq 1})
\end{equation}
This group acts on $\bA$ by automorphisms. Let $G_M$ be the associated bundle of groups. It acts by automorphisms on the bundle of algebras $\bA_M.$ 
\begin{definition} \label{dfn:gauge equiv Fed}
Two Fedosov connections are gauge equivalent if they are conjugated by a section of $G_M.$
\end{definition}
\begin{thm}\label{thm:Fedosov classification}
1) For every 
$$\theta=\frac{1}{i\hbar}\omega +\sum _{j=0}^\infty (i\hbar)^j \theta_j$$
where $\theta_j$ are closed two-forms on $M$, there exists a lifted Fedosov connection $\tn$ such that $\tn^2=\theta.$

2) Any Fedosov connection has a lifting. Two Fedosov connections are gauge equivalent if and only if the curvatures of their liftings are cohomologous as $\frac{1}{i\hbar}\bC[[\hbar]]$-valued two-forms. In particular, any Fedosov connection is locally gauge equivalent to the standard one.

3) For any Fedosov connection, the kernel of $\nabla:\Omega^0_M(\bA_M)\to \Omega^1_M(\bA_M)$ is isomorphic to $C^\infty_M[[\hbar]]$ as a sheaf of algebras. Therefore any Fedosov connection defines a deformation quantization of $M$.

4) Any deformation quantization comes from some Fedosov connection. Two deformation quantizations are isomorphic if and only if the corresponding Fedosov connections are gauge equivalent.
\end{thm}
This is mostly contained in \cite{F}. The complete proof can be found in \cite{NT}. See also \cite{BGNT}.
\section{The extended Fedosov construction}\label{s:the ext Fed}
\subsection{The algebra $\cA$}\label{ss:The alg curvA}  First consider a larger completion of the Weyl algebra. Recall that the assignment
\begin{equation}\label{eq:grading Weyl alg}
|\fx_j|=|\fxi_j|=1;\;|\hbar|=2
\end{equation}
turns $\fbA$ into a complete graded algebra
\begin{equation}\label{eq:decomp of A}
\fbA=\prod_{k=0}^\infty \fbA_k 
\end{equation}
Let $\fbA [\hbar^{-1}]  _k$ be the space of elements of degree $k$ in $\fbA[\hbar^{-1}].$ 

Now define
\begin{equation}\label {eq:double compl A}
\ffbA=\{\sum _{k=-N} ^\infty a_k | a_k\in  \fbA [\hbar^{-1}]  _k \}
\end{equation}
where $N$ runs through all integers. The product is the usual Moyal-Weyl product.

Now let $\Mp(2n)$ be the group defined in \ref{ss:SpN} (in the case $N=4$). This group acts on $\ffbA$ through $\Sp(2n).$ Consider the cross product $\Mp(2n) \ltimes \ffbA$.
\begin{remark}\label{rmk:completions} Here and everywhere by cross products we will mean their completed versions. In other words, elements of the cross product are infinite sums $\sum g_k a_k$ where $g_k\in \Sp^4,$ $a_k\in \bA[\hbar^{-1}],$ and $|a_k|\to \infty.$
\end{remark}
\begin{definition}\label{dfn:curvA}
$$\cA = \{\sum _{k=0}^\infty a_k e^{\frac{1}{i\hbar}c_k} | a_k\in \Mp(2n) \ltimes \ffbA ;\, c_k\in \bR;\, c_k \to \infty\}$$
Let $\cA_\Lambda$ be defined exactly as above,  but with an extra condition $c_k\geq 0.$ We will sometimes write $\cA_{\bK}$ instead of $\cA.$
\end{definition}
Note that we view $\Mp(2n)$ as a {\em discrete} group.
\subsubsection{The Novikov ring}\label{sss:Novikov rg}
Define
\begin{equation}\label{eq:Lambda}
\Lambda = \{\sum _{k=0}^\infty a_k e^{\frac{1}{i\hbar}c_k} | a_k\in \bC;\, c_k\in \bR;c_k\geq 0;\,\, c_k \to \infty\}
\end{equation}
\begin{equation}\label{eq:Lambda 1}
\bK = \{\sum _{k=0}^\infty a_k e^{\frac{1}{i\hbar}c_k} | a_k\in \bC;\, c_k\in \bR;\, c_k \to \infty\}
\end{equation}
Clearly, $\cA$ is an algebra over $\bK.$
\subsection{The bundle $\cA_M$}\label{ss:The alg curvAM} Since the action of $\Sp(2n)$ extends from $\fbA$ to $\cA$, we get the associated bundle of algebras $\cA_M$ on any symplectic manifold $M.$
\subsection{The extended Fedosov connection}\label{ss:extfedcon} Note that the action of the Lie algebra $\tgg$ extends to an action of $\ffbA$ and therefore any Fedosov connection $\nabla_{\bA}$ extends canonically to a connection that we denote by $\nabla_{\cA}.$
\section{Action up to inner automorphisms}\label{s:action up to inner} 
\subsection{Groups acting up to inner automorphisms}\label{ss:groups to inn}
\begin{definition}\label{dfn:action grp up to inn}
Let $\Gamma$ be a group and $A$ an associative algebra. An action of $\Gamma$ on $A$ up to inner automorphisms is the following data.

1) Automorphisms $T_g:A\isomoto A$ for all $g\in \Gamma.$ 

2) Invertible elements $c(g_1,g_2)$ of $A$ for all $g_1,\,g_2$ in $\Gamma$ such that 
\begin{equation}\label{eq: kozykel c 0}
T_{g_1}T_{g_2}=\Ad(c(g_1,g_2))T_{g_1g_2}
\end{equation}
\begin{equation}\label{eq: kozykel c}
c(g_1,g_2)c(g_1g_2,g_3)=T_{g_1}c(g_2,g_3)c(g_1,g_2g_3)
\end{equation}
\end{definition}
An {\em equivalence} between $(T,c)$ and $(T',c')$ is a collection $\{b(g)\in A^\times |g\in G\}$ such that 
\begin{equation}\label{eq:equiv T c}
T'_g = \Ad (b(g)) T_g; \;c'(g_1,g_2)=b(g_1) T_{g_1}(b_{g_2}) c(g_1,g_2) b(g_1g_2)^{-1}
\end{equation}
It $\{b'(g)\}$ is an equivalence between $(T,c)$ and $(T',c')$ and $\{b''(g)\}$ is an equivalence between $(T',c')$ and $(T'',c''), $ then their {\em composition} is defined by $b(g)=b''(g)b'(g)$ and is an equivalence between $(T,c)$ and $(T'',c'').$
\subsection{Derivations of square zero up to inner derivations}\label{ss:ders squazer up to inn}
\begin{definition}\label{dfn:dersquaze up inn} Let $A$ be a graded algebra and let $\Gamma$ be a group acting on $A$ up to inner automorphisms. {\em A derivation of $A$ of square zero up to inner derivations  compatible with the action of $\Gamma$} is the following data.

1) A derivation $D$ of $A$ of degree one;

2) an element $R$ of $A$ of degree two;

3) elements $\alpha(g)$ of $A$ of  degree one for every element $g$ of $\Gamma$, such that
$$
D^2=\ad(R);\; DR=0; \;T_g D T_g^{-1}=D+\ad(\alpha(g));\; 
$$
$$
D\alpha(g)+\alpha(g)^2=T_g R-R;
$$
$$
\alpha(g_1)+T_{g_1} \alpha(g_2)-\Ad (c(g_1,g_2))\alpha(g_1g_2)+D c(g_1,g_2)\cdot c(g_1,g_2)^{-1}=0.
$$
\end{definition}
Now assume that we are given two sets of data: $(T,c)$ with a compatible $(D,\alpha, R),$ and $(T',c')$ with a compatible $(D',\alpha', R').$ An equivalence 
$$(T,c),(D,\alpha, R)\isomoto (T',c'),(D',\alpha', R')$$
between them is an equivalence $\{b(g)\}$ between the actions and an element $\beta$ of $\cA$ of degree one such that
\begin{equation}\label{eq:eq ders up to inn}
D'=D+\ad(\beta);
\end{equation}
\begin{equation}\label{eq:eq ders up to inn 1}
 \alpha'(g) = -Db(g)\cdot b(g)^{-1}+ \Ad_{b(g)}(\alpha(g)+T_g\beta);
 \end{equation}
 \begin{equation}\label{eq:eq ders up to inn 11}
 R'=R+D\beta+\beta^2
\end{equation}
For two equivalences
$$(b'(g), \beta'): (T,c), (D,\alpha, R)\isomoto (T',c'),(D',\alpha', R')$$ 
and 
$$(b''(g), \beta''): (T',c'), (D',\alpha', R')\isomoto (T'',c''),(D'',\alpha'', R''), $$ 
their {\em composition} is an equivalence 
$$(b(g), \beta): (T,c),(D,\alpha, R)\isomoto (T'',c''), (D'',\alpha'', R'')$$
 given by 
 \begin{equation}\label{eq:comp for D,R}
 b(g)=b''(g)b'(g);\, \beta=\beta''+\beta'.
 \end{equation}
\begin{remark}\label{rmk:curva jasnaja} A graded algebra with $D$ and $R$ as in 1) and 2) subject to the first two equations in 3) is called {\em a curved differential graded algebra}. In other words, this is an $A_\infty$ algebra with the only nonzero operations being $m_0,\,m_1,\,m_2.$  Furthermore, $(T_g, \alpha(g))$ are {\em curved morphisms}, {\em i.e.} $A_\infty$ morphisms with the only nonzero operations $T_0,\,T_1.$
\end{remark}
\subsubsection{Lie algebras acting up to inner derivations}\label{sss:LAs to inn}
The above is a partial case of the following definition (that is not used in the sequel).
\begin{definition}\label{def:act lie alg up to inn}
Consider an action of a group $\Gamma$ on an algebra $A$ given by the data $T_g, c(g_1,g_2).$ Let $\cL$ be a Lie algebra. An action of $\cL$ on $A$ up to inner derivations compatible with the action of $\Gamma$ is the following data.

1) A linear map $D: \cL\to \Der (A),$ $X\mapsto D_X;$

2) linear maps $\alpha: \cL\to A$ for any $g\in \Gamma,$ $X\mapsto \alpha_X(g).$

3) a bilinear skew symmetric map $R: \cL\times \cL\to A$, satisfying
$$
[D_X, D_Y]=D_{[X,Y]}+\ad R(X,Y);
$$
$$
D_X(R(Y,Z))+D_Y(R(Z,X))+D_Z(R(X,Y))=
$$
$$
=[D_X, D_{[Y,Z]}]+[D_Y, D_{[Z,X]}]+[D_Z, D_{[X,Y]}];
$$
$$
T_g D_X T_g^{-1}=D+\ad(\alpha_X(g));
$$
$$
D_X\alpha_Y(g)-D_Y\alpha_X(g) +[\alpha(X,g), \alpha(Y,g)]-\alpha_{[X,Y]}(g)=T_g R(X,Y)-R(X,Y);
$$
$$
\alpha_X(g_1)+T_{g_1} \alpha_X(g_2)-\Ad (c(g_1,g_2))\alpha_X(g_1g_2)+D_X c(g_1,g_2)\cdot c(g_1,g_2)^{-1}=0.
$$
\end{definition}
More generally, let $A$ be a graded algebra and $\cL$ is a graded Lie algebra. The above definition makes sense with the following changes: $c(g_1,g_2)$ are of degree zero; $R$ and $\alpha$ are homogeneous of degree zero; and signs are present in the formulas. Definition \ref{dfn:dersquaze up inn} describes a partial case when $\cL$ is a one-dimensional graded Lie algebra concentrated in degree one.
\subsection{Modules with compatible structures}\label{ss:compamods}
For an algebra $A$ with an action $(T_g,\, c(g_1, g_2))$ of a group $G$ up to inner automorphisms and for an $A$-module $V,$ a compatible action of $G$ on $V$ is a collection $\{T_g:V\to V|g\in G\}$ of module automorphisms such that $T_{g_1} T_{g_2} =  c(g_1,g_2) T_{g_1g_2}.$ 

Given a graded algebra $A$ and a graded module $V$ as above, consider a derivation $(D_A, \alpha, R)$ of square zero of $A$ up to inner derivations compatible with the action of $G.$ A compatible derivation of $V$ is a derivation $D_V: V^\bullet \to V^{\bullet+1}$ such that 
\begin{equation}\label{eq:comp mod der sq 0}
D_V^2=R;\; D_V(av)=D_A(a)v+(-1) ^{|a|} aD_V(v); \, T_g D_V T_g^{-1}=D_V+\alpha(g)
\end{equation} 
for all homogeneous $a$ in $A$ and $v$ in $V.$ 

A morphism $V\to W$  is by definition an $A$-module morphism commuting $F$ such $D_WF=FD_V.$

Given an equivalence 
$$(\{b(g)\}, \beta): (T,c), (D, \alpha, R)\isomoto (T',c'), (D', \alpha', R')$$ 
and an action and derivation on an $A$-module $V$ compatible with $(T,c)$, then 
\begin{equation}\label{eq:equiv vs compat mo}
T'_g=b(g)T_g;\; D'_V = D_V + \beta
\end{equation}
define on $V$ an action and a derivation compatible with $(T',c').$ This operation is compatible with compositions of equivalences.
\subsection{Quotient groups acting up to inner automorphisms}\label{ss:quots acting up to}
Assume given a surjection of groups $G\to \Gamma$ with kernel $H$. Assume that $A$  is an associative algebra together with a $G$-equivariant morphism of groups $i:H\to A^{\times}.$ Consider an action of $G$ on $A$ by automorphisms, $g\mapsto \bfT_g .$ This is of course a partial case of \ref{ss:groups to inn} with $c(g_1,g_2)=1.$ 

Choose a section of $G\to \Gamma$ sending $g\in \Gamma$ to $\bg\in G.$ Put
\begin{equation}\label{eq:quotactsup} 
T_g=\bfT_{\bg};\; c(g_1,g_2)=i(\bg_1\bg_2({\overline{g_1g_2}})^{-1})
\end{equation}

Furthermore, let $\bfD, {\beta}(g), {\mathbf R}$ be a derivation of square zero up to inner derivations compatible with the action of $G$.
 Assume that 
 $$\beta(h)=-\bfD(ih)(ih)^{-1}$$ 
 for all $h\in H.$
Put
\begin{equation}\label{eq:quotactsup 1} 
D=\bfD;\; \alpha(g)=\beta(\bg);\; R=\bfR
\end{equation}
\begin{lemma}\label{lemma:quotactsup 1}
1) Formulas \eqref{eq:quotactsup 1} define a derivation of square zero up to inner derivations compatible with the action of $\Gamma$ given by \eqref{eq:quotactsup}. Given two different sections $s_1:g\mapsto {\overline g}$ and $s_2:g\mapsto {\widetilde g},$ formulas
$$b(g)=i({\widetilde g} {\overline g}^{-1});\; \beta=0$$
define an equivalence $B(s_2,s_1)$ between corresponding derivations. One has
$$B(s_3,s_2)B(s_2,s_1)=B(s_3,s_1)$$

2) Assume $(V, \bfT_g, \bfD_V)$ is an $A$-module with a compatible action of $G$ and with a compatible derivation. Put $D_V=\bfD_V;$ $T_g=\bfT_{\overline g}.$ Then $(V, T_g, D_V)$ is an $A$-module with a compatible action of $\Gamma$ and a compatible derivation.
\end{lemma} 
The proof is straightforward.

There is also an analog of the above Lemma for Lie algebra actions as in \ref{sss:LAs to inn}.
\subsection{The case of groupoids}\label{ss:case of grpds} Now let $G$ be a groupoid with the set of objects $X.$ Let $A=\{A_x|x\in X\}$ be a family of algebras. {\em An action of $G$ on $A$ up to inner automorphisms} is the data consisting of operators $T_g: A_x \otomosi A_y$ for all $g\in G_{x,y}$ and of invertible elements $c(g_1,g_2)\in A_x$ for all $g_1\in G{x_1,x_2}$ and $g_2\in G_{x_2,x_3}$ such that \eqref{eq: kozykel c} is true. We give the same definition for a family $A$ of graded algebras where we require $c(g_1,g_2)$ to be of degree zero. 

If $A=\{A_x\}$ is a family of graded algebras with an action of $G$ up to inner derivations, {\em a derivation of square zero up to inner derivations compatible with the action} of $A$ is a family of derivations $\{D_x: A_x\to A_x|x\in X\}$ and of elements $\{\alpha(g)\in A_{x_1}|x_1,x_2\in X,\,g\in G_{x_1,x_2}\}$ such that 
$$
D_x^2=\ad(R_x);\; D_xR_x=0; \;T_g D_{x_2} T_g^{-1}=D_{x_1}+\ad(\alpha(g));\; 
$$
$$
D\alpha(g)+\alpha(g)^2=T_g R_{x_2}-R_{x_1};
$$
$$
\alpha(g_1)+T_{g_1} \alpha(g_2)-\Ad (c(g_1,g_2))\alpha(g_1g_2)+D_{x_1} c(g_1,g_2)\cdot c(g_1,g_2)^{-1}=0.
$$
A similar definition can be given for a family of (graded) Lie algebras $\{\cL_x|x\in X\}.$

Now consider a family of subgroups $\{H_x \in G_{x,x}|x\in X\}$, a groupoid $\Gamma$ with the same set of objects $X,$ and an epimorphism of groupoids $G\to \Gamma$ such that $H_x=\Ker(G_{x,x}\to \Gamma_{x,x}).$ Let $\{i_x:H_x\to A_x^{\times}\}$ be a $G$-equivariant family of morphisms of groups. Choose a section $g\mapsto \bg$ of $G\to \Gamma.$

\begin{lemma}\label{lemma:quotactupto grpoids}
1) Given an action $\{\bfT_g\}$ of $G$ on $A$ with $c(g_1,g_2)=1,$ formulas \eqref {eq:quotactsup}  define an action of $\Gamma$ on $A$ up to inner automorphisms. 

2) Given a derivation of square zero $(\bfD, \bfR, \beta)$ up to inner derivations compatible with the action of G, assume that $\beta(h)=-\bfD i(h)\cdot i(h)^{-1}$ for all $x$ and all $h\in H_x$ Then formulas \eqref{eq:quotactsup 1} define a derivation of square zero up to inner derivations compatible with the action of $\Gamma.$

3) For two different choices of sections $s_1,s_2,$ same formulas as in Lemma \ref{lemma:quotactsup 1}, 1), define an equivalence $B(s_2,s_1)$ between to derivations corresponding to two sections $s_1,s_2)$. One has 
$$B(s_3,s_2)B(s_2,s_1)=B(s_3,s_1).$$
\end{lemma}
\subsection{Modules with a compatible structure}\label{ss:compmods groupoids} For $A$ and $G$ as in \ref{ss:case of grpds}, an $A$-module $V$ with a compatible action of $G$ is a collection $\{V_x|x\in X\}$ of $A_x$ modules together with isomorphisms $\{T_g: V_x\otomosi V_y|x,y\in X;\, g\in G_{x,y}\}$ satisfying
$$T_g(av)=T_g(a)T_g(v);\; T_{g_1}T_{g_2}=c(g_1,g_2)T_{g_1g_2}$$
If $A$ and $V$ are graded and $(D_A,\alpha(g),R)$ is a compatible derivation of square zero up to inner derivations, a compatible derivation of $V$ is a linear map $D_V: V^\bullet \to V^{\bullet+1}$ such that 
$$D_V^2=R;\; D_V(av)=D_A(a)v+(-1)^{|a|} aD_V(v); T_g D_V T_g^{-1}=D_V+\alpha(g)$$
for all homogeneous $a\in A_x,\,v\in V_x.$ 

There are analogs of Lemma \ref{lemma:quotactsup 1} that we leave to the reader. 
\subsection{The case of Lie groupoids}\label{sss:Case of Lie grpds} 
\subsubsection{Lie groupoids: notation and conventions}\label{Lie groupoids: not} Recall that a groupoid with a set of morphisms $\cG$ and the set of objects $M$ is a {\em Lie groupoid} \cite{Mac} if $\cG$ and $M$ are manifolds and the source and target maps $s,t: \cG\to M$ are smooth surjective submersions, and the composition, inverse, and the map $M\to \cG,$ $x\mapsto \id_x,$ are smooth.

For two points $x_0$ and $x_1$ of $M,$ $\cG_{x_0,x_1}=\{g\in \cG| t(g)=x_0, \, s(g)=x_1\}.$ This way, the composition is a map $\cG_{x_0,x_1}\times \cG_{x_1,x_2}\to \cG_{x_0,x_2}.$ 
If 
$$\cG\times_M \cG=\{(g,g')\in \cG\times \cG|s(g)=t(g')\},$$
then the multiplication can be described as a map
$$m: \cG\times_M \cG \to \cG.$$

We denote by $\ucG$ the sheaf of (pro)manifolds on $M\times M$ defined by $\ucG(W)=(s,t)^{-1} (W),$ $W\subset M\times M.$ 

More generally, we have the map 
$${\proj}_n: \cG\times_M\ldots \times_M \cG\to M\times \ldots \times M$$
where the product is $n$-fold on the left and $(n+1)$-fold on the right. In particular, ${\proj}_1=(s,t).$
Put
\begin{equation}\label{eq:under G n} 
\ucG^{(n)}(W) =\proj_n^{-1}(W)
\end{equation}
This is a sheaf of pro-manifolds on $M^{n+1}.$

By $\cO_M$ we denote a sheaf of (graded) algebras on $M$ that could be $C^\infty_M,$ $\Omega^\bullet _M,$ or the sheaf of $\Lambda$-valued forms or functions that we will consider later. All that we need is that $\cO_M$ be defined for every manifold $M$ (of given type) and that for every morphism $f:M\to N$ the inverse image $f^*\cO_N$ be defined, together with the morphisms $f^{-1}\cO_M\to f^*\cO_M$ and $f^*\cO_N\to\cO_M$ subject to the usual identities.

By $p_j:M^{n+1}\to M$ we denote the projection onto the $j$th factor. Let $\cA$ be a sheaf of $\cO_M$-algebras. 
\begin{definition} \label{dfn:action up to Lie groupoid case}
An action of $\cG$ on $\cA$ up to inner derivations is a morphism of sheaves on $M\times M$
$$\ucG \times p_2^*\cA\to p_1^*\cA;\; (g,a)\mapsto T_g a$$
and a morphism of sheaves on $M\times M\times M$
$$c: \ucG^{(2)} \to p_1 ^*\cA$$
subject to 
$$T_{g_1}T_{g_2}(a)=\Ad c(g_1,g_2) T_{g_1g_2}(a)$$
in $p_1^*\cA$, 
for any local section $a$ of $p_3^*\cA$ and any two local sections $g_2$ of $p_{23}^*\ucG$ and $g_1$ of $p_{12}^*\ucG.$
\end{definition}
\begin{remark}\label{rmk:composing local sections} Given two local sections $g_1,g_2$ as above, by their composition we mean the following. If $g_1=g_1(x_1,x_2,x_3)\in \cG_{x_1,x_2}$ and $g_2=g_2(x_1,x_2,x_3)\in \cG_{x_2,x_3},$ then $(g_1g_2)(x_1,x_2,x_3)=g_1(x_1,x_2,x_3)g_2(x_1,x_2,x_3)$ in $G_{x_1,x_3}.$ Similarly for $c(g_1,g_2).$
\end{remark}
\subsubsection{Flat connections up to inner derivations}\label{sss:Conns up to inn}
Here we assume that $\cO_M,$ or $\cO_M^\bullet,$ is a differential graded algebra with a differential $d$. {\em A connection} on a sheaf of graded $\cO^\bullet_M$-modules $\cE$ is a morphism of sheaves $\nabla:\cE\to \cE$ of degree one such that $\nabla(ae)=da\cdot e+(-1)^{|a|}a\nabla e.$

We also assume that for every $f:M\to N$ and every sheaf of graded $\cO^\bullet_N$-modules $\cE$, a natural connection $f^*\nabla$ on $f^*\cE$ is defined, subject to the usual properties. For us $\cO_M^\bullet $ will be the sheaf of $\Lambda$-valued forms, and $f^*\nabla$ will be a straightforward analog of the standard inverse image of a connection that we will define in \ref{ss:Inverse Images}.

\begin{definition}\label{dfn:action up to Lie groupoid case 1} Let $\cA^\bullet$ is a sheaf of graded $\cO_M^\bullet$-algebras with an action of $\cG$ up to inner automorphisms. A flat connection up to inner derivations compatible with the action of $\cG$ is the following data. 

1) A connection $\nabla:\cA^\bullet\to \cA^{\bullet +1}$ which is a derivation.

2) A section $R$ of $\cA^2$.
 
3) A morphism of sheaves $\alpha: \ucG  \to p_1^*\cA^\bullet$ of degree one, such that:

$$
\nabla^2=\ad(R);\; \nabla R=0; \;T_g (p_2^*\nabla) T_g^{-1}=p_1^*\nabla+\ad(\alpha(g));\; 
$$
$$
(p_1^* \nabla)\alpha(g)+\alpha(g)^2=T_g (p_2^*R)-p_1^*R;
$$
$$
 \alpha(g_1)+T_{g_1} \alpha(g_2)-\Ad (c(g_1,g_2))\alpha(g_1g_2)+(p_1^* \nabla) c(g_1,g_2)\cdot c(g_1,g_2)^{-1}=0.
$$
We will often write $\alpha(g)=-\nabla g\cdot g^{-1}.$
\end{definition}
\subsection{Modules with a compatible structure: the Lie groupoid case}\label{ss:MCSLGrpd} 
In the situation of Definition \ref{dfn:action up to Lie groupoid case 1}, let $(\cV^\bullet, \nabla_\cV)$ be a differential graded $\cA^\bullet$-module together with a morphism of sheaves 
$M\times M$
$$\ucG \times p_2^*\cV\to p_1^*\cV;\; (g,v)\mapsto T_g v$$
subject to:
$$T_{g_1}T_{g_2}(v)=c(g_1,g_2) T_{g_1g_2}(v)$$
in $p_1^*\cV^\bullet$, 
for any local section $v$ of $p_3^*\cV$ and any two local sections $g_2$ of $p_{23}^*\ucG$ and $g_1$ of $p_{12}^*\ucG;$
$$T_g(av)=T_g(a)T_g(v)$$
in $p_1^*\cV,$ for any local sections $a$ of $p_2^*\cA^\bullet$ and $v$ of $p_2^*\cV^\bullet;$
$$\nabla_\cV^2=R; \nabla_\cV(av)=\nabla_\cA (a)v+(-1)^{|a|} a\nabla_\cV (v)$$
for any homogeneous local sections $a$ of $\cA^\bullet$ and $v$ of $\cV^\bullet;$
$$T_g (p_2^*D_V) T_g^{-1}=\pi_1^* D_V+\alpha(g)$$
\subsubsection{The action of the quotient in the Lie groupoid case}\label{sss:The action of the quotient in the Lie groupoid case} 

Now consider two Lie groupoids $\cG$ and $\Gamma$ with the same manifold of objects $M$ and an epimorphism of groupoids $\cG\to \Gamma$ (over $M.$) Define $\cH_x=\Ker(\cG_{x,x}\to \Gamma_{x,x})$ and $\cH=\cucup _{x\in M} \cH_x.$  Consider the morphism $: \cH\to M.$ Define the sheaf of groups $\ucH(U)=s^{-1}(U)$ for $U\subset M.$ Let $i:\ucH \to \cA^{\times}$ be a $\cG$-equivariant morphism of sheaves of groups. Choose a section $g\mapsto \bg$ of $\cG\to \Gamma.$

\begin{lemma}\label{lemma:quotactupto grpoids Lie}
1) Given an action $\{\bfT_g\}$ of $\cG$ on $\cA$ with $c(g_1,g_2)=1,$ formulas \eqref {eq:quotactsup}  define an action of $\Gamma$ on $\cA$ up to inner automorphisms. 

2) Given a flat connection $(\bfD, \bfR, \beta)$ up to inner derivations compatible with the action of $\cG$ , assume that $\beta(h)=-\bfD i(h)\cdot i(h)^{-1}$ for all local sections of $\cH.$ Then formulas
$$
\nabla=\bfD;\; \alpha(g)=\beta(\bg);\; R=\bfR
$$
define a flat connection up to inner derivations compatible with the action of $\Gamma.$

3) For two different choices of sections $s_1,s_2,$ same formulas as in Lemma \ref{lemma:quotactsup 1}, 1), define an equivalence $B(s_2,s_1)$ between to derivations corresponding to two sections $s_1,s_2)$. One has 
$$B(s_3,s_2)B(s_2,s_1)=B(s_3,s_1).$$

4) Let $\cV$ be a graded $\cA$-module with a compatible action $\bfT$ of $\cG$ and a compatible connection $\bfD_\cV.$  Then formulas
$$T_g = \bfT_{\overline g};\; \nabla_\cV=\bfD_\cV$$
define a compatible action of $\Gamma$ and a compatible connection on $\cV.$
\end{lemma}
\begin{remark}\label{rmk: discontinuity}
Note that the morphisms of sheaves $c: {\underline{\Gamma}}^{(2)}\to p_1^*\cA$ and $\alpha: {\underline \Gamma}\to p_1^*\cA$ are discontinuous. For us $\Gamma$ will be an \'{e}tale groupoid, more precisely the fundamental groupoid of $M.$ We can only make a choice of a continuous $c$ and $\alpha$ on any small coordinate chart, but that will be enough for our purposes. More precisely, this will define to a {\em twisted} $A_\infty$ action as it is explained in \ref{appendix:infty and twisted}.
\end{remark}

\section{From actions up to inner automorphisms to $A_\infty$ actions}\label{s:From up to in to A infty} It is a well-known fact that inner isomorphisms act on the $\Ext$ functors trivially. Therefore, if a group acts on an algebra up to inner automorphisms, given compatible actions on two $A$-modules $V$ and $W,$ the group acts on the cohomology ${\Ext}_A^\bullet (V,W).$ In this section we prove a more precise version of this fact, namely we construct an $A_\infty$ action of the group on the standard bar complex.
\subsection{$A_\infty$ actions}\label{ss:A infty acciones}
An $A_\infty$ action of a group $G$ on a complex $C^\bullet$ is a collection $\{T(g_1,\ldots,g_n)\in \Hom^{1-n}(C^\bullet,C^\bullet)| g_1,\ldots, g_n\in G, n>0\}$
satisfying
\begin{equation}\label{eq:A infty action}
[d,T(g_1,\ldots,g_n)]+\sum_{j=1}^{n-1} (-1)^jT(g_1,\ldots,g_j)T(g_{j+1},\ldots,g_n)-
\end{equation}
$$
\sum_{j=1}^{n-1} (-1)^j T(g_1,\ldots, g_jg_{j+1},\ldots,g_n)=0
$$
We sometimes write $T_g$ instead of $T(g).$ The operators $T(g)$ induce an action of $G$ on the cohomology of $C^\bullet$.

An $A_\infty$ morphism between two $A_\infty$ actions $T$ and $T'$ is a collection $\{\phi(g_1,\ldots,g_n)\in \Hom^{-n}(C^\bullet,C^\bullet)| g_1,\ldots, g_n\in G, n\geq 0\}$
satisfying
\begin{equation}\label{eq:A infty action mor}
[d,\phi(g_1,\ldots,g_n)]+\sum_{j=1}^{n-1} (-1)^jT'(g_1,\ldots,g_j)\phi(g_{j+1},\ldots,g_n)-
\end{equation}
$$-\sum_{j=1}^{n-1} (-1)^j\phi(g_1,\ldots,g_j)T(g_{j+1},\ldots,g_n)
-\sum_{j=1}^{n-1} (-1)^j \phi(g_1,\ldots, g_jg_{j+1},\ldots,g_n)=0
$$
\subsection{The $\Ext$ functors}\label{ss:Exty} Let $A$ be an associative algebra and $V,\,W$ two $A$-modules. By $C^\bullet (V,A,W),$ or simply $C^\bullet (V,W),$ we denote the standard complex computing ${\Ext^\bullet }_A(V,W)$. Namely,
$$C^m(V,W)=\prod_{p+n=m} \Hom(A^{\otimes n},\Hom^p (V,W));$$
the differential $\delta$ is defined by
$$(\delta\varphi )(a_1,\ldots, a_{n+1})=(-1)^{|\varphi||a_1|}a_1\varphi(a_2,\ldots,a_{n+1})+$$
$$\sum_{j=1}^n (-1)^{\sum_{i=1}^j (|a_i|+1)} \varphi(a_1,\ldots,a_ja_{j+1}, \ldots,a_{n+1})+$$
$$ (-1)^{\sum_{i=1}^{n+1}(|a_i|+1)} \varphi(a_1,\ldots,a_n)a_{n+1}$$ 
\begin{lemma}\label{dfn:higher on Ext}

1) Let $T$ be an automorphism of $A$ together with compatible automorphisms of $V$ and $W$ ({\em i.e.} invertible operators $T$ such that $T(av)=T(a)T(v)$).
Put 
$$ (T{\varphi})(a_1,\ldots,a_n)=T{\varphi}(T^{-1}a_1,\ldots,T^{-1} a_n)T^{-1}$$

Then $\varphi\mapsto T\varphi$ is an automorphism of $C^\bullet(V,W).$

2) For an invertible element $c$ of $A$ of degree zero define
$$(\phi (c)\varphi)(a_1,\ldots,a_n)=$$
$$=-\sum _{j=0}^n (-1)^{\sum_{i=1}^j (|a_i|+1)} \varphi (a_1,\ldots,a_j, c, c^{-1}a_{j+1}c, \ldots, c^{-1}a_nc)c^{-1}$$
One has
$$[\delta, \phi(c)]={\Ad(c)}-\id$$

3) More generally, for $m$ invertible elements $c_1,\ldots,c_m$ of degree zero of $A$, define
$$(\phi (c_1,\ldots,c_m)\varphi)(a_1,\ldots,a_n)=$$
$$=-\sum _{0\leq j_1\leq\ldots j_m\leq n} (-1)^{\sum_{k=1}^m\sum_{i=1}^{j_k}(|a_i|+1)} \varphi (a_1,\ldots,a_{j_1}, c_1, c_1^{-1}a_{j_1+1}c_1, \ldots, c_1^{-1}a_{j_2}c_1, $$
$$c_2, (c_1c_2)^{-1}a_{j_2+1}(c_1c_2), \ldots, (c_1c_2)^{-1}a_{j_3}(c_1c_2), \ldots,$$
$$c_m, (c_1\ldots c_m)^{-1}a_{j_m+1}(c_1\ldots c_m), \ldots, (c_1\ldots c_m)^{-1}a_{n}(c_1\ldots c_m))(c_1\ldots c_m)^{-1}$$
One has
$$[d,\phi(c_1,\ldots,c_m)]+\Ad_{c_1}\phi(c_2,\ldots,c_m)+$$
$$+\sum_{j=1}^{m-1}(-1)^j \phi(c_1,\ldots, c_jc_{j+1},\ldots, c_m)+(-1)^{m} \phi(c_1,\ldots,c_{m-1})=0$$
\end{lemma}
In other words: the group of automorphisms of $(A,V,W)$ acts on $C^\bullet (V,A,W);$ the subgroup of inner automorphisms acts homotopically trivially, in the sense that there is an $A_\infty$ morphism, starting with the identity, between this action and the trivial action. Note that, as in 1) above, we denote by $\Ad_{c}$ both the inner automorphism of $A$ and the induced automorphism of $C^\bullet (V,A,W).$
\begin{lemma}\label{lemma:tasovki}
$$
\phi(c_1,\ldots,c_m)\phi(d_1,\ldots,d_n)=\sum \pm \phi(e_1,\ldots,e_{n+m})
$$
where the summation is over all $(e_1,\ldots,e_{n+m})$ such that:

a) as a set, $\{e_1,\ldots,e_{n+m}\}=\{d_1, \ldots, d_m, x_1c_1x_1^{-1}, \ldots, x_nc_nx_n^{-1}\}$, with $x_j$ defined below in c);

b) the order of elements $d_j$ is preserved; the order of the elements $x_jc_jx_j^{-1}$ is the same as the order of the elements $c_j;$

c) $x_j$ is the product of all $d_k$ to the left of $x_jc_jx_j^{-1}.$
The sign is $(-1)^N$ where $N$ is the number of instances when $d_k$ is to the left of $x_jc_jx_j^{-1}.$
\end{lemma}
For example, 
$$\phi(c)\phi(d)=\phi(c,d)-\phi(d,\,dcd^{-1})$$
\subsubsection{A lemma about $A_\infty$ actions}\label{sss:Lemma A infty actions}
\begin{lemma}\label{lemma:Lemma A infty actions}
Let $\tiG$ be a group and $H$ its normal subgroup. Let $G=\tiG/H.$ Consider a complex $C^\bullet$ with the following data:

1) An action of $\tiG$, $g\mapsto \cT_g$ for any $g\in \tiG.$

2) Operators $\Phi(c_1,\ldots, c_m): C^\bullet \to C^{\bullet-m},\,m\geq 0,$ for all $c_1,\ldots,c_m\in H,$ satisfying
$$[d,\Phi(c_1,\ldots,c_m)]+\cT_{c_1}\Phi(c_2,\ldots,c_m)+$$
$$+\sum_{j=1}^{m-1}(-1)^j \Phi(c_1,\ldots, c_jc_{j+1},\ldots, c_m)+(-1)^{m} \Phi(c_1,\ldots,c_{m-1})=0$$
$$
\Phi(c_1,\ldots,c_m)\Phi(d_1,\ldots,d_n)=\sum \pm \Phi(e_1,\ldots,e_{n+m})
$$
as in Lemma \ref{lemma:tasovki}.
For any section $g\mapsto \bg$ of the projection $\tiG\to G,$ there is an $A_\infty$ action of $G$ on $C^\bullet$ such that $T_g=\cT_{\bg}.$ 
\end{lemma}
\begin{proof} Consider the differential graded algebra $\cB(H,\tiG)$ generated by the group algebra of $\tiG$ and by elements $\Phi(c_1,\ldots,c_m)$ of degree $-m$ for all $c_1,\ldots,c_m$ in $H$, such that:

a) $$g\Phi(c_1,\ldots,c_m)g^{-1}=\Phi(gc_1g^{-1}, \ldots ,gc_mg^{-1})$$

for any $g\in \tiG;$

b) the differential $\partial$ satisfies 
$$\partial \Phi(c_1,\ldots,c_m)+{c_1}\Phi(c_2,\ldots,c_m)+$$
$$+\sum_{j=1}^{m-1}(-1)^j \Phi(c_1,\ldots, c_jc_{j+1},\ldots, c_m)+(-1)^{m} \Phi(c_1,\ldots,c_{m-1})=0$$

c) 
$$
\Phi(c_1,\ldots,c_m)\Phi(d_1,\ldots,d_n)=\sum \pm \Phi(e_1,\ldots,e_{n+m})
$$
as in Lemma \ref{lemma:tasovki}. This differential graded algebra is quasi-isomorphic to $k[G].$ In fact, as a complex it is the standard bar construction of $H$ with coefficients in the right module $k[\tiG].$ The quasi-isomorphism is the morphism of algebras such that 
\begin{equation}\label{eq:quism Bc}
\Phi(c_1,\ldots,c_m)\mapsto 0; g\mapsto {\rm{proj}}_G (g), g\in \tiG.
\end{equation}
There is (unique up to homotopy) morphism from the standard resolution ${\operatorname{Cobar}}{\operatorname{Bar}}(k[G])$ to $\cB(H,\tiG)$ over $k[G].$ Now define $T(g_1,\ldots,g_n)$ to be the action of the image of the generator $(g_1|\ldots|g_n)$ on $C^\bullet.$
\end{proof}
\subsubsection{The $A_{\infty}$ action on the standard complex} \label{sss:The A inf action on standard} Now assume that a group $G$ acts on an algebra $A$ up to inner automorphisms. Assume that $V $ and $W$ are two $A$-modules with compatible actions. This means that there are linear automorphisms $T_g$ of $V$ and $W$ for any $g\in G$ such that 
\begin{equation}\label{eq:compat Tg V}
T_g(av)=T_g(a)T_g(v);\;T_{g_1}T_{g_2}=c(g_1,g_2)T_{g_1g_2}
\end{equation}
($c(g_1,g_2)$ in the right hand side denotes the module action of the element of $A$). 
\begin{thm}\label{thm:A infty action on Ext}
There is an $A_\infty$ action of $G$ on $C^\bullet (V,A,W)$ such that $T(g)$ is equal to $T_g$ as in \ref{dfn:higher on Ext}.
\end{thm} 
\begin{proof} Let $\tiG=G\ltimes _c A^\times$ be the group whose elements are expressions $ag$, $g\in G$ and $A^\times,$ with the product
\begin{equation}\label{eq:G tilda c grp}
(a_1g_1)(a_2g_2)=a_1T_{g_1}(a_2) c(g_1,g_2) (g_1g_2)  
\end{equation}
and $H=A^\times.$ The theorem follows immediately from Lemmas \ref{dfn:higher on Ext}, \ref{lemma:tasovki}, and \ref{lemma:Lemma A infty actions}.
\end{proof}
\begin{remark}\label{rmk:all not bad c} The proof of Theorem \ref{thm:A infty action on Ext} actually leads to a rather simple recursive formula for the $A_\infty$ action. Namely, the construction of a morphism 
\begin{equation}\label{eq:Cobarbar to B}
{\operatorname{Cobar}}{\operatorname{Bar}}(k[G]) \to \cB(A^\times, G\ltimes_c A^\times)
\end{equation}
(see the proof of Lemma \ref{lemma:Lemma A infty actions}) is an inductive procedure in $n$ for finding the image of $(g_1|\ldots|g_n)$ under this morphism. 
Let us describe this procedure. Consider the subalgebra $\cB(A^\times, A^\times)$ of expressions $c_0 \Phi(c_1,\ldots, c_m).$ This subalgebra is quasi-isomorphic to $k$, the homotopy being 
\begin{equation}\label{eq:homotopy on B A A}
s(c_0 \Phi(c_1,\ldots, c_m))= \Phi(c_0, c_1,\ldots, c_m)
\end{equation}
Now define $\Psi(g_1,\ldots,g_n)$ in $\cB(A^\times, A^\times)$ recursively by
\begin{equation}\label{eq:Psi recur}
\Psi_1(g)=g;
\end{equation}
\begin{equation}\label{eq:Psi recur 1}
\Psi(g_1,\ldots, g_{n+1})=
\end{equation}
$$s\sum _{j=1}^n(-1)^j  \Psi(g_1,\ldots, g_j) T_{g_1\ldots g_j}\Psi(g_{j+1}, \ldots, g_{n+1})c(g_1\ldots g_j, g_{j+1}\ldots g_{n+1})$$
Here the product is described in Lemma \ref{lemma:tasovki}, and 
$$T_g(c_0 \Phi(c_1,\ldots, c_m))=(T_gc_0 \Phi(T_gc_1,\ldots, T_gc_m))$$
 
The elements  $\Psi(g_1,\ldots,g_n)$ is some linear combinations of $\phi(c_1,\ldots, c_k)$ and $c_j$ are algebraic expressions in $T_{h_0}c(h_1,h_2),$ $h_i$ being some products of $g_i.$ 
 
 Let $\psi(g_1,\ldots, g_n)$ be the image of $\Psi(g_1,\ldots, g_n)$ under the morphism of algebras 
 \begin{equation}\label{eq:Bar to Endom}
 \cB(A^\times, A^\times)\to \End(C^\bullet)
 \end{equation}
 sending $g\in G$ to $T_g,$ $c\in A^\times $ to ${\Ad(c)},$ and $\Phi(g_1,\ldots, g_n)$ to $\phi(g_1,\ldots, g_n)$. Then
 $$T(g_1, \ldots, g_n)=\psi(g_1,\ldots, g_n)T_{g_1\ldots g_n}$$
 For example,
 $$T(g_1,g_2)=\phi(c(g_1,g_2)) T_{g_1g_2}$$
\end{remark}
\subsubsection{The case of groupoids}\label{sss:grpds ainfty action ext}
Let $G$ be a groupoid with the set of objects $X$ that acts on a family of algebras $A=\{A_x|x\in X\}$ up to inner automorphisms. Let $V=\{V_x|x\in X\}$ and $W=\{W_x|x\in X\}$ two $A$-modules with compatible actions of $G$, {\em i.e.} with families $T_g: V_x\otomosi V_y$ and $T_g: W_x\otomosi W_y$, satisfying \eqref{eq:compat Tg V}.

Given a family of complexes $C^\bullet=\{C^\bullet _x|x\in X\},$ an $A_\infty$ action of $G$ on $C^\bullet$ is a collection of
$$T(g_1, \ldots, g_n): C^{\bullet+1-n}_{x_{n+1}}\otomosi C^{\bullet}_{x_1}$$
for any $g_j\in G_{x_j, x_{j+1}},$ $j=1,\ldots, n$, satisfying the identities in the beginning of \ref{ss:A infty acciones}. Morphisms between $A_\infty$ actions are defined similarly.

Define 
$$C^\bullet (V,A,W)_x=C^\bullet (V_x,A_x,W_x)$$
\begin{thm}\label{thm:A infty action on Ext 1}
There is an $A_\infty$ action of $G$ on $C^\bullet (V,A,W)$ such that $T(g)$ is equal to $T_g$ as in \ref{dfn:higher on Ext}.
\end{thm}
The proof is identical to the proof of Theorem \ref{thm:A infty action on Ext}.
\subsubsection{$A_\infty$ action on the standard complex and derivations}\label{sss:action a inf der case} Let $A$ be a graded algebra with an action of $G$ up to inner automorphisms. Let $D$ be a compatible derivation of square zero up to inner derivations. If $V$ and $W$ are two graded $A$-modules with compatible actions of $G$, we assume that they carry a compatible derivation, {\em i.e.} an operator $D:V\to V$ or $W\to W$ of degree one satisfying
\begin{equation}\label{eq:comp der modu}
D(av)=D(a)v+(-1)^{|a|}aD(v);\; D^2=R;\; T_g D T_g^{-1}=D+\alpha(g)
\end{equation}
Here $R$ and $\alpha(g)$ stand for the action of corresponding elements of $A$.
For any homogeneous derivation $E $ of $A$ that acts on $V$ and $W$ compatibly, put
\begin{equation}\label{eq:diffl D}
(E\varphi )(a_1,\ldots, a_n)=[E, \varphi (a_1,\ldots, a_n)]-
\end{equation}
$$-\sum_{j=1}^n (-1)^{\sum _{i=1}^{j-1} |E|(|a_i|+1)} \varphi (a_1,\ldots,E a_j, \ldots, a_n)$$
Put for any homogeneous element $a$ of $A$ put
\begin{equation}\label{eq:diffl D 1}
(\iota_a\varphi )(a_1,\ldots, a_n)=\sum_{j=0}^n (-1)^{\sum_{i=1}^j (|a|+1)(|a_i|+1)} \varphi (a_1,\ldots, a_j, a, a_{j+1}, \ldots, a_n)
\end{equation}
\begin{lemma}\label{lemma:rels for jotas}
$$[\delta, E]=0;\; [\delta, \iota _a]=\ad(a); \; [E, \iota_a]=(-1)^{|E|} \iota_{Da};\;[\iota_a,\iota_b]=0.$$
\end{lemma}
\begin{corollary}\label{cor:full diffl on stancomp}
$$(\delta+D-\iota_R)^2=0$$
on $C^\bullet(V,A,W).$
\end{corollary} 
\begin{remark}\label{rmk:full diff on stcompl}
We will always consider $C^\bullet (V,A,W) $ as the standard complex equipped with the total differential $\delta+D-\iota_R.$

We now define an $A_\infty$ action on this standard complex. We follow the proof of Theorem \ref{thm:A infty action on Ext}. The only change is a different choice of operators $T_g$ and $\phi(c_1,\ldots, c_n)$ (see Lemma \ref{dfn:higher on Ext}, 3)).
\end{remark}

\begin{equation}\label{eq:curvaT}
\cT_g=\exp(\iota _{\alpha(g)}) T_g
\end{equation}
for every $g\in G;$
\begin{equation}\label{eq:curvaT1}
{\widetilde{\Ad}}(c)=\exp(-\iota_{Dc\cdot c^{-1}}) \Ad(c)
\end{equation}
for every $c\in A^\times$ of degree zero.
\begin{lemma}\label{lemma:T curved g 1}
a) $[\delta, \iota_a]=\ad_a;\; [E, \iota_a]=(-1)^{|E|}\iota_{Da};\; [\iota_{a_1}, \iota_{a_2}]=0;$

b) $[\delta+D-\iota_R, \cT_g]=0;\; $

c) ${\widetilde{\Ad}}(c_1){\widetilde{\Ad}}(c_2)={\widetilde{\Ad}}(c_1c_2);$

d) $\cT_g \tAd_c \cT_g^{-1}=\tAd _{T_gc} ;$$
$$cT_{g_1}\cT_{g_2}={\widetilde{\Ad}}(c(g_1,g_2)) \cT_{g_1}\cT_{g_2}$
\end{lemma} 
\begin{proof} a) is straightforward. Let us prove b).
$$\cT_g (\delta + D -\iota_R) \cT_g^{-1}=e^{\iota_{\alpha(g)}}T_g (\delta + D -\iota_R) T_g^{-1} e^{-\iota_{\alpha(g)}}=$$
$$=e^{\iota_{\alpha(g)}} (\delta+D +\ad_{\alpha(g)} -\iota_{R+D\alpha(g)+\alpha(g)^2}) e^{-\iota_{\alpha(g)}}$$
(we used the equations in Definition \ref{dfn:dersquaze up inn}). Now observe that 
$$e^{\iota_{\alpha(g)}} D e^{-\iota_{\alpha(g)}}=D+\iota_{\alpha(g)}; $$
$$e^{\iota_{\alpha(g)}} \delta e^{-\iota_{\alpha(g)}}=\delta-\ad_{\alpha(g)}+\iota_{\alpha(g)^2}$$
which implies b).

Now prove c).
$$\tAd_{c_1}\tAd_{c_2}=\exp(-\iota_{Dc_1\cdot c_1^{-1}})\Ad_{c_1} \exp(-\iota_{Dc_2\cdot c_2^{-1}}) \Ad_{c_2}=$$
$$=\exp(-\iota_{Dc_1\cdot c_1^{-1}+\Ad_{c_1}(Dc_2\cdot c_2^{-1})}) \Ad_{c_1c_2}=$$
$$=\exp(-\iota_{D(c_1c_2)\cdot (c_1c_2){^{-1}} })\Ad_{c_1c_2}= \tAd_{c_1c_2}$$
Next, observe that, because of the third equation in Definition \ref{dfn:dersquaze up inn},
$$T_g(Dc\cdot c^{-1})=T_g(Dc) T_g(c)^{-1}=D(T_g(c))T_g(c)^{-1}+[\alpha(g), T_g(c)]T_g(c)^{-1}=$$
$$=D(T_g(c))T_g(c)^{-1}+\alpha(g)-\Ad_{T_g(c)}(\alpha(g))$$
which implies
$$\cT_g \tAd_c \cT_{g^{-1}}=e^{\iota_{\alpha(g)}} T_g e^{-\iota_{Dc\cdot c^{-1}}} \Ad_c T_g^{-1} e^{-\iota_{\alpha(g)}} =$$ 
$$=\exp(\iota_{\alpha(g)}-\iota_{T_g (Dc\cdot c^{-1})}) \Ad_{T_g(c)} \exp(-\iota_{\alpha(g)})=$$
$$\exp(\iota_{\alpha(g)}-\iota_{T_g (Dc\cdot c^{-1})}-\iota _{T_g (\alpha(g))})\Ad_{T_g(c)}=$$
$$\exp(-\iota_{DT_g(c)\cdot T_g(c)^{-1}})\Ad_{T_g(c)}=\tAd_{T_g(c)}$$
which is d). Finally,
$$\cT_{g_1}\cT_{g_2}=\exp(\iota_{\alpha(g_1)} )T_{g_1} \exp(\iota_{\alpha(g_2)})T_{g_2}=\exp(\iota_{\alpha(g_1)+T_{g_1} \alpha(g_2)})T_{g_1g_2}=$$
$$+\exp(\iota_{\alpha(g_1)+T_{g_1} \alpha(g_2)})\Ad_{c_(g_1,g_2)} T_{g_1}T_{g_2}$$
while
$$\Ad_{c(g_1,g_2)}\cT_{g_1g_2}=\exp(-\iota_{Dc(g_1,g_2)c(g_1,g_2)^{-1}})\Ad_{c(g_1,g_2)} \exp(\iota_{\alpha(g_1g_2)})T_{g_1g_2}=$$
$$=\exp(-\iota_{Dc(g_1,g_2)c(g_1,g_2)^{-1}} -  \iota_{\Ad_{c(g_1,g_2)}(\alpha(g_1g_2)}) \Ad_{c_(g_1,g_2)} T_{g_1}T_{g_2}$$ 
which implies e) because of the last equation in Definition \ref{dfn:dersquaze up inn}.
\end{proof}

Let ${\bf a}=(a_1,\ldots, a_n).$ Define:
\begin{equation}\label{eq:Tg on tensors}
T_g{\bf a}=(T_ga_1,\ldots, T_ga_n);\; \Ad_c  {\bf a}=(\Ad_c a_1,\ldots \Ad _c a_n);
\end{equation} 
and
\begin{equation}\label{eq:diffl D 11}
\iota_a {\bf a}=\sum_{j=0}^n (-1)^{\sum_{i=1}^j (|a|+1)(|a_i|+1)} (a_1,\ldots, a_j, a, a_{j+1}, \ldots, a_n)
\end{equation}
(for a homogeneous element $a$).

By analogy with \eqref{eq:curvaT}, \eqref{eq:curvaT1}, set 
\begin{equation}\label{eq:ngrygv}
\cT_g=\exp(\iota _{\alpha(g)}) T_g
\end{equation}
for every $g\in G;$
\begin{equation}\label{eq:ngrygv 1}
{\widetilde{\Ad}}(c)=\exp(-\iota_{Dc\cdot c^{-1}}) \Ad(c)
\end{equation}
for every $c\in A^\times$ of degree zero.

Note that $n$ could be equal to zero. In this case $\iota_a({\mathbf a})=0,$ $T_g{\bf a}={\bf a},$ and $\Ad(c)({\mathbf a})={\mathbf a}.$

For ${\mathbf a}_1=(a_1,\ldots, a_{n_1}),\, {\mathbf a}_2=(a_{n_1+1},\ldots,a_{n_2}), $ {\em etc.}, put
$$\varphi({\mathbf a}_1,{\mathbf a}_2,\ldots)=\varphi(a_1,\ldots, a_{n_1},a_{n_1+1},\ldots,a_{n_2},\ldots)$$

Every choice of $n_1,\ldots, n_{m+1}\geq 0$ such that $n_1+\ldots+n_{m+1}=n$ defines a presentation $(a_1,\ldots,a_n)=({\mathbf a}_1,\ldots,{\mathbf a}_{m+1}).$ Define
$$|{\mathbf a}_k|=\sum _{i=n_k+1}^{n_{k+1}} |a_i|.$$
Put 
\begin{equation}\label{eq:big phi}
(\phi(c_1,\ldots,c_m)\varphi)(a_1,\ldots,a_n)=\sum_{n_1,\ldots, n_{m+1}}(-1)^{N(n_1,\ldots, n_{m+1})} 
\end{equation}
$$\varphi({\mathbf a}_1, c_1, \tAd_{c_1}^{-1} {\mathbf a}_2,  c_2, \tAd(c_1c_2)^{-1}{\mathbf a}_3,\ldots, c_m, \tAd(c_1c_2\ldots c_m)^{-1} {\mathbf a}_{m+1}) (c_1c_2\ldots c_m)^{-1} $$
Here 
$$N(n_1,\ldots, n_{m+1})= \sum_{j=1}^m \sum _{i=1}^j (|{\mathbf a}_i|+n_i)$$
\begin{lemma}\label{lemma: algebra acts, der case}
The operators $\cT_g$, ${\widetilde{\Ad}}(c),$ and $\phi(c_1,\ldots,c_m)$ satisfy all the relations of Lemmas \ref{dfn:higher on Ext} and \ref{lemma:tasovki}.
\end{lemma} 
\begin{proof}
Define for ${\mathbf a}=(a_1,\ldots,a_n)$ and for a homogenous derivation $E$
\begin{equation}\label{eq: iota Delta on tensors}
E{\bf a}=\sum _{j=1}^n (-1) ^{|E|\sum_{p<j}|a_p|}(a_1,\ldots, E a_j, \ldots, a_n)
\end{equation}
Also put
\begin{equation}\label{eq:partial bar diffl}
\partial {\mathbf a}=\sum _{j=1}^{n-1} (-1) ^{\sum _{p\leq j} |a_p|} (a_1,\ldots, a_ja_{j+1}, \ldots, a_n)
\end{equation}
Note that Lemma \ref{lemma:T curved g 1} holds for $\cT_g$ and $\tAd_c$ as in \eqref{eq:ngrygv}, \eqref{eq:ngrygv 1} and for $D,$ $\iota,$ {\em etc.} as above, if one replaces $\delta$ by $\partial.$ (In fact, a) can be easily checked, and the rest follows formally from a)). It is easy to deduce Lemma \ref{lemma: algebra acts, der case} from this.
\end{proof}
We get a generalization of Theorem \ref{thm:A infty action on Ext}:
\begin{thm}\label{thm:A infty action on Ext der case}
There is an $A_\infty$ action of $G$ on $C^\bullet (V,A,W)$ such that $T(g)$ is equal to $\cT_g$ as in \eqref{eq:curvaT}.
\end{thm} 
\subsubsection{Behavior with respect to equivalences}\label{eq:como and Yo}
Now consider an equivalence between two actions up to inner automorphisms and compatible derivations
\begin{equation}\label{eq:fat b exhibited}
{\bf b}=(\{b(g)\}, \beta): (T,c), (D,\alpha,R)\isomoto (T',c'), (D',\alpha',R')
\end{equation}
If $V$ is a module with a derivation $D_V$ and an action $T_g$ compatible with the action on the left, let ${\bf b}_* V$ be $V$ equipped with the derivation $D'_V$ and with the action $T'_g$ compatible with the action on the right ({\em cf.} \eqref{eq:equiv vs compat mo}).
Let 
$$\cB_c=\cB(A^\times, G\ltimes _c A^\times);\; \cB_{c'}=\cB(A^\times, G\ltimes _{c'} A^\times)$$ 
({\em cf.} definitions in Lemma \ref{lemma:Lemma A infty actions} and in Theorem \ref{thm:A infty action on Ext}).
\begin{lemma}\label{lemma:isos c c'}
The formulas
$$g\mapsto b(g)g, g\in G; c\mapsto c, c\in A^\times$$
define an isomorphism 
$$G\ltimes _c A^\times \otomosi G\ltimes _{c'} A^\times$$
of groups over $G.$ Together with 
$$\Phi(c_1,\ldots, c_m)\mapsto \Phi(c_1,\ldots, c_m),$$
they define an isomorphism of differential graded algebras 
$${\mathbf b}^\dagger: \cB_c\otomosi \cB_{c'}$$
over $k[G].$
\end{lemma}
\begin{definition}\label{dfn:ebeta}
$${\mathbf b}^*=\exp(\iota _\beta): C^\bullet (V,A,W)\otomosi C^\bullet ({\mathbf b}_*V,A,{\mathbf b}_*W)$$
\end{definition}
\begin{proposition}\label{prop:equivalence on bar cplx} If one views $C^\bullet (V,A,W)$ as a differential graded $\cB_{c'}$-modules {\em via} the morphism ${\mathbf b}^\dagger,$ then ${\mathbf b}^*$ is a morphism of differential graded modules over $\cB_{c'}$. For two composable equivalences ${\mathbf b}_1$ and ${\mathbf b}_2$, one has 
$$({\mathbf b}_1{\mathbf b}_2)^\dagger = {\mathbf b}_2^\dagger{\mathbf b}_1^\dagger; \;( {\mathbf b}_1{\mathbf b}_2)^*={\mathbf b}_2^* {\mathbf b}_1^*$$
\end{proposition}
\begin{proof} The statement follows from
\begin{lemma}\label{lemma: ebetta and its}
a) $\tAd_{b(g)} \cT_g \exp(\iota_\beta) =\exp(\iota_\beta)\cT'_g$

b) $\tAd_c \exp(\iota_\beta)=\exp(\iota_\beta) \tAd'_c$
\end{lemma} 
To prove the lemma, observe
$$\tAd _{b(g)} \cT_g \exp(\iota_{\beta}) {\cT'_g} ^{-1}=$$
$$\exp(-\iota_{Db(g)\cdot b(g)^{-1}}) \Ad_{b(g)} \exp(\iota _{\alpha(g)}) T_g \exp(\iota_\beta){ T'_{g}}^{-1} \exp(-\iota_{\alpha'(g)})=$$
$$\exp(-\iota_{Db(g)\cdot b(g)^{-1}}) \exp(\iota_{\Ad_{b(g)}\alpha(g)})\exp(\iota_{\Ad_{b(g)}T_g\beta}) \exp(-\iota_{\alpha'(g)})=\exp(\iota_\beta)$$ 
because of \eqref{eq:eq ders up to inn}. This proves a). To prove b), note that
$$\tAd_c \exp(\iota_\beta) \tAd_c^{-1}= \exp(-\iota_{Dc\cdot c^{-1}}) \Ad_c \exp(\iota_\beta) \Ad_c^{-1} \exp(\iota_{D'c\cdot c^{-1}})=$$
$$\exp(-\iota_{Dc\cdot c^{-1}}) \exp(\iota_{T_c\beta})  \exp(\iota_{Dc\cdot c^{-1}+\beta-T_c\beta})=\exp(\iota_\beta)$$
\end{proof}
\subsubsection{Behavior with respect to Yoneda product}\label{sss:under Yo}
Now let us describe the relation of the $A_\infty$ action on a quotient to Yoneda product
\begin{equation}\label{eq:coprod Bc}
\smile: C^\bullet (V_1,A,V_2)\otimes C^\bullet(V_2,A,V_3)\to C^\bullet(V_1,A,V_3)
\end{equation}
given by
\begin{equation}\label{eq:Yoneda}
 (\varphi\smile \psi)(a_1,\ldots,a_{m+n})=(-1)^{(|\varphi|+m)\sigma_j(|a_j|+1)} \varphi(a_1,\ldots,a_m)\psi(a_{m+1},\ldots,a_{m+n})
\end{equation}
\begin{lemma}\label{lemma:coproduct Bc}
The coproduct
$$\Delta \phi(c_1,\ldots, c_m)=\sum_{j=1}^m\phi(c_1,\ldots, c_j)\otimes c_1\ldots c_j \phi(c_{j+1},\ldots,c_m)$$
turns the algebra $\cB_c$ into a differential graded bialgebra. The morphism \eqref{eq:quism Bc} is a bialgebra morphism.  If we write $\Delta a=\sum a^{(1)} \otimes a^{(2)},$ then 
$$a (\varphi\smile \psi)=\sum a^{(1)} \varphi \smile a^{(2)} \psi $$
for $a$ in $\cB_c.$ Morphisms $b^\dagger$ from Lemma \ref{lemma:isos c c'} are morphisms of bialgebras.
\end{lemma}
The proof is straightforward.
\subsection{$A_\infty$ action on the standard complex: the case of Lie groupoids}\label{sss:action a inf lie grpoid case}
\subsubsection{$A_\infty$ action of a Lie groupoid}\label{sss:A inf actn LG}
Consider a Lie groupoid $\cG$ with the manifold of objects $M.$ Let $\cA^\bullet$ be a sheaf of $\cO_M^\bullet$-algebras with an action of $\cG$ up to inner automorphisms and with a compatible flat connection up to inner derivations as in \ref{sss:Conns up to inn}. 
Recall the presheaves $\ucG^{(n)}$ on $M^{n+1}$ \eqref{eq:under G n}. Let also
\begin{equation}\label{eq:G ij}
\ucG^{(n)}_{jk}=p_{jk}^{-1}\ucG
\end{equation}
where $p_{jk}: M^{n+1}\to M^2$ is the projection to the $j$th and $k$th components.
\begin{definition}\label{dfn:A infty Lie groupoid}
An $A_\infty$ action of $\cG$ on a differential graded $\cO_M^\bullet$-module $\cC^\bullet$ is a collection of morphisms
$$T: \cG^{(n)}\to {\underline{\Hom}}^{1-n}(p^*_{n+1}\cC^\bullet, p^*_{1}\cC^\bullet),$$
$n\geq 1,$ such that \eqref{eq:A infty action} holds for every $g_1,\ldots, g_n$ where $g_j$ is a local section of $\ucG^{(n)}_{j,j+1}$. 
\end{definition} 
An $A_\infty$ morphism of $A_\infty$ actions is a collection of morphisms
$$\phi: \ucG^{(n)}\to {\underline{\Hom}}^{-n}(p^*_{n+1}\cC^\bullet, p^*_{1}\cC^\bullet),$$ 
$n\geq 0,$
such that \eqref{eq:A infty action mor} holds. 
\subsubsection{Action on the standard complex}\label{sss:act st LG}
Let $\cV^\bullet$ and $\cW^\bullet$ be two graded $\cA^\bullet$-modules with compatible actions of $\cG$ and with compatible connections $\nabla.$ Sometimes, to distinguish, we denote the three connections by $\nabla_{\cA},$ $\nabla_{\cV},$ and $\nabla_{\cW}$ respectively. Compatibility means, as usual, that 
$$\nabla(av)=\nabla (a)v+(-1)^{|a|}a\nabla(v)$$
for $a\in \cA^\bullet$ and $v\in \cV^\bullet.$

\begin{definition}\label{dfn:st compl Lie grpoid}{\em The standard complex } $\cCb (\cVb, \cAb, \cWb)$ is the complex of sheaves 
$$\cC^m=\prod_{p+n=m} {\underline{\operatorname{Hom}}}^p_{{\cO_M^{\bullet}}}({\otimes}_{{\cO_M^{\bullet}}}^n{\cAb}, {\underline{\operatorname{Hom}}}_{{\cO_M^{\bullet}}}(\cVb,\cWb))$$
with the differential $\delta+\nabla+\iota_R$ ({\em cf.} \eqref{eq:diffl D}, \eqref{eq:diffl D 1}, and Corollary \ref{cor:full diffl on stancomp}).

\end{definition}
\begin{remark}\label{rmk:standard cplx conn sheav} In other words, $\cCb$ is the standard complex computed over the algebra of scalars $\cO_M^\bullet$ and sheafifed. An example arises when $\cA$ is a bundle of algebras with a flat connection, $\cV$ and $\cW$ are bundles of modules with compatible flat connections, $\cO_M$ is the differential graded algebra of forms, and $\cVb$, resp. $\cWb,$ is the module of $\cV$- (resp. $\cW$)-valued forms. In this case $C^\bullet(\cV,\cA,\cW)$ is a bundle of complexes with an induced flat connection, and $\cCb(\cVb,\cWb)$ is the complex of forms with values in this bundle. Our situation is different in only one regard. Namely, our $\cOb$ will be mainly the algebra of $\Lambda$-valued forms. Accordingly, the exact nature of local cochains $\varphi(a_1,\ldots, a_n;\; v)$ that we allow needs to be specified. We will do this in \ref{ss:Inverse Images}.
\end{remark}

\begin{thm}\label{thm:A infty action on Ext der case grpoid case}
There is an $A_\infty$ action of $\cG$ on $\cCb (\cVb,\cAb,\cWb)$ such that $T(g)$ is equal to $\cT_g$ as in \eqref{eq:curvaT}.
\end{thm} 
\begin{proof} The operators $T(g_1,\ldots,g_n)$ are computed by a recursive procedure from Remark \ref{rmk:all not bad c} where $\phi(c_1,\ldots,c_n)$ are as in \eqref{eq:big phi}. The only difference is that the morphism \eqref{eq:Bar to Endom} sends $c$ not to $\Ad(c)$ but to ${\widetilde{\Ad}}(c)$ (cf. \eqref{eq:curvaT}, \eqref{eq:curvaT1}).
\end{proof} 
\subsection{The cochain complex of an $A_\infty$ action}\label{sss:cochains a inf} Given a sheaf of $\cO_M^\bullet$-modules $\cM^\bullet$ with an $A_\infty$ action of a Lie groupoid $\cG,$ define
$$C^\bullet (M,\cM^\bullet)=\prod_{n=0}^\infty \Gamma(M^{n+1}, { \underline{\Hom}}(\ucG^{(n)}, p_1^* \cM^{\bullet-n}))$$
with the differential 
$$(d\Phi)(g_1,\ldots,g_{n+1})=\nabla_\cM \Phi(g_1,\ldots,g_{n+1})+\sum_{j=1}^n T(g_1,\ldots,g_j)\Phi(g_{j+1},\ldots,g_{n+1})+$$
$$+\sum_{j=1}^n (-1)^j \Phi(g_1,\ldots, g_jg_{j+1}, \ldots, g_{n+1})+(-1)^{n+1} \Phi(g_1,\ldots,g_{n})$$ 
Here $g_j$ is a local section of $\ucG^{(n)}_{j,j+1}$, {\em cf.} \eqref{eq:G ij}.
\section{The $A_\infty$ action of $\pi_1(M)$ on standard complexes of $\cA^\bullet_M$-modules}\label{s:A infty action for symp}
\subsection{The action of $\pi_1(M)$ up to inner automorphisms on $\cA^\bullet_M$}\label{ss:action of pi one on curvA} Assume that $M$ is a symplectic manifold with a chosen $\Sp^4$ structure. In this section we construct:

1) a groupoid $\tG_M$ together with an epimorphism $\tG_M\to \pi_1(M)$ and a morphism of groups
\begin{equation}\label{eq:kernel and inners}
 \Ker(\tG_{x,x}\stackrel{p}{\longrightarrow} \pi_1(M)_{x,x})\stackrel{i}{\longrightarrow} \cA_{M,x}^\times;
\end{equation}

2) an action of $\tG_M$ on $\cA_M$ up to inner automorphisms such that any element $h$ of $\Ker(p)$ acts by conjugation with $i(h)$;

3) a flat connection on $\cA_M$ up to inner derivations compatible with the action of $\tG_M$, such that $\nabla$ is a Fedosov connection $\nabla_\cA$ with curvature $\frac{1}{\hbar}\omega.$

A more straightforward construction works in general under the assumption that $M$ has an $\Sp^4$ structure and yields the connection with $R=\frac{1}{i\hbar} \omega.$ A construction that is a little more involved yields a connection with $R=0$ under an additional restriction:
\begin{equation}\label{eq:hurewicz on form is ze}
\langle \pi_2(M), [\omega]\rangle=0
\end{equation}
meaning that the class of the symplectic form vanishes on the image of the Hurewicz homomorphism. 

By Lemma \ref{lemma:quotactupto grpoids} we will conclude that 
\begin{proposition}\label{prop:curvAis loc syst up to inn}
The sheaf of algebras 
\begin{equation}\label{eq: curvA grad}
\cA^\bullet _M=\Omega^\bullet _M(\cA)
\end{equation}
of $\cA_M$-valued forms on $M$ carries an action of $\pi_1(M)$ up to inner automorphisms and a compatible flat connection up to inner derivations such that $\nabla$ is a Fedosov connection $\nabla_\cA$ with curvature $\frac{1}{\hbar}\omega.$
\end{proposition}
Now Theorem \ref{thm:A infty action on Ext der case grpoid case} implies
\begin{thm}\label{Thm:Main on A infty actions}
For any two differential graded $\cA_M^\bullet$-modules $\cVb,$ $\cWb$ with a compatible action of $\pi_1(M)$ and a compatible connection, the standard complex $\cCb(\cVb,\cAb,\cWb)$ has a natural $A_\infty$ action of $\pi_1(M).$
\end{thm} 
\subsection{The construction of the groupoid $\tG_M$}\label{ss:The constr of tG} There are two options for constructing the groupoid $\tiG_M$ and a flat connection up to inner derivations. 

\subsubsection{The connection with $R=\frac{1}{i\hbar}\omega$}\label{sss:not quite flat con} Assume that $M$ is a symplectic manifold with an $\Sp^4$  structure. Let $g_{jk}$ be an $\Mp(2n,\bR)$-cocycle whose projection to $\Sp$ is a cocycle representing the tangent bundle. Consider the groupoid of the bundle represented by the cocycle $g$ (viewed as a twisted bundle with $c$=1). Here the role of $G$ (as in \ref{ss:grpd of a twisted bdl}) is played by the group $\Mp(2n)$ as in \ref{s:meta}.

Consider a lifted Fedosov connection with curvature $\frac{1}{i\hbar}\omega$ ({\em cf.} Theorem \ref{thm:Fedosov classification}). This is a partial case of a connection defined in \ref{sss:conns on tw bdles}. Now we can define a flat connection up to inner derivations as in \ref{sss:conns on tw bdles}. (Observe that $\frac{1}{i\hbar}\bA$ is a Lie subalgebra of the associative algebra $\cA$ and $\Mp(2n,\bR)$ is a subgroup of $\cA^\times$).
\subsubsection{The connection with $R=0$}\label{sss:quite flat con}
Consider the cocycle $g_{jk}$ as above in \ref{sss:not quite flat con}. Consider $\tg_{jk}\in \exp(\frac{1}{i\hbar}\bR)$ defined by
\begin{equation}\label{eq:cochain g tilde}
\tg_{jk}=\exp(\frac{1}{i\hbar} f_{jk})
\end{equation}
where
\begin{equation}\label{eq:cochain g tilde 1}
\omega|U_j=d\alpha_j;\; \alpha_j-\alpha_k=df_{jk}
\end{equation}
Observe that 
\begin{equation}\label{eq:cochain g tilde 2}
c_{jkl}=\exp(\frac{1}{i\hbar} (f_{jk}+f_{kl}-f_{jl}))
\end{equation}
takes values in $\exp(\frac{1}{i\hbar}\bR)$ and represents the class $\exp(\frac{1}{i\hbar}[\omega]).$ If our lifted Fedosov connection is represented by a collection of ${\widetilde{\mathfrak g}}$-valued one-forms $A_j,$ then 
\begin{equation}\label{eq:Fedosov twisted corrected}
{\widetilde A}_j=\frac{1}{i\hbar}\alpha_j+A_j
\end{equation}
represents a {\em flat} connection in the twisted bundle given by $\tg_{jk}, c_{jkl}.$ Now we can define a flat connection up to inner derivations exactly as we did in \ref{sss:not quite flat con}. Now we have $R=0.$

There is a short exact sequence of groups
\begin{equation}\label{eq:SES of groups for tG}
1\to \Mp(2n,\bR)\to ( \tG_M)_{x,x}\to \pi_1(M,x)\to 1
\end{equation}
for any point $x$ of $M.$
\section{Resum\'{e} of the general procedure}\label{s:Resume} We summarize the construction that we described up to this point. This includes the notions of objects and the construction of the infinity local system of morphisms between two objects. Next (in section \ref{ss:induced mods}) we will present a construction of a special type of objects.
\subsection{$\Omega^\bullet_{\bK, M}$-modules and their inverse images}\label{ss:Inverse Images} Recall the definiton of the sheaf $\Omega^\bullet _{\bK, M}$ of $\bK$-valued forms on a manifold $M$ (Definition \ref{dfn:K forms}). We will be considering the following class of sheaves of $\Omega^\bullet _{\bK, M}$-modules. Start with a vector bundle $E$ (finite or profinite) and a fiber bundle ${\mathfrak X}$ on $M.$ Local sections of the module $\cM^\bullet_{\E,\gX}$ are countable sums
\begin{equation} \label{eq:sections M E X}
\sum_{\varphi,\Phi} a_{\Phi,\varphi} \exp(\frac{1}{i\hbar}\varphi) e_\Phi
\end{equation} 
where $a_{\Phi,\varphi}$ are local differential forms with coefficients in $E$, $\varphi$ are local sections of $C^\infty_M,$ $e_\Phi$ are formal symbols corresponding to local sections $\Phi$ of $\gX,$ and $\varphi\to +\infty.$ For a smooth map $M\to N$ we define
\begin{equation}\label{eq:inv im K}
f^*\cM^\bullet _{\E,\gX}=\cM^\bullet _{f^*E,f^*\gX}
\end{equation}
We consider the following differentials on $\cM^\bullet_{E,\gX}.$ Let $E_0$ be a fiber of $E$ and let $X$ be a fiber of $\gX.$ Choose any local trivialization of the bundles $E$ and $\gX$ near $x_0.$ Also choose any local coordinate systems on $M$ near $x_0$ and on $\gX$ near $\Phi(x_0).$ Then we can identify local sections of $E$ with local functions $M\to E_0$ and local sections of $\gX$ with local maps $M\to \bR^{{\rm{dim}} \gX}.$ We require the diferential to be of the form
\begin{equation}\label{eq:diffl on MEX}
\nabla_\cM \sum_{\varphi,\Phi} a_{\Phi,\varphi} \exp(\frac{1}{i\hbar}\varphi) e_\Phi =\sum_{\varphi,\Phi} da_{\Phi,\varphi} \exp(\frac{1}{i\hbar}\varphi) e_\Phi+
\end{equation}
$$
+\sum_{\varphi,\Phi} \frac{1}{i\hbar} \varphi '  a_{\Phi,\varphi} \exp(\frac{1}{i\hbar}\varphi) e_\Phi dx+\sum_{\varphi,\Phi} A(x,\Phi(x)) \Phi '(x)  a_{\Phi,\varphi} \exp(\frac{1}{i\hbar}\varphi) e_\Phi dx+ 
$$
$$
+\sum_{\varphi,\Phi} B(x,\Phi(x))  a_{\Phi,\varphi} \exp(\frac{1}{i\hbar}\varphi) e_\Phi 
$$
 Here $A$ and $B$ are local $\End(E_0)$-valued functions on $M\times X.$ If $\nabla_\cM$ is of the above form for one choice of the local trivializations then it is true for any such choice. We will use the shorthand
 \begin{equation}\label{eq:shorthand}
 \nabla_\cM e_\Phi=(A\Phi'+B) e_\Phi
 \end{equation}
 
 Now $f:M\to N$ is a smooth map. A differential $\nabla_\cM$ on $\cM^\bullet _{E,\gX}$ induces a differential $f^*\nabla_\cM$ on $f^* \cM^\bullet _{E,\gX}=\cM^\bullet _{f^*E,f^*\gX}$ as follows. Let $x$ be local coordinates on $M$, $y$ local coordinates on $N$, and let the map be locally of the form $y=f(x).$ If
 $$\nabla_\cM e_\Psi=(A(y,\Psi(y))\Psi'(y) dy+B(y,\Psi(y))dy) e_\Psi$$
 for any $\Psi,$ then
 $$f^* \nabla_\cM e_\Phi=(A(f(x),\Phi(x))\Phi'(x) dx+B(f(x),\Phi(x))f'(x) dx) e_\Phi$$
 
 In other words: let $p: \gX\to M$ be the projection. Locally in $\gX$ (near $\Phi(x)$), we require that there exist linear operators $A(z): T_{z}\gX_{p(z)} \to \End E_{p(z)}$ and $B(z):T_{p(z)} M\to \End E_{p(z)}$ and a linear projection $P(z): T_z \gX \to T_z \gX_{p(z)},$ all smoothly depending on $z\in \gX,$ such that for any point $x$ of $M$ and for any $\eta\in T_xM,$ 
 \begin{equation}\label{eq:P A X etc}
 \nabla_\cM e_\Phi (x)(\eta) =( A(\Phi(x))P(d\Phi (x))\eta+B(\Phi(x))\eta)e_\Phi
 \end{equation}
 Note that if $\nabla_\cM$ satisfies this property for one choice of $P$ then it satisfies it for any other choice. This is because for any two projections $P_1$ and $P_2$, $(P_1-P_2)d\Phi(x): T_xM\to T_{\Phi(x)} \gX_{\Phi(x)}$ is a linear operator depending only on the value $\Phi(x).$ 
 
 For $f:M\to N,$ if $\nabla_\cM$ is locally determined by $A(z),$ $B(z),$ and $P(z),$ so is $f^*\nabla_\cM$.
\subsection{Oscillatory modules}\label{ss:Oscillatory mods} Consider the bundle $\cA_M^\bullet$ with the action of the groupoid $\tG_M$ up to inner automorphisms and a compatible flat connection up to inner derivations as defined in \ref{sss:quite flat con}. By definition, an oscillatory module $\cV^\bullet$ is a graded module over $\cA_M^\bullet$ of the type defined in \ref{ss:Inverse Images}, with a compatible action of the groupoid $\tG_M$ and a compatible flat connection as in \ref{ss:MCSLGrpd}.  

\subsection{$\Omega^\bullet_{\bK, M}$-modules with $\pi_1$-action}\label{ss:Omega with pi 1} These modules are defined in \ref{sss:A inf actn LG} (in our case here, $\cO_M^\bullet=\Omega_{\bK,M}^\bullet$  as in Definition \ref{dfn:K forms}). More generally, {\em twisted} $(\Omega^\bullet_{\bK, M}, \pi_1(M))$ modules are defined in \ref{ss:twisted A inf mos gros}.  By Theorem \ref{thm:A infty action on Ext der case grpoid case} and Lemma \ref{lemma:quotactupto grpoids}, under the assumptions $c_1(M)=0$ and \eqref{eq:hurewicz on form is ze}, the standard complex (Definition \ref{dfn:st compl Lie grpoid}) of two oscillatory modules is a twisted $\Omega^\bullet_{\bK, M}$-module with $\pi_1$-action. We denote this complex by $\cC^\bullet (\cVb, \cA^\bullet, \cWb).$
\subsection{Infinity local systems of $\bK$-modules}\label{ss:Infty loc systs Kmods} An infinity local systems of $\bK$-modules on a manifold $X$ is a collection of complexes of $\bK$-modules $\cC^\bullet_x,$ $x\in X,$ together with linear maps
\begin{equation}\label{eq:loc syst formula 1}
T(g_1,\ldots,g_n):\cC^\bullet _{x_{n+1}}\to \cC^{\bullet+1-n} _{x_1}
\end{equation}
for any $g_j\in \pi_1(X)_{x_j,x_{j+1}},$ $j=1,\ldots,n,$ subject to \eqref{eq:A infty action}. In other words, this is an $A_\infty$ action of the fundamental groupoid $\pi_1(X),$ {\em cf.} \ref{sss:grpds ainfty action ext}. 
\subsubsection{From twisted $(\Omega^\bullet_{\bK, M}, \pi_1(M))$ modules to infinity local systems}\label{sss:from ommods to locsyss}
If $\cM^\bullet$ is an $\Omega^\bullet_{\bK, M}$-module with a twisted $\pi_1$-action (as in \ref{ss:Omega with pi 1}), then 
\begin{equation}\label{eq:ind lim M}
\cC^{\bullet}_x=\varinjlim_{x\in U} C^{\bullet}(U, \cM^\bullet) 
\end{equation}
is an infinity local system of $\bK$-modules. ({\em cf.} \ref{sss:cochains a inf} for the definition of the cochain complex $C^{\bullet}(U, \cM^\bullet)$). This is explained in detail in \ref{sss:From twisted mods to inf loc sys}.
\begin{definition}\label{def:HOM}
Given two oscillatory modules $\cV^\bullet$ and $\cW^\bullet$ on a symplectic manifold $M$ that has an $\Sp^4$ structure and satisfy \eqref{eq:hurewicz on form is ze}, we denote by $\uRHOM (\cVb, \cWb)$ the complex $\cC^\bullet$ ({\em cf.} \ref{ss:Infty loc systs Kmods}) constructed from the complex $\cM^\bullet=\cC^\bullet(\cVb, \cA^\bullet, \cWb)$ ({\em cf}. \ref{ss:Omega with pi 1}).
\end{definition} 
\section{Objects constructed from Lagrangian submanifolds}\label{s:Objects from Lagrangians}
\subsection{Induced modules}\label{ss:induced mods}
\subsubsection{The case of groups acting on algebras}\label{sss:ind grps on algs} Let $i:B\to A$ be a morphism of algebras and let $j: P\to G$ be a morphism of groups. Assume that $P$ acts on $B$ by automorphisms and $G$ acts on $A$ by automorphisms. We denote these automorphisms by $S_p,$ $p\in P,$ and $T_g,$ $g\in G.$ We assume that $i(T_p b)=T_{jp} (ib)$ for any $p$ and $g.$ For simplicity, we consider here only true actions, {\em i.e.} those for which $c(g_1,g_2)=1$ and $c(p_1,p_2)=1.$

Let $W$ be a $B$-module with a compatible action of $P$ denoted by $S_p: W\isomoto W,$ $p\in P. $ Define the induced module $V$ as follows. First consider the $A$-module $A\otimes _B W.$ Note that it carries a compatible action of $P:$ 
\begin{equation}\label{eq:der on ind mo}
S_p(a\otimes w)=T_{jp}(a) S_p(w).
\end{equation}
Now let $V$ be the quotient of the space of formal linear combinations
\begin{equation}\label{eq:formal combs Tg}
\sum _{g \in G} T_g v_g,\; v_g \in A\otimes _B W,
\end{equation}
by the linear span of $T_{gj(p)} (a\otimes w)-T_g S_p(a\otimes w),$ $g\in G,\, p\in P,\, a\in A,\, w\in W.$ Define the $A$-module structure on $V$ by
\begin{equation}\label{eq:ind alge mod stru}
a\sum T_g g v_g=\sum T_{g} (T_g^{-1}a)v_g
\end{equation}
and a compatible group action of $G$
\begin{equation}\label{eq:ind grp mod stru}
T_{g_0}  \sum T_g  v_g=\sum T_{g_0g} v_g
\end{equation}
This is just another way of defining the induced module
\begin{equation} \label{eq:cross indu}
V=(G\ltimes A)\otimes _{P\ltimes B} W
\end{equation}
Now assume that $A$ and $B$ are graded algebras. Let $\{D:A\to A;\, \alpha(g)|g\in G; R_A\}$ and $\{E:B\to B;\, \beta(g)|g\in B; R_B\} $ be derivations of square zero of $A$ and of $B$ up to inner derivations. We assume that these derivations are compatible with $i$ and $j,$ {\em i.e.}
\begin{equation}\label{eq:comp for der grps}
i(E(b))=D(i(b));\, i(\beta(p))=\alpha(jp);\, i(R_B)=R_A.
\end{equation}
Let $E_W: W\to W$ be a compatible derivation of $W.$ Then $A\otimes _B W$ carries a derivation $E_{A\otimes _B W}$ compatible with the action of $B;$ 
\begin{equation}\label{eq:der on tens prod}
E_{A\otimes _B W}(a\otimes w)=D_A(a) \otimes w+(-1)^{|a|} a\otimes E_W(w).
\end{equation}
This allows to define a derivation of the induced module $V$ compatible with the action of $G;$
\begin{equation}\label{eq:der of indu}
D_V(\sum T_g v_g)=\sum T_g (\alpha(g^{-1})v_g)+\sum T_g E_{A\otimes _B W}(v_g)
\end{equation} 
\subsubsection{The case of groupoids}\label{sss:grpoids ind} Now generalize the situation of \ref{sss:ind grps on algs} to the case when $P$ is a groupoid with the set of objects $Y$ and $G$ is a groupoid with the set of objects $X.$ Denote by $j:Y\to X$ the action of the morphism of groupoids $j$ on objects. In this case $A=\{A_x|x\in X\}, $ $B=\{B_y|y\in Y\},$ and $W=\{W_y|y\in Y\}.$ Put
\begin{equation}\label{eq:ind mod grpd}
(A\otimes_B W)_y = A_{jy} \otimes_ {B_y} W_y
\end{equation}
Formulas \eqref{eq:der on ind mo} and \eqref{eq:der on tens prod} define a compatible action of $P$ and a compatible derivation on $A\otimes_B W.$

\begin{equation}\label{eq:formal combs Tg groupoid}
V_x=\{\sum _{y\in Y, g \in G_{x,jy}} T_g v_g |v_g \in (A\otimes _B W)_y\}/ \langle T_{gj(p)} v-T_g(S_pv) \rangle
\end{equation}
Formulas \eqref{eq:ind alge mod stru}, \eqref{eq:ind grp mod stru}, \eqref{eq:der on tens prod}, and \eqref{eq:der of indu} define on $V$ an $A$-module structure, a compatible action of $G,$ and a compatible derivation.
\subsubsection{The case of Lie groupoids}\label{sss:Lie grpoids ind} Now let $\cG$ and $\cP$ be Lie groupoids with the manifolds of objects $X$ and $Y$ respectively. Let $j: \cP\to \cG$ be a morphism of Lie groupoids, {\em i.e.} a smooth map $X\to Y$ and a smooth map $\cP\to \cG$ over $X\times X$ that preserves the composition and the unit. Let $\cB^\bullet$ be a sheaf of $\cO_Y^\bullet $-algebras and let $\cA^\bullet$ be a sheaf of $\cO_X^\bullet $-algebras, together with a morphism $i: \cB^\bullet \to j^*\cA^\bullet.$ Consider an action $S$ of $\cP$ on $\cB^\bullet$ and an action $T$ of $\cG$ on $\cA^\bullet.$ We assume that the morphism $i$ preserves the action of $\cP.$ Furthermore, let $(\nabla_\cB, \beta, R_\cB)$ be a compatible flat connection up to inner derivations on $\cB^\bullet$ and let $(\nabla_\cA, \alpha, R_\cA)$ be a compatible flat connection up to inner derivations on $\cA^\bullet.$
We require the following compatibility conditions generalizing \eqref{eq:comp for der grps}:
\begin{equation} \label{eq:conds on conn Lie grpd case}
i(\nabla_\cB b)=(j^*\nabla_\cA)(ib)
\end{equation}
in $j^*\cA^\bullet$ on $Y,$ for any local section $b$ of $\cB^\bullet;$
\begin{equation}\label{eq:compat of curv}
i(R_\cB)=j^*(R_\cA)
\end{equation}
in $j^*\cA^\bullet;$ and
\begin{equation}\label{eq:comp conn Lie grpd 2}
i(\beta(p))=\alpha(jp)
\end{equation}
in $j^*\cA^\bullet$ for any local section $p$ of $\cP.$
\begin{remark}\label{rmk:alfa restricted why}
The latter equation requires some explanation. It is not {\em a priori} clear why, for a local section $g$ of $\cG,$ $\alpha(g)$ depends only on the restriction of $g$ to $Y\times Y.$ To ensure this, we will always assume that the form $\alpha(g)$ is obtained from a local section $g$ by the same procedure as the factor in front of $e_\Phi$ in the right hand side of \eqref{eq:shorthand} is obtained from a local section $\Phi.$
\end{remark}
Now assume that $\cW^\bullet$ is a $\cBb$-module with a compatible action $S$ of $\cP$ and a compatible connection $\nabla_\cW.$ The module $j^* \cA^\bullet \otimes _{\cBb} \cW$ has a compatible action $S$ of $\cP$ and a compatible connection $\nabla_{\cA\otimes _{\cB}\cW}$ given by
\begin{equation} \label{eq:S on ind}
S_p(a\otimes w)=T_{jp}(a)\otimes S_pw ;
\end{equation}
\begin{equation}\label{eq:nabla on ind grpds Li}
\nabla_{\cA\otimes _{\cB}\cW} (a\otimes w)=\nabla_\cA(a) \otimes w+(-1)^{|a|} a \otimes \nabla_\cW (w)
\end{equation}
Now define the induced module $\cV$ as follows. First, for any open subsets $U$ of $X$ and $U'$ of $Y$ and any smooth map $f:U\to U',$ let $\ucG _f$ be the inverse image of $\ucG$ under
\begin{equation}\label{eq:ucG f}
U\isomoto {\rm{graph}}(f)\hookrightarrow X\times Y \hookrightarrow X\times X
\end{equation}
The space of local sections of $\cV^\bullet$ over $U$ is the space of formal linear combinations
\begin{equation}\label{eq:combina f}
\sum_{U,U'}\sum  _{f: U\to U'} \sum_{g\in \ucG_f (U)} T_g v_g;\; v_g\in (\cAb\otimes_\cBb \cWb)(U')
\end{equation}
factorized by the linear span of
\begin{equation}\label{eq:factorizing by}
T_{gj(p)}(a\otimes w)-T_g (S_p(a\otimes w))
\end{equation}
for some $h:U'\to U''$, $f:U\to U',$ $g$ a local section of $\ucG_f(U)$, and $p$ a local section of ${\underline{\cP}}|{\rm{graph }}(h).$ We interpret $gj(p)$ as a local section of $\ucG_{hf}.$

Formulas 
\begin{equation}\label{eq:structures on ind v Lie case}
T_{g_0}\sum T_gv_g=\sum T_{g_0g} v_g;\; a\sum T_g v_g=\sum T_g (T_{g^-1}(a)v_g);
\end{equation}
\begin{equation}\label{eq:structures on ind v Lie case conn}
\nabla_\cV \sum T_g v_g=\sum T_g \alpha (g^{-1}) v_g+T_g \nabla_{\cA\otimes _\cB \cW} (v_g)
\end{equation}
define an $\cA^\bullet$-module structure, a compatible action of $\cG,$ and a compatible connection on $\cVb.$ Note that the last formula relies again on the assumption discussed in Remark \ref{rmk:alfa restricted why}. Indeed, we need to be sure that $\alpha(g^{-1})|{\rm{graph}}(f)$ depends only on $g|{\rm{graph}}(f).$
\subsubsection{General definition of an induced module}\label{sss:induced with a twistec}
Finally, let us assume, analogously to what we did in \ref{sss:The action of the quotient in the Lie groupoid case}, that there is a Lie groupoid $\Gamma$ on $X$ and a Lie groupoid $\Pi$ on $Y$ together with a morphism $\Pi\to j^*\Gamma,$ an epimorphism $\cG\to \Gamma,$ and an epimorphism $\cP\to \Pi$ such that the diagram 
$$
\begin{CD}
\cP @>>> \Pi\\
@VVV  @VVV\\
j^*\cG @>>> j^*\Gamma
\end{CD}
$$
commutes. Let $\cH_x=\Ker(\cG_{x,x}\to \Gamma_{x,x})$ and $\cQ_y=\Ker(\cP_{y,y}\to \Pi_{y,y}).$ Denote by $\ucH,$ resp. $\ucQ,$ the sheaf of sections of the bundle of groups $\cH,$ resp. $\cQ.$ We also assume that there are morphisms of sheaves $i: \ucH\to \cA^\times$ and $i: \ucQ\to \cB^\times$ such that the diagram 
$$
\begin{CD}
\ucQ @>>> \cA^\times\\
@VVV  @VVV\\
j^*\ucH @>>> j^*\cB^\times
\end{CD}
$$
commutes.
We also assume that the $\cB^\bullet$-module $\cW^\bullet$ and the flat connection up to inner derivations $(\nabla_\cB, \beta, R_\cB)$ satisfies
$$S_q w=i(q)w;\;\beta(q)=-\nabla_\cB i(q)\cdot( iq)^{-1}$$
for any local sections $q$ of $\ucQ$ and $w$ of $\cW^\bullet.$
\begin{definition}\label{dfn:ind with a tvistec} Under the assumptions above, the induced module is the quotient of the module $\cV^\bullet$ (\eqref{eq:structures on ind v Lie case}, \eqref{eq:structures on ind v Lie case conn}) by the submodule generated by elements $T_hv-i(h)v,$ $h$ being any local section of $\cH$ and $v$ any local section of $\cV^\bullet.$ 
\end{definition}
\subsection{The induced oscillatory module $\cV_L$}\label{ss:Objects from Lagrangians}
\subsubsection{The algebra $\cB$ and the module $\ffbV_{\bK}$}\label{sss:A0-mod V} Recall the grading
\begin{equation}\label{eq:grading Weyl alg 1}
|\fx_j|=|\fxi_j|=1;\;|\hbar|=2
\end{equation}
Now define
\begin{equation}\label{eq:formal V}
\fbV=\bC[[\hbar]];\; 
\ffbV=\{\sum _{k=-N} ^\infty v_k | v_k\in  \fbV [\hbar^{-1}]  _k \}
\end{equation} 
where $N$ runs through all integers. 
\begin{definition}\label{df:ffbV K etc} Put
\begin{equation}\label{eq:curW 0}
\ffbV_{\bK} = \{\sum _{k=0}^\infty  e^{\frac{1}{i\hbar}c_k} v_k | \in  \ffbV ;\, c_k\in \bR;\, c_k \to \infty\}  
\end{equation} 
\begin{equation}\label{eq:curW 1}
\ffbV_{\Lambda} = \{\sum _{k=0}^\infty  e^{\frac{1}{i\hbar}c_k} v_k | \in  \ffbV ;\, c_k\geq 0;\, c_k \to \infty\}
\end{equation} 
\end{definition} 

Now define the subalgebra $\cB$ of $\cA$ (Definition \ref{dfn:curvA}) by
\begin{equation}\label{eq:curvaBdef} 
\cB={\MP}(n)\ltimes  {\widehat{\widehat{\bA}}}
\end{equation}
({\em cf.} \ref{sss:Metalinear structures}).
\begin{lemma}\label{lemma:A0 acts on W0} The formulas
$$\fx\mapsto \fx;\; \fxi\mapsto i\hbar\ddfx;$$
$$ 
\left [ \begin{array}{cc}
b&0\\
0& b^{-1}\end{array}\right ] \mapsto T_b,\;(T_bf)(x)=\frac{1}{\sqrt{\det(b)}}f(b^{-1}x);
$$
$$
\left [ \begin{array}{cc}
1&a\\
0& 1\end{array}\right ] \mapsto \exp(-\frac{i\hbar}{2}a\ddfx^2)
$$
define an action of $\MP(n)$ that together with the action of ${\widehat{\widehat \bA}}$ turns $\ffbV_{\bK}$ into a $\cB$-module.
\end{lemma}
\begin{definition}\label{df:ind formal mod}
$$\fcV=\cA\fotimes _{\cB}\ffbV$$
\end{definition}  
Here by $\fotimes$ we mean the completed tensor product. Namely,
$$
\cA\fotimes _{\cB}\ffbV=\varprojlim _{N\to\infty} \cA\otimes _{\cB}\ffbV / \exp(\frac{N}{i\hbar}) \cA_\Lambda \otimes _\cB \ffbV
$$
In \ref{s:metaa} we interpret $\fcV$ as an algebraic version of the metaplectic representation (Proposition \ref{prop:meta as ind}). 
\subsubsection{The sheaf of algebras $\cB_L$ and the sheaf of modules ${\widehat{\widehat{\bV}}}_L$}\label{sss:sheaf BL} Let $L$ be a Lagrangian submanifold of $M.$ Recall that we assume the existence of an $\Sp^4$ structure on $M.$ Consider the restriction to $L$ of the $\Mp(n)$-valued cocycle ${\widetilde g}_{jk}$ as in \ref {sss:not quite flat con} or in \ref{sss:quite flat con} (it does not matter which one of them because $\omega|L=0$). Consider the cohomologous $\MP(n)$-valued cocycle ${\widetilde p}_{jk}$ as in \eqref{eq:lifting p t ij}.

The group $\MP(n,\bR)$ ({\em cf.} \ref{sss:Metalinear structures}) acts on $\cB$ by automorphisms. It also acts on ${\widehat{\widehat{\bV}}}_\bK$ compatibly. Let $\cB_L$ be the bundle of algebras and ${\widehat{\widehat{\bV}}}_L$  the bundle of modules on $L$ associated to these actions and to the principal $\MP$-bundle defined by $p_{jk}.$ Note that the Lie algebra ${\widetilde{\mathfrak{g}}}$ \eqref{eq:extension g} acts by derivations on $\cB$ and on ${\widehat{\widehat{\bV}}}_\bK.$ Therefore any given Fedosov connection defines a connection on $\cB_L$ and on ${\widehat{\widehat{\bV}}}_L.$ If the curvature of this connection is $\frac{1}{i\hbar}\omega$ then the connection on ${\widehat{\widehat{\bV}}}_L$ is flat. We denote these connections by $\nabla_\cB$ and $\nabla_\bV.$
\begin{definition}\label{dfn:curvablagr}  By $\cB^\bullet _L,$ resp. by  ${\widehat{\widehat{\bV}}}_L^\bullet,$  we denote the differential graded algebra of $\cB_L$-valued $\bK$-forms with differential $\nabla_\cB,$ resp. the differential graded module of ${\widehat{\widehat{\bV}}}_L$-valued $\bK$-forms with differential $\nabla_\bV.$
\end{definition}
\subsubsection{The Lie groupoid $\bfP_L$}  \label{sss:Lie groupoid PL} \eqref{eq:lifting p t ij}   Construct the groupoid $\bfP_L$ as the groupoid of the (twisted in general, but not in this case) bundle defined by this cocycle as in \ref{ss:grpd of a twisted bdl}. We have a short exact sequence of groups 
\begin{equation}\label{eq:short e s PL}
1\to \MP(n,\bR)\to (\bfP_L)_{x,x} \to \pi_1(L,x)\to 1
\end{equation}
for every point $x$ of $L.$
\begin{definition}\label{dfn:VL}   
Let ${\widehat{\widehat{\bV}}}_L^\bullet$ be the $\cB^\bullet _L$-module with the compatible action of $\bfP_L$ and the compatible connection $\nabla_\bV$ as in Definition \ref{dfn:curvablagr}. The oscillatory module $\cV^\bullet_L$ is the $\cA_M^\bullet$-module with a compatible action of $\tG_M$ and a compatible connection induced from ${\widehat{\widehat{\bV}}}_L^\bullet$ as in Definition \ref{dfn:ind with a tvistec}.
\end{definition}
\subsection{Filtrations} \label{ss:Filts}
\begin{proposition}\label{prop:Fil}
Assume that $L$ is a Lagrangian submanifold of $M$ such that $\langle [\omega, \pi_2(M,L)\rangle =0.$ Then there is a filtration ${\rm{Filt}}^a \cV^\bullet _L,$ $a\in {\mathbb R},$ such that:
\begin{enumerate}
\item ${\rm{Filt}}^a \cV^\bullet _L \subset {\rm{Filt}}^b \cV^\bullet _L$ for $a\geq b;$
\item ${\rm{Filt}}^a {\mathbb K} \cdot {\rm{Filt}}^b \cV^\bullet _L \subset {\rm{Filt}}^{a+b} \cV^\bullet _L;$
\item ${\rm{Filt}}^a \cV^\bullet _L $ is preserved by $\nabla_\cV$ and by the action of $\cA_M^\bullet$ (but not necessarily by the action of ${\widetilde{\mathbf G}}_M$).
\end{enumerate}
Here ${\rm{Filt}}^a {\mathbb K}$ consists of sums as in \eqref{eq:Lambda 1} with the additional condition $c_k\geq a$ for all $k.$
\end{proposition}
\begin{proof} Similarly to what we did in Section \ref{ss:charts and cocycles}, for any chart $T$ choose a one-form $\alpha_T$ on $T$ such $d\alpha_T=\omega_T,$ for any two charts $T$ and $T'$ a function $f_{TT'}$ on $T\times_M T'$ such that $\alpha_T-\alpha_{T'}=df_{TT'},$ and for any three charts $T,T',T''$ put $c_{TT'T''}={\rm{exp}}(f_{TT'}-f_{TT''}+f_{T'T''})$ which is a locally constant function on $T'\times _M T'\times _M T''.$ We can choose them in such a way that they all vanish on $L.$ For any path $T$ from $x_0$ in $L$ to $x_1$ in $M,$ and for a small open $U_{x_1}$ containing $x_1,$ we get an open subset $U_T$ of ${\widetilde{M/L}}$ (homeomorphic to the image of $U_{x_1}$ in $M/L$). Consider a cover of ${\widetilde{M/L}}$ by such $U_{T}.$ We would like to define ${\rm{Fit}}^0 \cV^\bullet _L$ to be the linear span of those elements of $\cV_L^\bullet$ that are, under the trivialization with respect to a chart $T$, represented by ${\rm{exp}}(\frac{1}{i\hbar} \varphi_T) g_T v$ where $g_T\in {\rm{Sp}}^4(2n),$ $v\in {\widehat{\widehat{\mathbb V}}}_\Lambda$, and $\varphi_T$ are some functions on $U_T$. To make such a filtration well-defined, we have to have
$${\rm{exp}}(\frac{1}{i\hbar} (\varphi_T-\varphi_{T'})) ={\rm{exp}}(\frac{1}{i\hbar}f_{TT'}) c_{TST'}$$
on $U_T\cap U_{T'}$, for any $T$ and $T'$ as above and for any homotopy $S$ between them. We will find such $\varphi_T$ if we show that the right hand side of the above:  a) does not depend on $S$ and b) defines a one-cocyclewith respect to the cover of ${\widetilde{M/L}}$ by $U_T$. But under our assumption b) follows immediately from Lemma \ref{lemma:strong ekviv}. As for a), for two different homotopies $S$ and $S'$ between $T$ and $T'$, 
$$c_{TST'}c_{SS'T}=c_{TSS'}c_{TS'T'}$$
But $c_{SS'T}=c_{TSS'}=1.$ Indeed, $S\times _M S'\times _MT=T$ and the same is true for $T'$, and $c$ vanishes when restricted to $L.$
\end{proof}
\subsection{The microsupport of a filtered module}\label{ss:microsu} Assume $\cV^\bullet$ has a filtration as in Proposition \ref{prop:Fil}. Define
\begin{equation}\label{eq:microsu}
\mu{\rm{Supp}}(\cV^\bullet)={\rm{supp}} H^\bullet (s(\varinjlim _{a\to 0,\;a\geq 0} {\rm{Filt}}^0 \cV^\bullet / {\rm{Filt}}^a \cV^\bullet ), \nabla_\cV)
\end{equation}
Here $s$ denotes the sheafification of a presheaf.
\subsection{The case of $\bR^{2n}$}\label{ss:the case of R 2n osc}
\subsubsection{The groupoid $\tG_M$}\label{sss:the groupoid tG} Sections of $\tG_M$ are in bijection with smooth functions $g(x_1,\xi_1; x_2,\xi_2)$ on $M\times M$ with values in ${\Mp(n)}$. We will denote a section corresponding to $g$ by a formal symbol 
\begin{equation}\label{eq:elts of tG flat case} 
\sigma (x_1,\xi_1; x_2,\xi_2) = \exp(\frac{1}{i\hbar}(\xi_2-\xi_1)\fx+(x_1-x_2)\fxi) g(x_1,\xi_1; x_2,\xi_2)
\end{equation}
The composition consists of formal multiplication of exponentials and multiplication of elements of ${\Mp(2n)}.$
\subsubsection{The flat connection up to inner derivations on $\cA_M$ compatible with the action of $\tG_M$}\label{sss:the connection on tG} For a section $\sigma$ as in \eqref{eq:elts of tG flat case},
$$-\alpha(\sigma)=\nabla_\tG \sigma\cdot \sigma^{-1}=d_{\rm{DR}}g\cdot g^{-1}+\frac{1}{i\hbar} (\xi_2 dx_2-\xi_1 dx_1)+$$
$$(-\frac{\fxi_1}{i\hbar}dx_1+\frac{\fx_1}{i\hbar}d\xi_1)-\Ad_g (-\frac{\fxi_2}{i\hbar}dx_2+\frac{\fx_2}{i\hbar}d\xi_2)$$
\subsubsection{The sheaf $\cV_f^\bullet $}\label{sss:the space Vf}
Denote by $\cV_f^\bullet $ the oscillatory module corresponding to the Lagrangian submanifold ${\rm{graph}}(df).$ One has
\begin{equation}\label{eq:ahut} 
\cV^\bullet _f={\widehat \cV}^\bullet _M=\Omega_{\bK,M}^\bullet (\fcV).
\end{equation}
In other words, local sections of $\cV^\bullet _f$ are $\cV$-valued $\bK$-forms on $M$ ({\em cf.} Definition \ref{df:ind formal mod}). 
\begin{remark}\label{rmk:expl of sigma() case R2n}
a) Sections of ${\widetilde{\widetilde{\mathbb V}}}_L$ are identified with ${\widetilde{\widetilde{\mathbb V}}}$-valued functions as follows. If $v(x,{\widehat x})$ is a ${\widetilde{\widetilde{\mathbb V}}}$-valued function, then the corresponding section of ${\widetilde{\widetilde{\mathbb V}}} _L$ is 
\begin{equation}\label{eq: exp v etc}
{\rm{exp}}(\frac{1}{i\hbar} (f(x+{\widehat x})-f'(x){\widehat x})) v(x,{\widehat x})
\end{equation}
b) A section $w(x,dx,{\widehat x})$ of \eqref{eq:ahut}  is identified with $\cV^\bullet _f$ given by 
\begin{equation}\label{eq: w x fx etc}
{\underline w}=\sigma ((x,\xi);(x,f'(x))) {\rm{exp}}(\frac{1}{i\hbar} (f(x+{\widehat x})-f'(x){\widehat x})) w(x,dx,{\widehat x})
\end{equation}
where $\sigma(x,\xi; x,f'(x))$ is as in \eqref{eq:elts of tG flat case}.
\end{remark}
\subsubsection{The connection on $\cV_f$}\label{sss:the conn Vf osc}
\begin{equation}\label{eq:nabla V osc}
\nabla_\cV=-\frac{\xi-f'(x)}{i\hbar}dx+(\frac{\partial}{\partial x}-\ddfx-\frac{1}{i\hbar}f''(x)\fx) dx + (\ddxi+\frac{1}{i\hbar}\fx)d\xi
\end{equation}
Indeed, under the identification as in b) in Remark \ref{rmk:expl of sigma() case R2n}, the connection $\Delta|L$ becomes
$${\rm{Ad}}\; {\rm{exp}} (-\frac{1}{i\hbar}(f(x+\fx)-f'(x)\fx))(-\frac{f'(x)}{i\hbar}+\frac{\partial}{\partial x}-\frac{\partial}{\partial \fx}+\frac{1}{i\hbar}f''(x)\fx)dx$$
which is equal to 
$$\nabla_{\rm{st}}=(\frac{\partial}{\partial x}-\frac{\partial}{\partial \fx})dx$$
Now, if we denote
$$\sigma_f=\sigma(x,\xi;x,f'(x)),$$
as well as 
$$A=\frac{1}{i\hbar}\xi dx - \frac{\partial}{\partial \fx} dx+\frac{\fx}{i\hbar} d\xi;\; p(x,\xi)=(x,f'(x)), $$
then
$$\nabla _\cV (\sigma_f {\underline w})= \sigma_f (A-p^*A) \sigma _f {\underline w}+ \sigma_f {\underline{\nabla_{\rm{st}} w}}$$
Since
$$A-p^*A=-\frac{1}{i\hbar} \xi dx- \frac{\partial}{\partial \fx} dx+\frac{\fx}{i\hbar} d\xi+\frac{1}{i\hbar} f'(x)dx+ \frac{\partial}{\partial \fx} dx-\frac{\fx}{i\hbar} f''(x)dx,$$
we conclude that \eqref{eq:nabla V osc} holds.
\subsubsection{The action of ${\widehat{\A}}$ on $\cV_f$}\label{sss:the action of A on Vf osc}
The formal variables act as follows: $\fx$ by multiplication, and $\fxi$ by $i\hbar\ddfx + f'(x+\fx)-f'(x)$.

Indeed, under the identification b) in Remark \ref{rmk:expl of sigma() case R2n}, $\fxi$ acts by
$${\rm{Ad}}\; {\rm{exp}} (-\frac{1}{i\hbar}(f(x+\fx)-f'(x)\fx)) (i\hbar \frac{\partial}{\partial \fx})=i\hbar \frac{\partial}{\partial \fx} +f'(x+\fx)-f'(x)$$
and $\fx$ acts by
$${\rm{Ad}}\; {\rm{exp}} (-\frac{1}{i\hbar}(f(x+\fx)-f'(x)\fx))(\fx)=\fx.$$
\subsubsection{The action of $\tG_M$ on $\cV_f$}\label{sss:the action of G on Vf osc}  A section $\sigma$ as in  \eqref{eq:elts of tG flat case} acts by
\begin{equation}\label{eq:action of tG on Vf osc}
\exp(-\frac{1}{i\hbar}(f(x_1+\fx)-f'(x_1)\fx-f(x_2+\fx)-f'(x_2)\fx) g(x_1,\xi_1;x_2,\xi_2)
\end{equation}
This is obvious because of how we make the identification in b), Remark \ref{rmk:expl of sigma() case R2n}.
\subsection{Comparison between $\cV_f^\bullet$ and $\cV_0^\bullet$}\label{ss:compa of Vs}
\begin{corollary}\label{cor:compa Vs} One has an isomorphism
$${\rm{exp}} (-\frac{1}{i\hbar}(f(x+\fx)-f'(x)\fx)): \cV_0^\bullet \isomoto \cV_f^\bullet$$
\end{corollary}
\subsection{The filtration and the microsupport}\label{ss:Fimi}
The filtration on $\cV_f^\bullet$ as constructed in \ref{ss:Filts} is defined as follows:
$${\rm{Filt}}^0 \cV^\bullet _f=\Omega ^\bullet _{{\mathbb R}^{2n}} (\cV_\Lambda)$$
where 
$$\cV_\Lambda={\rm{Sp}}^4 (2n) \cdot {\widehat{\widehat{\mathbb V}}}_\Lambda$$
The microsupport of $\cV^\bullet _\Lambda$ is ${\rm{graph}}(df)$, as seen from the formula for $\nabla_\cV$ in \ref{sss:the conn Vf osc}.
\section{ The complex computing $\RHom (\cV_0^\bullet, \cV^\bullet)$}\label{s:Rhom int points Maslov}
\subsection{The simplified version}\label{ss:simpliver} Set, as above, ${\mathbb A}={\mathbb C}[\fx, \fxi, \hbar]$ with the Moyal-Weyl product; $\tiG={\rm{Sp}}^4(2n);$ $\tiP={\rm{MPar}}(2n);$ $\cA_0={\mathbb C}[\tiG]\ltimes {\mathbb A}[\hbar ^{-1}];$ $\cB_0={\mathbb C}[\tiP]\ltimes {\mathbb A}[\hbar^{-1}];$ ${\mathbb V}={\mathbb C}[\fx,\hbar]$ on which $\fxi$ acts by $i\hbar \frac{\partial}{\partial \fx}$ and $\fx$ by multiplication; $\cV_0=\cA_0\otimes _{\cB_0} {\mathbb V}[\hbar^{-1}]$ (note that, since operators ${\rm{exp}}(ai\hbar (\frac{\partial}{\partial \fx})^2)$ are well defined on ${\mathbb V}$, the group $\tiP$ acts on ${\mathbb V}$ compatibly with the action of ${\mathbb A}$). In this simplified version all the tensor products and tensor products are not completed. We start with computing ${\rm{Ext}}^\bullet _{\cA_0} (\cV_0, \cV_0).$
\begin{proposition}\label{prop:simpliver}
$${\rm{Ext}}^\bullet _{\cA_0} (\cV_0, \cV_0)\isomoto H^\bullet (\tiP, {\mathbb C}[\hbar, \hbar^{-1}])$$
where $H^\bullet$ stands for the (discrete) group cohomology.
\end{proposition}
\begin{proof} First, by a version of Shapiro's lemma \cite{CE}, we have
$${\rm{Ext}}^\bullet _{\cA_0} (\cV_0, \cV_0)\isomoto H^\bullet (\tiP, C^\bullet ({\mathbb V}, {\mathbb A}, \cV_0))$$
where $C^\bullet$ in the right hand side is the standard complex computing ${\rm{Ext}}_{\mathbb A}({\mathbb V}, \cV_0).$ Secondly, we have
$$\cV_0=\bigoplus_{\lambda\in \tiG/\tiP} \cV_{0, \lambda}$$
An element $p$ of $\tiP$ sends $\cV_{0,\lambda}$ to $\cV_{0, p\lambda}$ and therefore we have a $\tiP$-module decomposition
$$\cV_0=\bigoplus_{\cO} \cV_\cO$$
where $\cO$ are $\tiP$-orbits in $\tiG/\tiP$ and
$$\cV_\cO=\bigoplus_{\lambda\in \cO} \cV_{0,\lambda}$$
\begin{lemma}\label{lemma:statfaza simpl}
For all $\cO$ except the one-point orbit,
$$H^\bullet (\tiP, C^\bullet ({\mathbb V}, {\mathbb A}, \cV_\cO))=0$$
\end{lemma}
This follows from results in \ref{sss:stationary phase statement}. Finally, ${\mathbb C}[\hbar, \hbar^{-1}]\to C^\bullet ({\mathbb V}, {\mathbb A}, {\mathbb V}[\hbar^{-1}])$ is a quasi-isomorphism and ${\mathbb V}[\hbar^{-1}]=\cV_\cO$ where $\cO$ is the one-point orbit.
\end{proof}
This proves the simplified case of Theorem \ref{thm:comp v0 vf}. The actual theorem is more complicated because our actual module consists of forms with values in completed $\cV_0$ and we take not only complexes of derived modules between them bur also the derived invariants of the fundamental groupoid with values in the De Rham complex. It is almost evident that passing to derived invariants of the fundamental groupoid will get rid of the dependence on the point $(x,\xi)$ of our space and reduce the problem to the above, after some completion and tensoring by the Novikov ring. The remainder of this section just makes this explicit (in addition to Section \ref{sss:stationary phase statement} that was mentioned and used above). The main and only point is to construct explicitly a resolution of the differential graded module $\cV_0$ that carries an action of $\pi_1({\mathbb R}^{2n}).$
\subsection{The statement of the result}\label{sss:statement Hom V0 V} Let $M=\bR^{2n}.$ Given an oscillatory module $\cV^\bullet $ on $M$, construct the following complex. Note first that the group $\MP(n,\bR)$ acts on the linear span of $d\fx_1,\ldots,d\fx_n$ through the projection $\MP(n)\to \ML(n).$ Introduce the vector space
\begin{equation}\label{eq:Space with dxhats}
\wedge (d\fx_1,\ldots, d\fx_n) d^{-\frac{1}{2}} \fx
\end{equation}
where
$$d^{\frac{1}{2}} \fx=(d\fx_1\ldots d\fx_n)^{-\frac{1}{2}}$$
is a formal element on which a pair $(g,u)$ in $\MP$ acts via multiplication by $u.$ 
Consider the space 
\begin{equation}\label{eq: V with dx hats}
\wedge (d\fx_1,\ldots, d\fx_n) d^{-\frac{1}{2}} \fx \otimes \cV^\bullet
\end{equation}
with the following structures.
\subsubsection{The differential}\label{sss:the differential cochains}
Define the differential on \eqref{eq: V with dx hats} as
$$\tnabla_\cV=\frac{1}{i\hbar}(\xi dx+\fxi d\fx-\fx d\xi)+\nabla_\cV$$
One checks that $\tnabla^2=0.$ In fact,
$$\frac{1}{i\hbar}\nabla_\cV (\xi dx+\fxi d\fx-\fx d\xi)+\frac{1}{(i\hbar)^2}(\xi dx+\fxi d\fx-\fx d\xi)^2=$$
$$\frac{1}{i\hbar}(-d\xi d\fx+d\xi dx +dx d\xi -d\fx d\xi)=0$$
\subsubsection{The action of $\MP$}\label{sss:the action of MP cochains} 
Denote by $\MP(n,\bR)_M$ the sheaf of smooth sections of the associated (in our case trivial) bundle of groups with fiber $\MP(n).$

There is an obvious action of $\MP(n,\bR)_M$ on \eqref{eq: V with dx hats} but we have to modify it to make it commute with the differential. Put
\begin{equation}\label{eq:Rh}
{\mathbf R}_h=h+\frac{1}{i\hbar}[\iota_{d\fx}dx, h]
\end{equation}
Here $[,]$ stands for the commutator of operators on \eqref{eq: V with dx hats}; 
$$\iota_{d\fx} dx=\sum_{j=1}^n \iota_{d\fx_j} dx_j;$$ 
and $\iota_{d\fx_j}$ is the graded derivation of $\wedge(d\fx_1,\ldots,d\fx_n)$ that sends $d\fx_j$ to one and $d\fx_k$ to zero for $k\neq j.$
One checks immediately that 
\begin{equation}\label{eq:R is action}
{\mathbf R}_{h_1} {\mathbf R}_{h_2}={\mathbf R}_{h_1h_2}
\end{equation}
\begin{lemma}\label{lemma:deformed Rh and alpha}
\begin{equation}\label{eq;comp h apfa etc 1}
\tnabla_{\cV}{\mathbf R}_h={\mathbf R}_h \tnabla_{\cV}
\end{equation}
\end{lemma}
\begin{proof} For a local section $h$ of $\MP(n)_M,$ define $\alpha(h)\in \Omega^1_M(\cA_M)$ by
\begin{equation}\label{eq:alpha of h in MP}
\alpha(h)=-dh\cdot h^{-1}+A_{-1}-\Ad_h (A_{-1})
\end{equation}
where $A_{-1}=\frac{1}{i\hbar}(-\fxi dx + \fx d\xi).$  
Note that
\begin{equation}\label{eq;comp h apfa etc 4}
\nabla_{\cV}({\mathbf R}_hv)=-\alpha(h){\mathbf R}_h  v+{\mathbf R}_h \nabla_\cV v;
\end{equation}
\begin{equation}\label{eq;comp h apfa etc 2}
-\frac{1}{i\hbar}(\xi dx+\fxi d\fx-\fx d\xi)({\mathbf R}_hv)=-\alpha(h){\mathbf R}_h  v-{\mathbf R}_h \frac{1}{i\hbar}(\xi dx+\fxi d\fx-\fx d\xi) v
\end{equation}
The first equation is equivalent to the fact that $\cV^\bullet$ is a differential graded $\cA^\bullet _M$-module. The second is checked by a direct computation:
$$\frac{1}{i\hbar}[\fxi dx+\xi dx -\fx d\xi, h]=-\frac{1}{i\hbar}[\fx d\xi, h];$$
$$\frac{1}{i\hbar}[\fxi dx+\xi dx -\fx d\xi, [\iota_{d\fx}dx, h]]=\frac{1}{i\hbar}[[\fxi dx+\xi dx -\fx d\xi, \iota_{d\fx}]dx, h]+$$
$$[\iota_{d\fx}dx, \frac{1}{i\hbar}[\fxi dx+\xi dx -\fx d\xi,h]]=\frac{1}{i\hbar}[\fxi dx, h]$$
(the second summand vanishes). Therefore 
$$\frac{1}{i\hbar}[\fxi dx+\xi dx -\fx d\xi,{\mathbf R}_h]=\frac{1}{i\hbar}[\fxi dx-\fx d\xi, h]=-[\nabla_\cV, h].$$
Lemma \ref{eq;comp h apfa etc 1} immediately follows.
\end{proof}
\begin{proposition}\label{prop:hom v0 v} The standard complex computing group cohomology
$$C^\bullet  (\MP(n)_M, \wedge (d\fx_1,\ldots, d\fx_n) d^{-\frac{1}{2}} \fx \otimes \cV^\bullet)$$
is quasi-isomorphic to the complex $\cC^\bullet(\cV_0^\bullet, \cA^\bullet_M, \cV^\bullet).$
\end{proposition}
More precisely,
$$C^\bullet =\oplus _{m=0}^\infty \Hom ((\MP(n)_M) ^m, \wedge (d\fx_1,\ldots, d\fx_n) d^{-\frac{1}{2}} \fx \otimes \cV^\bullet);$$
$$(\delta D)(h_1,\ldots,h_{m+1})=(-1)^m \tnabla _\cV D(h_1,\ldots,h_{m+1})+{\mathbf R}_{h_1} D(h_2,\ldots, h_{m+1})+$$
$$+\sum_{j=1}^m (-1)^j D(h_1,\ldots, h_jh_{j+1},\ldots, h_{m+1})+(-1)^{m+1}D(h_1,\ldots, h_m);$$
The following section \ref{s:the resol and the computat} is devoted to the proof of Proposition \ref{prop:hom v0 v}.
\subsection{The resolution of $\cV_0$ and the computation of $\bR{\Hom}(\cV_0,\cV)$}\label{s:the resol and the computat}
\subsubsection{A resolution of $\cV_0$}\label{sss:resolution of Vf osc} As above, let $M=\bR^{2n}.$ First construct a resolution $\bP^\bullet$ that is only free over $\wbA_M,$ not over $\cA_M.$ This resolution is a free module over 
\begin{equation}\label{eq:A dot formal formal}
{\widehat{\widehat{\bA^\bullet }}}_M= \Omega_M^\bullet (\wbA_M)
\end{equation}
with the space of generators 
$$
\wedge(e_1,\ldots, e_n)v_0; \; |v_0|=0;\; |e_j|=-1
$$
with the differential $\nabla_\bP$ defined by the following properties:
\begin{equation}\label{eq:diff on reso}
\nabla_\bP v_0=\frac{1}{i\hbar}(-\xi dx+\fx d\xi)v_0;\; \nabla e_j=\fxi_j;
\end{equation}
$$\nabla_\bP(a v)=\nabla_\cA a\cdot  v+(-1)^{|a|} a \nabla_\bP v$$
for any $a$ in $\ffAbM$ and any $v$ in $\bP;$  and 
$$\nabla_\bP(\beta v_0)=\nabla_\bP \beta\cdot  v+(-1)^{|\beta|} \beta \nabla_\bP v_0$$
for any $\beta$ in $\wedge(e_1,\ldots,e_n).$ A simple computation shows that $\nabla^2=0.$

Next we construct a $\cB_{M}^\bullet$-free resolution of the $\cB_{_M}^\bullet$-module $\fcV^\bullet _M$. Here, as always, $\cB_{M}^\bullet$ stands for forms with coefficients in the (trivial) bundle of algebras associated to $\cB,$ and $\fcV^\bullet _M$ stands for forms with coefficients in the bundle of modules associated to ${\widehat{\fcV}},$ {\em cf.} Definition \ref{df:ind formal mod}. We first observe that $\bP^\bullet$ is in fact a $\cA^\bullet$-module, though not free. Indeed, to define an $\cB_{M}^\bullet$-action, we have to define an $\MP_M$-action compatible with the action of the smaller algebra and with the differential. We are going to do this next.
\subsubsection{The action of $\MP(n)_M$} \label{sss:the action of MP} The action of $\MP(n)_M$ extends from $\fcV^\bullet _M$ to $\bP^\bullet$ because of the following. The group $\MP$ also acts on $\wedge(e_1,\ldots,e_n).$ The latter action is induced by the linear action on $\bR^n$ which in our context is the easiest to describe as follows: identify $e_j$ with $\fxi_j$ and therefore $\bR^n$ with the linear span of $\fxi_j$ in $\fbA.$ The action of $\MP$ through the composition $\MP\to\GL\to\Sp$ on $\fbA$ leaves this subspace invariant. This is the action that we mean. 

Recall again that an element of $\MP(n)$ may be represented by a pair
$$(\left [ \begin{array}{cc}
b&a\\
0&{b^t}^{-1} \end{array}\right ], u);\; \det(b)=u^2.
$$ 
This element sends $v_0$ to $u^{-1}v_0.$ Combined with the above, we get an action of $\MP(n)$ on $\wedge(e_1,\ldots, e_n)v_0.$

Unfortunately, this action does not make $\bP^\bullet$ a differential graded $\cB_{M}$-module. To achieve that, we have to change the action as follows:
\begin{equation}\label{eq:changed action of MP}
{\mathbf R}_h = h+[\frac{1}{i\hbar}edx, h]
\end{equation}
Here $edx=\sum_j e_j dx_j.$ The commutator is just the commutator of operators on $\bP^\bullet.$ This action, unlike the previous one, makes $\bP^\bullet$ a differential graded $\bP^\bullet$-module, which is equivalent to the following.

One has 
\begin{equation}\label{eq;comp h apfa etc}
\nabla_{\bP}({\mathbf R}_hv)=-\alpha(h){\mathbf R}_h  v+{\mathbf R}_h \nabla_\bP v
\end{equation}

\subsubsection{The resolution $\cP^\bullet$}\label{sss:the diff in P}
Now define 
\begin{equation}\label{eq:curvaP}
\cP^\bullet =\cB_{-\bullet} (\MP(n)_M, \bP^\bullet)=\oplus_{m=0}^\infty \bC[\MP(n)_M ]^{\otimes m} \fotimes \bP^{\bullet}
\end{equation}
The action of $\cB_{M}$ on $\cP^\bullet$ is given by
$$h((h_1,\ldots,h_m)\otimes v)=(hh_1,\ldots,h_m)\otimes {\mathbf R}_h v$$
({\em cf.} \eqref{eq:changed action of MP});
$$a((h_1,\ldots,h_m)\otimes v)= (h_1,\ldots,h_m)\otimes av$$
for $h$ in $\MP(n)_M$ and $a$ in ${\widehat{\widehat{\bA}}}_M.$

This is the standard bar resolution of the $\MP$-module $\bP^\bullet.$ More precisely, the differential is given by
$$\nabla_{\cP}=\nabla^{(0)}_{\cP}+\nabla^{(1)}_{\cP}$$
\begin{equation}\label{eq:diffl on curvaP2}
\nabla ^{(0)} ((h_1,\ldots, h_m)\otimes v)=(-1)^m (h_1,\ldots, h_m)\otimes \nabla_\bP v
\end{equation}
\begin{equation}\label{eq:diffl on curvaP}
\nabla^{(1)} ((h_1,\ldots,h_m)\otimes v)=\sum_{j=1}^{m-1} (-1)^{j} (h_1,\ldots,h_jh_{j+1},\ldots,h_m)\otimes v
\end{equation}
$$
+(-1)^m (h_1,\ldots,h_{m-1})\otimes  v
$$
Finally, put
\begin{equation}\label{eq:curvaR defini}
\cR^\bullet = \cA_M ^\bullet \fotimes_{\cB_{M}^\bullet} \cP^\bullet
\end{equation}
\subsection{The complex $\Hom(\cR^\bullet, \cV^\bullet)$}\label{ss:complex hom(R,V)} The complex 
\begin{equation}\label{eq:hom (R V) cplex}
\Hom_{\cA^\bullet} (\cR^\bullet, \cV^\bullet)
\end{equation}
is now straightforward to compute for any oscillatory module $\cV$ on $\bR^{2n}.$ It is the complex of cochains of the group $\MP(n)_M$ with coefficients in the module $\wedge(e_1^*,\ldots,e_n^*)v_0^*\otimes \cV,$
\begin{equation}\label{eq:hom R V explicitely}
\Hom_{\cA^\bullet} (\cR^\bullet, \cV^\bullet)\isomoto C^\bullet (\MP(n)_M, \wedge(e_1^*,\ldots,e_n^*)v_0^*\otimes \cV)
\end{equation}
Here $|e_j^*|=1;\; |v_0^*|=0;$ the action of $\MP$ on $\wedge(e_1^*,\ldots,e_n^*)v_0^*$ is dual to the one from \ref{sss:the action of MP}. It is straightforward that this complex is identical to the one in Proposition \ref{prop:hom v0 v}.
\subsubsection{The case $\cV=\cV_f$}\label{sss:V is Vf} Now we are able to compute $\bR\Hom_{\cA_M^\bullet} (\cV_0^\bullet, \cV_f^\bullet).$ 
Recall (Definition \ref{eq:curW 0})
$$\ffbV_{\bK} = \{\sum _{k=0}^\infty  e^{\frac{1}{i\hbar}c_k} v_k | \in  \ffbV ;\, c_k\in \bR;\, c_k \to \infty\}$$
Here we view this space with the following action of $\MP(n,{\bR}):$
$$ 
\left [ \begin{array}{cc}
b&0\\
0& b^{-1}\end{array}\right ] \mapsto S_b,\;(S_bf)(x)=f(b^{-1}x);
$$
$$
\left [ \begin{array}{cc}
1&a\\
0& 1\end{array}\right ] \mapsto \exp(-\frac{i\hbar}{2}a\ddfx^2)
$$
Now define the $\MP(n)$-module 
\begin{equation}\label{eq:space on which df acts}
\Omega_\bK^{\bullet,\bullet}=\wedge(d\fx_1,\ldots, d\fx_n)\otimes \bC[\Mp(n)]\otimes _{\MP(n)} \ffbV_{\bK}
\end{equation}
and the $\MP(n)_M$-module of $\Omega_\bK^\bullet$ forms with coefficients in \eqref{eq:space on which df acts}.
\begin{remark}\label{rmk:lambda forms with coeffs...} Intuitively, $\Omega_\bK^{\bullet,\bullet}$ is the space of expressions 
\begin{equation}\label{eq:formal double form}
\sum_{J,K,j} \exp(\frac{1}{i\hbar}\varphi_{j,J,K} (x,\xi,\fx)) a_{j,J,K}(x,\xi,\fx) dx_J d\fx_K
\end{equation}
where linear term of $\varphi_{j,J,K} (x,\xi,\fx)$ with respect to $\fx$ is zero, and its quadratic term may be infinite; more precisely, it is allowed to be not just a quadratic form but a point of the Lagrangian Grassmannian.
\end{remark}
The differential on $\wedge(d\fx_1,\ldots,d\fx_n)\otimes {\widehat \cV}^\bullet _\bK)$ is
\begin{equation}\label{eq:diffl when V is Vf}
d_f=\ddxi d\xi+(\ddx-\ddfx)dx+\ddfx d\fx+\frac{1}{i\hbar} (f'(x+\fx)-f'(x))d\fx+\frac{1}{i\hbar} (f'(x)-f''(x)\fx)dx
\end{equation}
One has
\begin{equation}\label{eq:diffl when V is Vf 1}
d_f=(\exp(-\frac{1}{i\hbar}(f(x+\fx)-f'(x)\fx)) d_0 (\exp(\frac{1}{i\hbar}(f(x+\fx)-f'(x)\fx))
\end{equation}

\begin{proposition}\label{prop:vf} The standard complex $\cC\bullet _{\cA_M^\bullet}(\cV^\bullet _0, \cVb_f)$
is quasi-isomorphic to the complex 
\begin{equation}\label{eq:cplex for hom v0 vf}
C^\bullet (\MP(n)_M, \Omega_\bK^{\bullet,\bullet}).
\end{equation}
\end{proposition}
\subsubsection{A stationary phase statement}\label{sss:stationary phase statement}
\begin{lemma}\label{lemma:statfaza} For any positive integer $p,$ consider $\bR^p$ viewed as a discrete group. One has
$$H^\bullet (\bR^p, \bC[\bR^p])=0.$$
\end{lemma}
\begin{proof} One has $\bR^p\isomoto \bigoplus \bQ.$ Therefore $\bR^p\isomoto \bQ\oplus \bR^p.$ By K\"{u}nneth formula, 
$$H^\bullet (\bR^p,  \bC[\bR^p])\isomoto H^\bullet(\bQ, \bQ)\otimes H^\bullet (\bR^p,  \bC[\bR^p]).$$
But $H^0 (\bQ,\bQ)=0.$ If $k$ is the minimal integer such that $H^k (\bR^p,  \bC[\bR^p])\neq 0,$ K\"{u}nneth formula tells that $H^k=0,$ whence the contradiction.
\end{proof}
\begin{corollary}\label{cor:statfaza 2}
Let $\Omega$ be an orbit of $\MP(n,\bR)$ in the Lagrangian Grassmannian $\Lambda(n)$ that consists of more than one point.Then 
$$H^\bullet(\MP(n), \bC[\Omega])=0.$$
\end{corollary}
\begin{proof}
Let $N$ be the subgroup of $\MP(n,\bR)$ consisting of pairs 
$$(\left [ \begin{array}{cc}
1&a\\
0&1 \end{array}\right ], 1)
$$ 
(in other words, $N=\Ker (\MP(n)\to \ML(n))$).
Choose a point in $\Omega.$ Denote its stabilizer by $Z.$ Then $Z$ is a real  vector subspace of $N.$ Let $W$ be a complementary subspace to $Z$. Consider the Lyndon spectral sequence 
$$E_2^{pq}=H^p (N/Z, H^q (Z, \bC[\Omega]))\implies H^{p+q}(N, \bC[\Omega]).$$
But $\Omega\isomoto Z$ as a $Z$-set, so $H^\bullet(N, \bC[\Omega])=0$ by Lemma \ref{lemma:statfaza}. Now consider the Lyndon spectral sequence 
$$E_2^{pq}=H^p (\ML(n), H^q (N, \bC[\Omega]))\implies H^{p+q}(\MP(n), \bC[\Omega]).$$
The statement follows.
\end{proof}
\subsection{The computation of $\uRHOM(\cV_0,\cV_f)$}\label{ss:comp of rhom 0 f}
Recall the notation
\begin{equation}\label{eq:alg of cochains of MPar body}
\cS^\bullet=C^\bullet (\MP(n), \bK)
\end{equation}
\begin{thm}\label{thm:comp v0 vf}
$$\uRHOM^\bullet  (\cVb_0,\cVb_f)\isomoto \cS^\bullet$$
with the action of a path from $(x_1\xi_1)$ to $(x_2,\xi_2)$ given by muttiplication by $\exp(\frac{1}{i\hbar} (f(x_1)-f(x_2)).$
\end{thm}
\begin{proof} First one checks that all the structures for $\cVb_0$ and $\cVb_f$ are conjugate by multiplication by $\exp(\frac{1}{i\hbar}(f(x+\fx)-f'(x)\fx)).$ So we can reduce the statement to the case $f=0.$ The cohomology in question is computed by the complex 
\begin{equation}\label{eq:extext}
C^\bullet (\pi_1(M), C^\bullet (\MP(n)_M, \Omega^{\bullet, \bullet})).
\end{equation}
First compute the cohomology of $\pi_1(M)$. An argument identical to the one in Introduction (starting before \eqref{eq:delta}) shows that this cohomology is isomorphic to 
\begin{equation}\label{eq:cohom after pi 1}
H^\bullet (\MP(n), \bC[\Mp]\fotimes _{\bC[\MP(n)]} {\widehat{\widehat{\bV}}}_\bK)
\end{equation}
(In other words, all dependence on $x,\xi,$ and $dx,d\xi$ is eliminated). Now, by Corollary \ref{cor:statfaza 2}, all contributions from all Lagrangian submanifolds other than $L_0=\{\xi=0\}$ are also eliminated. Our cohomology is therefore computed by the complex
\begin{equation}\label{eq:last simplifi}
C^\bullet (\MP(n), \wedge(d\fx_1,\ldots,d\fx_n)\otimes {\widehat{\widehat{\bV}}}_\bK)
\end{equation}
of group cochains of $\MP(n)$ with coefficients in the complex $\wedge(d\fx_1,\ldots,d\fx_n)\otimes {\widehat{\widehat{\bV}}}_\bK$ of formal forms in $\fx$ with the differential $\ddfx d\fx.$
\end{proof}
\subsection{The case of sheaves}\label{ss:sheaves see Maslov}  
Here we compare the computation above to the analogous computation for the microlocal category of sheaves as in \ref{ss:Sheaves intro}. 
\begin{proposition}\label{prop:V0 Vf locally} Let $f$ and $g$ be two $C^\infty $ functions on $\bR^n.$ For a bounded contractible open subset of $\bR^n,$ the module of horizontal sections of the local system $\uRHOM(\cV^\bullet _g, \cV^\bullet _f)$ on $U$ is a free $\cS^\bullet$-module with one generator $J(f,g)$ lying in ${\rm{Filt}}^{-\inf_{U}(f-g)}.$ The composition is as follows:
$$J(f,g)J(g,h)=\exp(\frac{1}{i\hbar}c(f,g,h))J(f,h)$$
where
$$c(f,g,h)=\inf_U(f-h)-\inf_U(f-g)-\inf_U(g-h))$$
\end{proposition}
\begin{proof} It is easy to see that
$${\underline {\bR\rm{HOM}}}(\cV^\bullet _g, \cV^\bullet _f)\isomoto {\underline {\bR\rm{HOM}}}(\cV^\bullet _0, \cV^\bullet _{f-g})$$
Put 
\begin{equation}\label{eq:J fg mods}
J(f,g)=\exp(\frac{1}{i\hbar}((f-g)(x+\fx)-(f-g)'(x)\fx-\inf_U(f-g))
\end{equation}
The statement follows from Theorem \ref{thm:comp v0 vf}.
\end{proof}
Compare this to the following result of Tamarkin. Recall the definitions from \ref{ss:Tamarkin}. Put
$$\bK_\bZ=\{\sum _{k=0}^{\infty} a_k e^{-\frac{c_k}{i\hbar}}\}$$ 
where $a_k \in \Z,$ $c_k\in \R,$ and $c_k\to \infty.$ For any two objects $\cF$ and $\cG$ of $D(T^*\bR^n),$ let $\HOM _\bK(\cF,\cG)=\bK_\bZ \otimes _{\Lambda _\bZ}\HOM(\cF,\cG).$ Let ${\rm{Filt}}^c \HOM_\bK=e^{\frac{c}{i\hbar}} \HOM.$
\begin{proposition}\label{prop:V0 Vf locally sheavs} Let $f$ and $g$ be two $C^\infty $ functions on $\bR^n.$ For a bounded contractible open subset $U$ of $\bR^n,$ consider the objects $\cF_f$ and $\cF_g$ of $D(T^*U)$ as in \ref{ss:Tamarkin}.  The complex ${ {\rm{HOM}}}_\bK(\cF_g, \cF_f)$ is quasi-isomorphic to a free $\bK_\bZ$-module with one generator $J(f,g)$ lying in ${\rm{Filt}}^{-\inf_{U}(f-g)}.$ The composition satisfies the same formulas as in Proposition \ref{prop:V0 Vf locally}. 
\end{proposition}
\begin{proof} Recall that $\cF_f=\bZ_{t+f\geq 0}.$ It is immediate that
\begin{equation}\label{eq:tatata}
{ {\rm{HOM}}}_\bK (\cF_g, \cF_f)\isomoto { {\rm{HOM}}}_\bK (\cF _0, \cF _{f-g})
\end{equation}
Let $J(f,g)$ be the morphism $\bZ_{t\geq 0}\to \bZ_{t+f-g-\inf_U(f-g)\geq 0}$ which is the restriction to the subset $\{t+f-g-\inf_U(f-g)\geq 0\}\subset \{t\geq 0\}.$ It is clear that the right hand side of \eqref{eq:tatata} is the free $\bK_\bZ$-module generated by $J(f,g),$ that $J(f,g)$ is in ${\rm{Filt}}^{-\inf_U(f-g)},$ and that the composition is as in Proposition  \ref{prop:V0 Vf locally}.
\end{proof}
\subsubsection{Matrix units}\label{sss:matrix units}
Now put 
\begin{equation}\label{eq:matr units}
{\mathbf E}_{\bof,\bog}=\exp(\frac{1}{i\hbar} \inf_U(f-g)) J(f,g)\in \HOM_\bK (\cF,\cG)
\end{equation}
in $D(T^*U).$ Then
\begin{equation}\label{eq:matr uns relat}
{\mathbf E}_{\bof,\bog} {\mathbf E}_{\bog,\boh}={\mathbf E}_{\bof,\boh}
\end{equation}
\section{$\RHom$ and theta functions}\label{s:Rhom theta} \subsection{Modules associated to the Lagrangian submanifold $\xi=mx$}\label{ss:xi-mx} In this section, $M={\bT}^2$ and ${\widetilde M}=\bR^2$ with the standard symplectic form $\omega=d\xi dx.$
\subsubsection{The groupoid $\tG_M$}\label{sss:the groupoid tG tor} Local sections of $\tG_M$ are in bijection with smooth local functions $g(x_1,\xi_1; x_2,\xi_2)$ on $\tM\times \tM$ with values in ${\Mp(2)}.$ As in \ref{sss:the groupoid tG}, we denote a section corresponding to $g$ by a formal symbol 
\begin{equation}\label{eq:elts of tG flat case tor} 
\sigma (x_1,\xi_1; x_2,\xi_2) = \exp(\frac{1}{i\hbar}(\xi_2-\xi_1)\fx+(x_1-x_2)\fxi) g(x_1,\xi_1; x_2,\xi_2)
\end{equation}
These sections satisfy
\begin{equation}\label{eq:rel for tor grpd 1}
\sigma(x_1, \xi_1;x_2,\xi_2)=\exp(\frac{1}{i\hbar}(x_1-x_2)) \sigma(x_1, \xi_1+1;x_2,\xi_2+1); 
\end{equation}
\begin{equation}\label{eq:rel for tor grpd 2}
\sigma(x_1, \xi_1;x_2,\xi_2)= \sigma(x_1+1, \xi_1;x_2+1,\xi_2).
\end{equation}
As in \ref{sss:the groupoid tG}, the composition consists of formal multiplication of exponentials and multiplication of elements of ${\Mp(2)}.$

The flat connection up to inner derivations on $\tG_M$ is given exactly as in \ref{sss:the conn Vf osc}: for a section $\sigma$ as in \eqref{eq:elts of tG flat case},
$$-\alpha(\sigma)=\nabla_\tG \sigma\cdot \sigma^{-1}=d_{\rm{DR}}g\cdot g^{-1}+\frac{1}{i\hbar} (\xi_2 dx_2-\xi_1 dx_1)+$$
$$(-\frac{\fxi_1}{i\hbar}dx_1+\frac{\fx_1}{i\hbar}d\xi_1)-\Ad_g (-\frac{\fxi_2}{i\hbar}dx_2+\frac{\fx_2}{i\hbar}d\xi_2)$$
\subsubsection{The sheaf $\cV_{L_m}^\bullet $}\label{sss:the space VLm}
Denote by $\cV^\bullet_{L_m} $ the oscillatory module corresponding to the Lagrangian submanifold $\xi=mx.$ Local sections of $\cV^\bullet _{L_m}$ are sums 
\begin{equation}\label{eq;sums of vs Lm}
v=\sum _{k\in {\bZ}} v_k
\end{equation}
where $v_k$ is a local section of $\cV^\bullet _{m\frac{x^2}{2}+kx}$ on $\tM.$ In other words, $v_k$ is an $\Omega_\bK$-form on $\tM$ with coefficients in $\fcV$ (Definition \ref{df:ind formal mod}). The connection $\nabla_\cV$ is given by ({\em cf.} \ref{sss:the conn Vf osc})
\begin{equation}\label{eq:nabla V osc tor}
\nabla_\cV v_k=(-\frac{\xi-mx-k}{i\hbar}dx+(\frac{\partial}{\partial x}-\ddfx-\frac{1}{i\hbar}m\fx) dx + (\ddxi+\frac{1}{i\hbar}\fx)d\xi)v_k
\end{equation}
The action of ${\widehat{\widehat \bA}}_M$ is as follows ({em cf.} \ref{sss:the action of A on Vf osc}): $\fx$ by multiplication, and $\fxi$ by $i\hbar\ddfx + m\fx$.

\begin{remark}\label{rmk:notes on Vk}
The component $v_k$ is an element of the form $\sigma(x,\xi;x,\xi-mx-k) w_k$ where $w_k$ is a local section of the module $\bV_{L_m}$ ({\em cf.} \ref{ss:Objects from Lagrangians}).
Also note that sums \eqref{eq;sums of vs Lm} may be infinite but we require that $v_k\in \exp(\frac{1}{i\hbar} N_k) \fcV_{\Lambda}$ where $N_k\to \infty$ as $|k|\to\infty.$
\end{remark}
Components $v_k$ satisfy
\begin{equation}\label{eq:conds per for v}
v_k(x,\xi)=v_{k+1}(x,\xi+1)=v_{k-m}(x+1,\xi).
\end{equation}

The action of $\tG_M$ on $\cV_{L_m}$ is as follows: 
\begin{equation}\label{eq:action of tG on Vf osc tor}
\sigma(x_1,\xi_1;\,x_2,\xi_2)v_k=\exp(-\frac{1}{i\hbar}(\frac{mx_1^2}{2}+kx_1-\frac{mx_2^2}{2}-kx_2)) g(x_1,\xi_1;x_2,\xi_2)v_k
\end{equation}
({\em cf.} \ref{sss:the action of G on Vf osc}).
It is easy to see directly that all the structures are compatible with each other (of course this also follows from the fact that the above construction is obtained by applying the general procedure of \ref{s:LagrDistr}).
\subsection{The computation of $\uRHOM(\cV^\bullet _{L_0},\cV^\bullet _{L_m})$}\label{ss:comp of rhom 0 m tor}
\subsubsection{Matrices with coefficients in $\cS^\bullet$}\label{sss:Matr(S)} 
 Let $e_\Lambda,$ resp. $E$, be the free module over $\Lambda,$ resp. $\bK,$ with generators $e_k,\,k\in \bZ.$ Recall the differential graded algebra $\cS$ from \eqref{eq:alg of cochains of MPar body}. Put also
 \begin{equation}\label{eq:alg of cochains of MPar Lambda}
\cS^\bullet_\Lambda=C^\bullet (\MP(n), \Lambda)
\end{equation}
Let
\begin{equation}\label{eq:Matr Lambd 2 a}
{\rm{Matr}} (\cS)=\varprojlim_{N\to\infty} \Hom(E, \cS^\bullet \otimes E)/\exp(\frac{1}{i\hbar}N)\Hom(E,\cS^\bullet _\Lambda \otimes E)
\end{equation}
Let ${\mathbf{E}}_{k\ell}$ be the matrix unit, {\em i.e.} the homomorphism sending $e_k$ to $e_\ell$ and $e_j$ to zero if $j\neq k.$
\subsubsection{}\label{sss:comptori}
\begin{thm}\label{thm:comp v0 vf tor}
The sheaf of complexes $\uRHOM^\bullet  (\cVb_{L_0},\cVb_{L_m})$ is quasi-isomorphic to the sheaf of sections of the trivial bundle with fiber ${\rm{Matr}}(\cS^\bullet),$
with the action of $\pi_1 (M)$ as follows. Let $\gamma_1$ and $\gamma_2$ be the two generators of $\pi_1(M),$ namely $\gamma_1$ the loop $\xi=\xi_0, x=x_0+t$ and $\gamma_2$ the loop $x=x_0, \xi=\xi_0+t.$ Then for a matrix unit ${\mathbf{E}}_{k\ell}$
$$\gamma_1^q \gamma_2^p: {\mathbf{E}}_{k\ell}\mapsto \exp(\frac{1}{i\hbar} ( \frac{mq^2}{2}+q(\ell-k))) {\mathbf{E}}_{k+p, \ell+p-mq}$$ 
\end{thm}
\begin{proof} First construct the $\cA^\bullet _M$-free resolution $\cR_{L_0}^\bullet$ of $\cV^\bullet _{L_0}$ as in \eqref{eq:curvaR defini}. Local sections of $\cR^\bullet _{L_0}$ are sums \eqref{eq;sums of vs Lm} with the same relations \eqref{eq:conds per for v} $m=0;$ $v_k$ are elements of $\cR_k^\bullet$ on $\tM$  which is constructed exactly as $\cR^\bullet$ in  \eqref{eq:curvaR defini} with the only modification: equation \eqref{eq:diff on reso tor} becomes
\begin{equation}\label{eq:diff on reso tor}
\nabla_\bP v_{0, k}=\frac{1}{i\hbar}(-(\xi +k) dx+\fx d\xi)v_{0,k}
\end{equation}
Now, local sections of $\Hom_{\cA^\bullet _M}(\cR^\bullet_{L_0}, \cV^\bullet_{L_m})$ are sums $\sum_{k,\ell} b_{k\ell}$ where 
$$b_{k\ell}\in \cC^\bullet _{k\ell};$$
here $\cC^\bullet _{k\ell}$ is the complex \eqref{eq:cplex for hom v0 vf} computed for the function
\begin{equation}\label{eq:C kl}
f_{k\ell}(x,\xi)=mx^2+(\ell-k)x
\end{equation}
Local sections $b_{k\ell}$ satisfy the following:
\begin{equation}\label{eq:bkl rela}
b_{k\ell}(x,\xi)=b_{k,\ell-m} (x+1,\xi)=b_{k+1,\ell+1} (x,\xi+1)
\end{equation}
(Note that all $\cC_{k\ell}^\bullet$ are identical as graded spaces, with the differential $d_{k\ell}$ on $\cC^\bullet_{k\ell}$ given by 
$$d_{k\ell}=\Ad(\exp(-\frac{1}{i\hbar}(\frac{mx^2}{2}+(\ell-k)x+\frac{m\fx^2}{2}))) d_{00}).$$
The action of the fundamental groupoid is as follows. A path $\gamma: (x_1,\xi_1)\to (x_2,\xi_2)$ in $\tM$ preserves each $\cC^\bullet_{k\ell}$ and acts on it by 
\begin{equation}\label{eq:b kl under gamma}
(\gamma b)_{k\ell}(x_1,\xi_1) = \exp(\frac{1}{i\hbar}(\frac{mx_1^2}{2}+(\ell-k)x_1-\frac{mx_2^2}{2}+(\ell-k)x_2 ))b_{k\ell}(x_2,\xi_2)
\end{equation}
because of \eqref{eq:elts of tG flat case} and because
$$(f_{k\ell}(x_1+\fx)-f'_{k\ell}(x_2)\fx)-f_{k\ell}(x_2+\fx)-f_{k\ell}'(x_2)\fx)=$$
$$=\frac{mx_1^2}{2}+(\ell-k)x_1-\frac{mx_2^2}{2}-(\ell-k)x_2.$$
When $x_2-x_1=q$ and $\xi_2-\xi_1=p,$ the right hand side of \eqref{eq:b kl under gamma} becomes 
$$(\gamma b)_{k+p,\ell+p-mq} (x,\xi)=\exp(\frac{1}{i\hbar} ( \frac{mq^2}{2}+q(\ell-k))) b_{k\ell} (x,\xi).$$
The statement now follows from Theorem \ref{thm:comp v0 vf}.
\end{proof} 
\begin{corollary}\label{cor:hor theta}
For $m>0,$ the space of horizontal sections of $\uRHOM^\bullet  (\cVb_{L_0},\cVb_{L_m})$ is $m$-dimensional over $\bK$ with the basis
$$\theta _a = \sum _{q\in \bZ} \sum_ {k\in \bZ} \exp(\frac{1}{i\hbar} (mq^2+aq)) {\mathbf{E}}_{k,k+a-qm}$$
where $a=0,1,\ldots,m-1.$
\end{corollary}
\subsection{The case of sheaves}\label{ss:sheaves see theta} Following Tamarkin, we define the category $D({\mathtt{\mathbf T}}^2).$ First define the folowing diffeomorphisms of $\bR\times \bR:$ 
\begin{equation}\label{eq:S12}
S_1(x,t)=(x+1,t);\; S_2(x,t)=(x,t+x);
\end{equation} 
One has
\begin{equation}\label{eq:S12rel}
S_2S_1=T_1S_1S_2;\;T_1S_1=S_1T_1;T_1S_2=S_2T_1
\end{equation} 
where $T_1(x,t)=(x,t+1).$ (In other words, we have an action of the Heisenberg group ${\rm{Heis}}(3,\bZ)$ on $\bR\times\bR$.)

Define objects of $D(\T^2)$ as equivariant objects of $D(\bR^2)$, {\em i.e.} objects $\cF$ of $D(\bR)^2$ together with isomorphisms 
\begin{equation}\label{eq:sigma12}
\sigma_1: \cF\isomoto S_{1*} \cF;\; \sigma_2: \cF\isomoto S_{2*}\cF
\end{equation}
in $\HOM_\bK$ such that 
\begin{equation}\label{eq:sigmas rel}
\sigma_2\sigma_1\tau_1=\sigma_1\sigma_2
\end{equation}
or more precisely
\begin{equation}\label{eq:sigmas relll}
(T_1S_1)_*\sigma_2\cdot T_{1*} \sigma_1\cdot \tau_{1}=S_{2*}\sigma_1\cdot\sigma_2
\end{equation}
as morphisms $\cF\to (S_2S_1)_* \cF=(T_1S_1S_2)_*\cF.$
\begin{example}\label{ex:Fm torus} For an integer $n,$ put
\begin{equation}\label{eq:Fm torus}
\cF_m=\prod_{k\in \bZ} \cF_{m\frac{x^2}{2}+kx}
\end{equation}
In fact, 
$$(S_{1}^qS_2^p)(x,t)=(x+q,t+px);$$
$$(S_{1}^qS_2^p)^* \cF_{m\frac{x^2}{2}+kx}=\bZ_{\{t+px+m\frac{(x+q)^2}{2}+k(x+q)\geq 0\}}=T_{m\frac{q^2}{2}+kq}^* \cF_{m\frac{x^2}{2}+(k+mq+p)x};$$
\end{example}
In other words, if 
\begin{equation}\label{eq:Lk vs Fs}
\cL_k=\cF_{m\frac{x^2}{2}+kx},
\end{equation}
then
\begin{equation}\label{eq:Fm on torus}
\cF_m=\prod _{k\in \bZ} \cL_k;\; (S_{1}^qS_2^p)^*\cL_k=(T_{\frac{mq^2}{2}+kq})^*\cL_{k+mq+p}
\end{equation}
\subsection{Comparison between the categories}\label{ss:comparison tori} Consider the following automorphisms of the pair $(\tG_{\bR^2},\, \cA_{\bR^2}).$ 
Let $\sigma (x_1,\xi_1; x_2,\xi_2)$ be as in \eqref{eq:elts of tG flat case tor}. Define
\begin{equation}\label{eq:rel for tor grpd 2 aut}
(S_1)\sigma(x_1, \xi_1;x_2,\xi_2)= \sigma(x_1+1, \xi_1;x_2+1,\xi_2).
\end{equation}
\begin{equation}\label{eq:rel for tor grpd 1 aut}
(S_2\sigma)(x_1, \xi_1;x_2,\xi_2)=\exp(\frac{1}{i\hbar}(x_1-x_2)) \sigma(x_1, \xi_1+1;x_2,\xi_2+1); 
\end{equation}
For a section $a$ of $\cA_{\bR^2},$ define
\begin{equation}\label{eq:S12 for Acurva}
(S_1a)(x,\xi,\fx,\fxi)=a(x+1,\xi,\fx,\fxi);\; (S_2a)(x,\xi,\fx,\fxi)=a(x,\xi+1,\fx,\fxi)
\end{equation}
It is easy to see that these maps preserve all the structures, {\em i.e.} the product on $\cA,$ the composition on $\tG,$ the action of $\tG$ on $\cA,$ and the flat connection up to inner derivations. Therefore for an oscillatory module $\cV^\bullet $ on ${\bR^2},$ one can define new oscillatory modules $S_1^*\cV^\bullet$ and $S_2^*\cV^\bullet$ as follows. As differential graded $\Omega^\bullet_\bK$-modules,  they are the inverse images of $\cV^\bullet$ under the shifts $(x,\xi)\mapsto (x+1,\xi)$ and $(x,\xi)\mapsto (x,\xi+1);$ the algebra $\cA_{\bR^2}$ and the groupoid $\tG_{\bR^2}$ act via automorphisms $S_1,\,S_2.$
One has
\begin{equation}\label{eq:s1 s2 Vm etc}
(S_2^p)^*(S_1^q)^* \cV^\bullet _{m\frac{x^2}{2}+kx}=\cV^\bullet _{m\frac{x^2}{2}+(k+mq-p)x}
\end{equation}
Note that the central subgroup $\{T_c|c\in \bZ\}$  of ${\rm {Heis}}(\bZ)$ acts on $\HOM(\cF_0,\cF_m).$ Therefore the automorphisms $\sigma_1$ and $\sigma_2$ generate an action of $\bZ^2.$
\subsubsection{Matrix units for tori}\label{sss:matuni for tori}
Put 
\begin{equation}\label{eq:fmk} 
f_{m,k}(x)=m\frac{x^2}{2}+kx
\end{equation}
Let $m>0.$ Define the matrix unit ${\mathbf{E}}_{k\ell}$ as follows. Let 
\begin{equation}\label{eq:math un klll}
{\mathbf E}_{f_{m_1,k}, f_{m+m_1, \ell}}\in \HOM_\bK (\cF_{f{m_1,k}}, \cF_{ f_{m+m_1, \ell}})
\end{equation}
be as in \eqref {eq:matr units}. Let $i_k,$ resp. ${\rm{pr}}_k,$ be the emmbedding of, resp. the projection onto, the $k$th component in the decomposition in \eqref {eq:Fm torus}. Define ${\mathbf E}_{k\ell}$ as the composition
$$ i_\ell \circ {\mathbf E}_{f_{m_1,k}, f_{m+m_1, \ell}}  \circ {\rm{pr}}_k:  \cF_{m_1} \to \cF_{m_1\frac{x^2}{2}+kx}\to \cF_{(m+m_1)\frac{x^2}{2}+\ell x}\to \cF_{m+m_1}$$
One has 
$$\HOM_\bK(\cF_{f_{m_1,k}}, \cF_{f_{m+m_1, \ell}})=\bK {\mathbf E}_{kl}$$
$${\mathbf E}_{j\ell}={\mathbf E}_{jk}{\mathbf E}_{k\ell}$$
\begin{proposition}\label{prop:action of sigma on Ekl} The action of the group $\bZ^2$ on $\HOM(\cF_0,\cF_m)$ is as follows. 
$$\sigma_1^q \sigma_2 ^p {\mathbf{E}}_{k\ell} = \exp(\frac{1}{i\hbar} (m\frac{q^2}{2}+(\ell-k)q) ){\mathbf{E}}_{\ell+p, k+p-mq}$$
\end{proposition}
Now let $m<0.$ There is a generator
\begin{equation}\label{eq:Zkl1}
{\mathbf Z} (f_{m_1,k}, f_{m+m_1, \ell})\in \bR^1\Hom (\cF_{f_{m_1,k}}, (T_{-\sup f_{m,\ell-k} })_* \cF_{f_{m+m_1, \ell}})
\end{equation}
obtained as follows. First, to simplify notation, assume $m_1=k=0,$ as well as $\sup(f_{m,\ell})=0$ (the general case follows immediately). Replace $\cF_0=\bZ_{t\geq 0}$ by the complex 
\begin{equation}\label{eq:new Z0}
\bZ_{t<0} \to \bZ
\end{equation}
The complex 
$\Hom (\bZ_{t<0} \to \bZ, \bZ_{t\geq f_{m,\ell}})$ is isomorphic to $\bZ\oplus\bZ\stackrel{(1,-1)}{\longrightarrow} \bZ$ and 
computes $\R\Hom (\cF_{f_{0,0}}, \cF_{f_{m,\ell}}).$ 

Put
\begin{equation}\label{Zkl2}
{\mathbf Z}_ {f_{m_1,k}, f_{m+m_1, \ell}}=\exp(\frac{1}{i\hbar} \sup( f_{m,\ell-k})) {\mathbf Z} (f_{m_1,k}, f_{m+m_1, \ell})
\end{equation}
Define ${\mathbf Z}_{k\ell}$ as the composition
$$ i_\ell \circ {\mathbf Z}_{f_{m_1,k}, f_{m+m_1, \ell}}  \circ {\rm{pr}}_k:  \cF_{m_1} \to \cF_{m_1\frac{x^2}{2}+kx}\to \cF_{(m+m_1)\frac{x^2}{2}+\ell x}\to \cF_{m+m_1}$$
We have thus defined 
\begin{equation}\label{eq:E and Z}
{\mathbf E}_{k\ell}(m_2,m_1) \in \HOM_\bK ^0 (\cF_{m_1}, \cF_{m_2}), \,m_1\leq m_2;
\end{equation}
\begin{equation}\label{eq:E and Z 1}
{\mathbf Z}_{k\ell}(m_2,m_1) \in \HOM_\bK ^1 (\cF_{m_1}, \cF_{m_2}), \,m_1> m_2; 
\end{equation}
They satisfy
\begin{equation}\label {eq:E and Z 2}
{\mathbf E}_{jk}(m_3,m_2) {\mathbf E}_{k'\ell}(m_2,m_1)=\delta_{kk'} {\mathbf E}_{k\ell}(m_3,m_1);
\end{equation}
\begin{equation}\label{eq:E and Z 3}
{\mathbf E}_{jk}(m_3,m_2) {\mathbf Z}_{k'\ell}(m_2,m_1)=\delta_{kk'} {\mathbf Z}_{k\ell}(m_3,m_1)
\end{equation}
if $m_1>m_3$ and zero otherwise;
\begin{equation}\label{eq:E and Z 4}
{\mathbf Z}_{jk}(m_3,m_2) {\mathbf E}_{k'\ell}(m_2,m_1)=\delta_{kk'} {\mathbf Z}_{k\ell}(m_3,m_1)
\end{equation}
if $m_1>m_3$ and zero otherwise;
\begin{equation}\label{eq:E and Z 5}
{\mathbf Z}_{jk}(m_3,m_2) {\mathbf Z}_{k'\ell}(m_2,m_1)=0
\end{equation}
\begin{proposition}\label{prop:action of sigma on Zkl} The action of the group $\bZ^2$ on $\HOM(\cF_0,\cF_m)$ is as follows. 
$$\sigma_1^q \sigma_2 ^p {\mathbf{Z}}_{k\ell} = \exp(\frac{1}{i\hbar} (m\frac{q^2}{2}+(\ell-k)q) ){\mathbf{Z}}_{\ell+p, k+p-mq}$$
\end{proposition}
\section{Appendix. Metaplectic and metalinear groups}\label{s:meta} We recall the classical material that is contained, for example, in \cite{GS} and in \cite{S}. 
\subsection{Metalinear groups and metalinear structures}\label{sss:Metalinear structures} Recall \cite{GS} that the metalinear group is by definition
\begin{equation}\label{eq:ML}
\ML (n,\R)=\{(g,\, z)| g\in \GL(n, \R),\, z^2=\det (g)\}
\end{equation}
This is a twofold cover of $\GL(n, \R)$. There is a morphism
\begin{equation}\label{eq:ML to C}
{\det} ^{\frac{1}{2}}:
\ML (n,\R)\to \C^{\times}; \; (g, z)\mapsto z.
\end{equation}
Denote by ${\rm MO}(n)$ the preimage of ${\rm{O}}(n)$ in $\ML(n).$ Let also
\begin{equation}
{\rm{MU}}(n)=\{(u,\, \zeta)| u\in {\rm{U}}(n, \C),\, \zeta^2=\det (u)\}
\end{equation}
\begin{definition}\label{dfn:metaplegro}
Let $\Mpp(2n, \R)$ be the universal twofold cover of $\Sp(2n, \R).$ We  call this group {\em the metaplectic group}.
\end{definition}
There is a commutative diagram
$$
\begin{CD}
{\rm MO}(n) @>>> \ML(n,\R)\\
@VVV  @VVV\\
{\rm{MU}}(n) @>>> \Mpp(2n, \bR)
\end{CD}
$$
where the horizontal embeddings are homotopy equivalences.

A metalinear structure on a real vector bundle $E$ is a lifting of the transition automorphisms $g_{jk}^E$ to an $\ML (n,\R)$-valued cocycle ${\widetilde g}_{jk}^E$. For a real bundle $E$ with a metalinear structure, the complex line bundle $\wedge^{\frac{1}{2}}E$ is by definition given by the transition automorphisms ${\det}^{\frac{1}{2}}({\widetilde g}_{jk}^E)$, {\em cf.} \eqref{eq:ML to C}.

A metaplectic structure on a symplectic vector bundle $E$ is a lifting of the transition automorphisms $g_{jk}^E$ to an $\Mpp (n,\R)$-valued cocycle ${\widetilde g}_{jk}^E$.
A metalinear structure on a manifold (resp. a metaplectic structure on a symplectic manifold) is by definition the corresponding structure on its tangent bundle.
\begin{lemma}\label{lemma:c1 and meta} A manifold $X$ has a metalinear structure if and only if $T^*X$ has a metaplectic structure.
If a symplectic manifold has a metaplectic structure then any Lagrangian submanifold of $M$ has a metalinear structure. 
\end{lemma}
\begin{proof} The obstruction to existence of a metalinear, resp. metaplectic, structure is as follows. Pick any transition isomorphisms $g_{jk}$ for the tangent bundle. Lift them to a cochain ${\widetilde g}_{jk}$ with values in $\ML,$ resp. in $\Mpp.$ Then compute the two-cocycle $a_{jk\ell}={\widetilde g}_{jk}{\widetilde g}_{k\ell}{\widetilde g}_{j\ell}^{-1}$ with values in $\bZ/2\bZ.$ The cohomology class of this cocycle is the obstruction. If $M=T^*X,$ this cohomology class is determined by its restriction to $X$. But on $X$ the symplectic transition functions $g_{jk}$ for $TM$ can be chosen as the image of $\GL(n)$-valued transition functions for $TX$ under the embedding $\GL\to \Sp.$ This proves the first statement of the Lemma. Now, for a Lagrangian submanifold $L$ of $M,$ the transition isomorphisms for $TM|L$ are cohomologous to an $\Mpp$-valued cocycle $p_{jk}:$ $g_{jk}=h_j p_{jk} h_k^{-1}.$ Lift $h_j$ to $\Mpp(2n)$ somehow. Put 
\begin{equation}\label{eq:lifting p t ij} 
{\widetilde p}_{jk}={\widetilde h}_j^{-1} {\widetilde g}_{jk} {\widetilde h}_k.
\end{equation}
This is a cocycle cohomologous to ${\widetilde g}_{jk}|L.$ It takes values in the preimage of the subgroup of $\Sp(2n)$ consisting of matrices preserving the Lagrangian submanifold $L_0=\{\fxi=0\}.$ The image of this cocycle under the projection to $\GL$ via $\ML$ is a cocycle defining the bundle $TX.$
\end{proof}
\subsection{The Maslov class of a Lagrangian submanifold}\label{ss:Maslov of L} 
\subsubsection{The case $c_1(M)=0$}\label{sss:Maslov when c1=0} 
The cohomology class of the two-cocycle $a_{jk\ell}$ constructed as in the proof of Lemma \ref{lemma:c1 and meta} above but when we use the universal cover $\tSp(2n,\bR)$ instead of $\Mpp(n).$ This is now a class in $H^2(M,\bZ)$ that represents $c_1(M),$ the first Chern class of $TM$ viewed as a complex bundle after we reduce the structure group $\Sp$ to the maximal compact subgroup $U(n).$ Indeed, $\tSp$ is homotopy equivalent to
$${\widetilde U}(n)=\{(u,x)|u\in U(n),\, x\in \bR, \det(u)=e^{2\pi i x}\}.$$ The proof of Lemma \ref{lemma:c1 and meta} applied to this case establishes the following fact.

Consider the group 
\begin{equation}\label{eq:GL inf cov}
{\widetilde{\GL}}(n,\bR)=\{(g,x)|x\in \GL(n,\bR);x\in \bR; \det (g)=e^{2\pi i x}\}
\end{equation}
(Of course ${\widetilde {\GL}},$ unlike ${\widetilde U}$ or $\tSp,$ has nothing to do with universal covers).
\begin{lemma}\label{lemma:Maslov CY} A trivialization of $c_1(M)$ defines a ${\widetilde{GL}}(n)$-structure on any Lagrangian submanifold $L$ of $M,$ {\em i.e.} a lifting of the transition automorphisms of $TL$ to a ${\widetilde{\GL}}(n)$-valued cocycle.
\end{lemma}
Assume that $L$ is oriented. Then there is another ${\widetilde{GL}}(n)$-structure on $L,$ due to the fact that ${\rm{SL}}(n)$ is a subgroup of ${\widetilde{GL}}(n).$ The two liftings differ by a class in $\lambda(L)\in H^1(L,\bZ).$ We will call this class {\em the Maslov class of an oriented Lagrangian submanifold} of a symplectic manifold $M$ with a trivialization of $c_1(M).$
\subsubsection{The case $2c_1(M)=0$}\label{sss:Maslov when 2c1=0}
Now consider the group 
\begin{equation}\label{eq:GL inf cov 1}
{\widetilde{U}}^{(2)} (n)=\{(g,x)|g\in U(n);x\in \bR; \det (g)^2=e^{2\pi i x}\}
\end{equation}
Note that 
\begin{equation}\label{eq:GL inf cov 2}
\{(g,x)|x\in \GL(n,\bR);x\in \bR; \det (g)^2=e^{2\pi i x}\}\isomoto \GL(n,\bR)\times \bZ
\end{equation}
Arguing exactly as before, we get
\begin{lemma}\label{lemma:Maslov 2CY} A trivialization of $2c_1(M)$ defines a ${{\GL}}(n)\times \bZ$-structure on any Lagrangian submanifold $L$ of $M.$ 
\end{lemma}
Projecting to $\bZ$, we get a class $\mu(L)\in H^1(L,\bZ).$ We call $\mu(L)$ {\em the Maslov class of a Lagrangian submanifold} of a symplectic manifold $M$ with a trivialization of $2c_1(M).$

Note that
\begin{equation}\label{eq:lammudva}
\mu(L)=2\lambda(L)
\end{equation}
for a trivialization of $c_1,$ the induced trivialization of $2c_1,$ and an oriented $L.$
\begin{remark}\label{rmk:t2Sp and Lambda un}
Let ${\widetilde{\Lambda}}(n)$ be the universal cover of the Lagrangian Grassmannian $\Lambda(n).$ Define the group $\tSp ^{(2)}(2n,\bR)$ by the condition that the following square be Cartesian.
$$
\begin{CD}
\tSp ^{(2)}(2n,\bR)  @>>> {\widetilde{\Lambda}}(n)\\
@VVV  @VVV\\
\Sp(2n,\bR) @>>> {\Lambda(n)}
\end{CD}
$$
Then ${\widetilde U}^{(2)}$ is a homotopy equivalent subgroup of $\tSp ^{(2)}(2n,\bR) .$
\end{remark}
\begin{example} \label{ex:Lambda 1} For $n=1,$ $U(1) \isomoto S^1;$ also $\Lambda(1)\isomoto S^1.$ Under these identifications, the projection $U(1)\to \Lambda(1)$ becomes the map $\zeta\mapsto \zeta^2.$
\end{example}
\subsection{The groups $\Sp^N$}\label{ss:SpN} Here we use definitions and notation from \cite{S}. For $N\geq 1,$ let $\Lambda^N (n)$ be the universal $N$-fold cover of $\Lambda(n).$ Define the group $\Sp^N(2n,\bR)$ by requiring the following diagram to be Cartesian:
$$
\begin{CD}
\Sp ^N(2n,\bR)  @>>> {\Lambda}^N (n)\\
@VVV  @VVV\\
\Sp(2n,\bR) @>>> {\Lambda(n)}
\end{CD}
$$
In other words, $\Sp^N(2n)=\tSp^{(2)}(2n)/(\bZ/N)$. Define also
$$U^N (n)=\{(u, \zeta)| u\in U(n), \zeta\in \bC, \det(u)^2=\zeta^N\}={\widetilde U}^{(2)}/(\bZ/N)$$
This is a subgroup of $\Sp^N(n)$ and the embedding is a homotopy equivalence. 

A $\Sp^N(2n)$-structure on $M$ is the same as a trivialization of $2c_1(M)$ in $H^2(M,\bZ/N).$

The universal $N$-fold cover of $\Sp(2n)$ is a subgroup of $Sp^{2N}(2n).$ In particular, the metaplectic group $\Mpp(2n)$ is a subgroup of $\Sp^4(2n).$ The latter is generated by $\Mpp(2n)$ and the central subgroup $\{\pm1,\pm i\}.$ The intersection of the two is $\{\pm 1\}, $ the kernel of $\Mpp\to \Sp.$

The following makes sense for any $N.$ We fix $N=4$ just to fix the notation for the rest of the paper.

\begin{definition}\label{def:MPar} a) Define $P(n,\bR)$ as the subgroup of $\Sp(2n,\bR)$ consisting of pairs $(A,z)$ where 
$A=\left [ \begin{array}{cc}
b&a\\
0& (b^{-1})^t\end{array}\right ]$
is a symplectic matrix. In other words, $P(n)$ is the subgroup of $\Sp(2n)$ consisting of matrices preserving the Lagrangian submanifold $L_0=\{\fxi=0\}.$

b) Define $\MP(n,\bR)$ as the subgroup of $\Mp(2n,\bR)$ consisting of pairs $(A,z)$ where 
$A=\left [ \begin{array}{cc}
b&a\\
0& (b^{-1})^t\end{array}\right ]$
is a symplectic matrix, $z$ is a complex number, and $\det(b)^2=z^4.$ In other words, this is the lifting to $\Mp(2n)$ of $P(n).$
\end{definition}
\begin{lemma}\label{lemma:coho cocy sp4}
a) $\MP(n,\bR)\isomoto P(n,\bR)\times \{\pm 1,\pm i\}$

b) If a symplectic manifold $M$ has an $\Sp^4$ structure and $L$ is a Lagrangian submanifold then formulas \eqref{eq:lifting p t ij} define an $\MP(n)$-valued cocycle cohomologous to the transition isomorphisms of $TM|L.$  

c) If $M$ has a real polarization then it has an $\Mp(2n)$-structure.
\end{lemma}
\begin{definition}\label{dfn:Maslov sp 4} The projection of the cohomology class from Lemma \ref{lemma:coho cocy sp4}, b) to $H^1(L,\bZ/4\bZ)$ is called {\em the Maslov class} of $L$.
\end{definition}
When the trivialization of $2c_1(M)$ modulo $4$ comes from a trivialization of $2c_1(M)$ then the Maslov class defined above is equal to $\exp(\frac{i\pi}{2}\mu(L))$ that was defined in \ref{sss:Maslov when 2c1=0}.
\section{Appendix. The algebraic metaplectic representation}\label{s:metaa} Most of the material of this section is contained in \cite{OM}. 
 Recall the algebra $\cA$ from \ref{ss:The alg curvA} and the $\cA$-module from Definition \ref{df:ind formal mod}. In this section we give an interpretation of this module in terms of the metaplectic representation. 
 \subsection{Symmetries of the deformation quantization algebra of a formal neighborhood}\label{ss:Deformation quantization of a formal neighborhood and symms}  Any continuous automorphism $g$ of $\A$ induces a symplectic linear transformation $g_0$ of $\C^{2n}.$ Denote by $G$ the group of those $g$ whose linear part $g_0$ preserves the real structure. We have
\begin{equation}\label{eq:G semi dir}
G=\Sp(2n,\R)\ltimes \exp (\g_{\geq 1})
\end{equation}
Define the central extension
\begin{equation}\label{eq:tG semi dir}
\tG=\exp(\oih \C\oplus \C)\times \tSp(2n,\R)\ltimes \exp (\tgg_{\geq 1})
\end{equation}
where $\tSp(2n,\bR)$ is the universal cover of $\Sp(2n,\bR).$
One has an exact sequence
\begin{equation}\label{eq:extension G}
1\to \Z\times \exp(\oih\C[[\hbar]])\to\tG\to G\to 1
\end{equation}

Define also $P$ to be the subgroup of $G$ consisting of elements $g$ whose linear part preserves the Lagrangian subspace
\begin{equation}\label{eq:L0}
L_0=\{\fxi_1=\ldots=\fxi_n=0\}
\end{equation}
Let $\tP$ be the preimage of $P$ in $\tG$.
\subsection{The algebraic Fourier transform}\label{ss:The algebraic Weil representation} Let $\fy=(\fy_1,\ldots,\fy_n)$ be $n$ formal variables. For a symmetric real $n\times n$ matrix $a$, put
\begin{equation}\label{eq:Ha}
\caH_a^{\fy}=\exp(\frac{a\fy^2}{2i\hbar})\fC[[\fy, \hbar]]((e^{\frac{c}{i\hbar}}|c\in \C))
\end{equation}
Here
\begin{equation}
\fC[[\fy, \hbar]]=\{\sum_{k=-N}^\infty v_k|v_k\in \C[[\fy]]((\hbar))_k\}
\end{equation}
with respect to the grading \eqref{eq:grading}; for any vector space $V,$ we define 
\begin{equation} \label{eq:completion}
V((e^{\frac{c}{i\hbar}}|c\in \C))=\{\sum_{k\in {\mathbb N};{\rm{Re}}(c_k)\to+\infty}e^{\frac{c_k}{i\hbar}}v_k\},
\end{equation}
$v_k\in V.$
In particular, the operator of multiplication by $h$ is automatically invertible.

Put also
For a nondegenerate $a$, define the Fourier transform (cf. \cite{K})
\begin{equation}\label{eq:Fourier}
F: \caH_a^{{\fy}}\isomoto\caH_{-a^{-1}}^{\feta}
\end{equation}
as follows. Heuristically,
\begin{equation}
(Ff)(\feta)=\frac {e^{-\frac{\pi i n}{4}}}{(2\pi i \hbar)^{n/2}}
\int e^\frac{\fy\feta}{i\hbar}f(\fy)d\fy;
\end{equation}
To give the above formula a rigorous meaning, put
$$
F(f(\fy)\exp(\frac{a\fy^2}{2i\hbar}))(\feta)
=f(i\hbar\frac{\partial}{\partial\feta})F(\exp(\frac{a\feta^2}{2i\hbar}))=
$$
$$
f(i\hbar\frac{\partial}{\partial\feta})\frac{e^{-\frac{\pi i n}{4}}}{\det(\sqrt{ia})}
\exp(\frac{-a^{-1}\feta^2}{2i\hbar})
=\frac{e^{-\frac{\pi i p(a)}{2}}}{\det\sqrt{|a|}}f(i\hbar\frac{\partial}{\partial\feta})\exp(\frac{-a^{-1}\feta^2}{2i\hbar})
$$
Here $p(a)$ is the number of positive eigenvalues of $a$. We used the branch of the square root for which ${\sqrt {ix}}>0$ if $ix>0$; it is defined on the complex plane with the line $\{ix<0, \; x\in \R\}$ removed. The final term in the above chain of equalities can be viewed as the definition of the first term.

\begin{remark}\label{rmk:who cares?} The definition of the Fourier transform $F$ extends to elements of the form \begin{equation}
{\bf f}(\fy)=\exp(\frac{a\fy^2}{2 i\hbar}+i\fy\fz)f(\fy)
\end{equation}
where $a$ is nondegenerate and $\fz$ is another formal parameter:
\begin{equation}\label{eq:Fourier with a shift}
F({\bf f})(\feta)=F(\exp(\frac{a\fy^2}{2i \hbar})f(\fy))(\feta+\fz)
\end{equation}
\end{remark}

One has
\begin{equation}\label{eq:Fourier props}
F^2 ({\bf f})(\fy)=i^{-n} {\bf f}(-\fy);\;F\fy F^{-1}=i\hbar\frac{\partial}{\partial\feta};\;Fi\hbar\frac{\partial}{\partial\fy}F^{-1}=-\feta
\end{equation}
\subsection{The two-dimensional case}\label{two dimen case} For the readers convenience, we first present the case $n=1.$
\begin{equation} \label{eq:H 2-dim cas}
\cH=\bigoplus_{a\in \bR}  \cH^{\fx}_a \bigoplus \bigoplus_{a\in \bR}  F\cH^{\fx}_a /\sim
\end{equation}
where
\begin{equation}\label{eq:equiv 2 dim}
F f(\fx)\exp ({\frac{a\fx^2}{2i\hbar}})
\sim \frac{e^{-\frac{\pi i}{2}p(a)}}{\sqrt{|a|}} 
f(i\hbar\frac{\partial}{\partial \fx}) \exp({-\frac{a^{-1}\fx^2}{2i\hbar}})
\end{equation}
for $a\neq 0.$ Here $p(a)=1$ if $a>0$ and $p(a)=0$ otherwise.
\subsubsection{The action of ${\widehat{\A}}$ on $\caH$}\label{sss:The action of A on H} The algebra ${\widehat{\A}}$ acts on the space $\caH$ as follows. 
If ${\mathbf f}$ is in the first summand in \eqref{eq:H 2-dim cas}, then $\fx$ acts on it by multiplication and $\fxi$ by $i\hbar \ddfx,$ the latter defined by
$$\ddfx (\exp(\frac{a\fx^2}{2i\hbar}f(\fx)))=\exp(\frac{a\fx^2}{2i\hbar}) (\ddfx+a\fx)f(\fx).$$
As for $F{\mathbf f},$ $\fx$  sends it to $-i\hbar F\ddfx {\mathbf f}$ and $\fxi$ sends it to $F\fx{\mathbf f}.$
\subsubsection{Some operators on $\cH$}\label{sss:ome operators on H 2 dim cas} 
{\em The operator $F:\cH\to\cH.$} Define for ${\bf f}(\fx)=\exp(\frac{a\fx^2}{2i\hbar})f(\fx)$
$$F: {\bf f}\mapsto F{\bf f}\mapsto i^{-1} {\bf f}(-\fx)$$ 

{\em The operator $\exp(\frac{a\fx^2}{2i\hbar}):\cH\to\cH.$}
1) $$\exp(\frac{a\fx^2}{2i\hbar}): \exp(\frac{c\fx^2}{2i\hbar}) f(\fx)\mapsto \exp(\frac{(a+c)\fx^2}{2i\hbar}) f(\fx)$$
for $c\in \bR;$

2) $$F\exp(\frac{c\fx^2}{2i\hbar}) f(\fx)\mapsto 
\frac{e^{\frac{-\pi i}{2}p(c)}}{{\sqrt{|c|}}} f(-i\hbar\ddfx+a\fx) \exp(\frac{(a-c^{-1})\fx^2}{2i\hbar})$$
for $c\neq 0;$

3) $$F\exp(\frac{c\fx^2}{2i\hbar}) f(\fx)\mapsto iFf(\fx-ai\hbar\ddfx)\frac{e^{{-\frac{\pi i}{2}}(p(c)+p(\frac{-c}{1-ac}))}}{\sqrt{|1-ac|}} \exp(\frac{c}{1-ac}\frac{\fx^2}{2i\hbar})
 $$
 for $c\neq a^{-1}.$
These maps preserve the equivalence relation and therefor define operators on $\cH$.
\subsubsection{The action of $\Mp(2,\R)$ on $\caH$}\label{sss:The action of G on H} The group $\Mp(2,\R)$ acts by the algebraic version of the metaplectic representation that we are going to describe next. 
\subsection{The metaplectic projective representations of ${\rm{SL}}_2(\bR)$}\label{sss:meta rep SL2}

Define the action of generators of ${\rm{SL}}_2(\bR)$ by exactly the same formula as the usual metaplectic representation
\begin{equation}\label{eq:Weil rep} T\colon \left [ \begin{array}{cc}
1&0\\
a&1\end{array}\right ] \mapsto \exp(\frac{a\fx^2}{2i\hbar});
\;
\left [ \begin{array}{cc}
0&1\\
-1&0\end{array}\right ] \mapsto F;
\end{equation}
\begin{equation}\label{eq:Weil rep 1} 
\left [ \begin{array}{cc}
b&0\\
0& b^{-1}\end{array}\right ] \mapsto T_b;\;(T_bf)(x)=\frac{1}{\sqrt{\det(b)}}f(b^{-1}x)
\end{equation}
The corresponding representation of ${\mathfrak{sl}}(2)$ is given by 
\begin{equation}\label{eq:sl 2 triple}
X_-=\frac{\fx^2}{2i\hbar};\, H=\frac{\fx\fxi}{i\hbar}=-\frac{\fx*\fxi}{i\hbar}-\frac{1}{2}; \, X_+=-\frac{\fxi^2}{2i\hbar}
\end{equation}
\subsubsection{The Bruhat relations}\label{sss:Bruhat}
The following are well known to be the defining relations of ${\rm{SL}}_2$ (together with the requirement that $a\mapsto \left [ \begin{array}{cc}
1&0\\
a&1 \end{array}\right ]$ is a morphism from the additive group and $b\mapsto \left [ \begin{array}{cc}
b&0\\
0&b^{-1} \end{array}\right ]$ is a morphism from the multiplicative group).
\begin{equation}\label{eq:Bruhat 0}
\left [ \begin{array}{cc}
0&1\\
-1&0 \end{array}\right ] \left [ \begin{array}{cc}
1 &0\\
a&1\end{array}\right ] \left [ \begin{array}{cc}
0&1\\
-1 &0 \end{array} \right ]^{-1}=\left [ \begin{array}{cc}
1&-a\\
0 &1 \end{array}\right ]
\end{equation}
\begin{equation}\label{eq:Bruhat 00}
\left [ \begin{array}{cc}
0&1\\
-1&0 \end{array}\right ] \left [ \begin{array}{cc}
b &0\\
0&b^{-1}\end{array}\right ] \left [ \begin{array}{cc}
0&1\\
-1 &0 \end{array} \right ]^{-1}=\left [ \begin{array}{cc}
b^{-1}&0\\
0&b \end{array}\right ]
\end{equation}

\begin{equation}\label{eq:Bruhat 11}
\left [ \begin{array}{cc}
b&0\\
0&b^{-1} \end{array}\right ] \left [ \begin{array}{cc}
1&0\\
a&1\end{array}\right ] \left [ \begin{array}{cc}
b&0\\
0 &b^{-1} \end{array}\right ]^{-1}=\left [ \begin{array}{cc}
1&0\\
b^{-2} a&1\end{array}\right ] 
\end{equation}
\begin{equation}\label{eq:Bruhat 1}
\left [ \begin{array}{cc}
1&0\\
a&1\end{array}\right ] \left [ \begin{array}{cc}
0&1\\
-1&0\end{array}\right ] \left [ \begin{array}{cc}
1&0\\
a^{-1}&1\end{array}\right ] = \left [ \begin{array}{cc}
a^{-1}&0\\
0&a\end{array}\right ]  \left [ \begin{array}{cc}
1&a\\
0&1\end{array}\right ] 
\end{equation}
for $a\neq 0.$
\begin{prop}\label{prop:weil meta exact} Formulas \eqref{eq:Weil rep} define a representation of ${\widetilde{\rm{SL}}}_2$ in which an element of 
$\pi_1({\rm{SL}}_2)$ 
acts by $e^{\frac{\pi i}{2}c}$ where $c$ is its image in $\pi_1(\Lambda)\isomoto \bZ.$
\end{prop} 
\begin{proof} All the Bruhat relations except \eqref{eq:Bruhat 1} are true for operators $T(g)$ defined in \eqref{eq:Weil rep}, whereas
\begin{lemma}\label{lemma:bruhat rep rel}
$$
T({\left [ \begin{array}{cc}
1&0\\
a&1\end{array}\right ] }) T({\left [ \begin{array}{cc}
0&1\\
-1&0\end{array}\right ]}) T({\left [ \begin{array}{cc}
1&0\\
a^{-1}&1\end{array}\right ] })= 
$$
$$
=\frac{\sqrt {|a|}}{\sqrt{a}} e^{{\frac{\pi i}{2}} p(a)} T({\left [ \begin{array}{cc}
a^{-1}&0\\
0&a\end{array}\right ]} ) T({\left [ \begin{array}{cc}
1&a\\
0&1\end{array}\right ] })
$$
\end{lemma}
\end{proof}
\subsection{The case of a general $n$}\label{general n}

Now define
\begin{equation}\label{eq:alg Weil}
\caH=\bigoplus_{I\subset \{1,\ldots, n\}}\bigoplus _a F_{I,J} \caH^{\fx}/\sim
\end{equation}
where $a$ runs through all symmetric $n\times n$ real matrices and the equivalence relation is defined as follows. For every subset $K$ of $\{1,2,\ldots,n\},$ define 
\begin{equation}\label{eq:FK on H}
F_K:\bigoplus_a F_I\cH^{\fx}\to \bigoplus_a F_{I\triangle K}\cH^{\fx}
\end{equation}
(where $\triangle$ stand for the symmetric difference) as follows: if $L$ is the complement of $I\cap K,$ then
\begin{equation}
(F_{K}F_I{\mathbf f})(\fx_{I\cap K}, \fx_L)=i^{-|I\cap K|}F_{I\triangle K}{\bf f}(-\fx_{I\cap K}, \fx_L)
\end{equation}
Let $J$ be the complement of $I$.
\begin{equation}\label{eq:element of H}
{\bf f}(\fx_I ,\fx_J)=\exp({\frac{a\fx_I^2+b\fx_I\fx_J+c\fx_J^2}{2i\hbar}})f(\fx_{I}, \fx_J)
\end{equation}
such that $a_I$ is a nondegenerate symmetric matrix. Then
\begin{equation}\label{eq:equiv fourier general n}
F_KF_{I}{\mathbf f}\sim F_K\frac{\exp(-\frac{\pi i}{2}p(a))}{\sqrt{\det(|a|)}} f(i\hbar \frac{\partial}{\partial \fx_I})\exp(\frac{-a^{-1}(\fx_I+b\fx_J)^2}{2i\hbar})
\end{equation}
for all $K.$
\subsubsection{Operators on $\cH$}\label{sss:ops on H gen} Clearly, the operators $F_K$ \eqref{eq:FK on H} preserve the equivalence relation and therefore descend to $\cH.$
 \subsubsection{The action of ${\widehat{\A}}$ on $\caH$}\label{sss:The action of A on H gen case} The algebra ${\widehat{\A}}$ acts on the space $\caH$ as follows. 
On the summand $F_I \caH^{\fx}_a$, 
\begin{equation}\label{eq:acts alg n gen} 
\fx_j F_I {\mathbf f}=-F_I i\hbar \frac{\partial}{\partial \fx_j}{\mathbf f}, j\in I;\; \fx_j F_I {\mathbf f}=F_I \fx_j {\mathbf f},j\notin I;
\end{equation}
\begin{equation}\label{eq:acts alg n gen 1} 
\fxi_j F_I {\mathbf f}=F_I i\hbar \frac{\partial}{\partial \fx_j}{\mathbf f}, j\notin I;\; \fxi_j F_I {\mathbf f}=F_I \fx_j {\mathbf f},j\in I.
\end{equation}

\subsubsection{The action of $\Mp(2n)$ on $\caH$}\label{sss:The action of G on H 1} This action is exactly as described in \ref{sss:The action of G on H}. In particular, $\Mp(2n,\R)$ acts by the metaplectic representation as in\ref{sss:meta rep SL2}:
\begin{equation}\label{eq:Weil rep 11} T\colon \left [ \begin{array}{cc}
1&0\\
a&1\end{array}\right ] \mapsto \exp(\frac{a\fx^2}{2i\hbar});
\;
\left [ \begin{array}{cc}
0&1\\
-1&0\end{array}\right ] \mapsto F;
\end{equation}
more generally, let ${\mathbf F}_I$ be the matrix that is the direct sum of $\left [ \begin{array}{cc}
0&1\\
-1&0\end{array}\right ]$ in coordinates $\fx_I,\fxi_I$ and the identity matrix in the rest of the Darboux coordinates maps to $F_I;$
\begin{equation}\label{eq:Weil rep 12}
\left [ \begin{array}{cc}
b&0\\
0&^tb^{-1}\end{array}\right ] \mapsto T_b;\;(T_bf)(x)=\frac{1}{\sqrt{\det(b)}}f(b^{-1}x)
\end{equation}
\begin{remark}\label{rmk:orbit weil}
The construction of $\caH$ mimics very closely the construction of the orbit of $1$ under the action of $\Mp(2n)\ltimes \C[\fx, \fxi]$ on the space of distributions via differential operators and the standard metaplectic representation. 
\end{remark}
\begin{lemma}\label{lemma:weil and lagrangian}
Assign to $F_I \exp({\frac{ a\fx^2}{2i\hbar}})f(\fx)\in \cH$
the Lagrangian subspace
${\mathbf F}_I (\{\fxi=a\fx\})$ where ${\mathbf F}_I$ was defined after \eqref{eq:Weil rep 11}.
This is a well-defined map $\cH \to \Lambda (n)$ where $\Lambda(n)$ is the Grassmannian of Lagrangian subspaces in $\R^{2n}.$ The space $\caH$ is identified with the space of finitely supported sections of a $\tG$-equivariant vector bundle on $\Lambda(n).$
\end{lemma}
The Lagrangian Grassmannian is a homogeneous space of $\tG$ via the projection $\tG\to \Mp\to \Sp.$ In fact,
$$\Lambda(n)\isomoto \tG/\tP.$$
\begin{lemma}
The lines $\C F_I \exp (\frac{a\fx^2}{2i\hbar})$ where $a$ runs through real symmetric $n\times n$ matrices and $I$ through subsets of $\{1,\ldots,n\}$ form a line subbundle of $\caH$ which is isomorphic to the bundle on $\Lambda(n)$ determined by the \v{C}ech one-cohomology class $(-1)^{\mu _L}$ where $\mu_L$ is the generator of $H^1(\Lambda(n), \Z)$ (the Maslov class).
\end{lemma}

\begin{lemma} \label{lemma:H as a(G,g,A)-module}
The actions described in \ref{sss:The action of A on H gen case} and \ref{sss:The action of G on H 1} turn $\caH$ into an $\cA$-module.
\end{lemma}
\subsection{The algebraic metaplectic representation as an induced module}\label{ss:meta as ind} 
\begin{prop}\label{prop:meta as ind}
$$\cH\isomoto {\widehat{\cV}}=\cA{\widehat{\otimes}} _{\cB} \ffbV_{\bK}$$
\end{prop}
({\em cf.} \ref{sss:A0-mod V}).
\section{Appendix. Twisted bundles and groupoids}\label{s:appendix 1}
\subsection{Charts and cocycles}\label{ss:charts and cocycles}
Suppose we have a manifold  $X$ with two sheaves of groups $\uC\subset \uG$ where $\uC$ is constant and central in $\uG.$ Consider a class $c\in H^2(X,\uC).$ A $\uG$-bundle on $X$ twisted by $c$ is given by an equivalence class of $g_{ij} \in \uG(U_i\cap U_j)$ for an open cover $X=\cucup U_i$ such that 
\begin{equation}\label{eq:tw coc}
g_{ij}g_{jk}=c_{ijk} g_{ik}
\end{equation}
where $c_{ijk}$ is a \v{C}ech cocycle representing $c.$ Two data $g_{ij}$ and $g'_{ij}$ are equivalent if 
\begin{equation}\label{eq:tw coc eq}
g_{ij}=h_i g'_{ij} h_j^{-1} b_{ij}
\end{equation}
for some common refinement of the two covers, where $h_i \in \uG(U_i)$ and $b_{ij}\in \uC(U_i \cap U_j).$ Note that this definition makes all $\uC$-bundles equivalent. 

By a {\em chart} we mean a map $T\to X$ where $T$ is any topological space. {\em A good collection of charts} on $X$ is a collection of charts $T\to X,$ $T\in \cT,$ such that 
for every $T_0, \ldots, T_p$ in $\cT,$ 
every one-cocycle on $T_0 \times _X \ldots \times _X T_p$ with values in the pullback of $\uG,$ and every one- or two-cocycle with values in the pullback of $\uC,$ can be trivialized.
\begin{lemma}
For any good collection of charts and any twisted bundle, one can define 
\begin{equation}\label{eq:charts cocs}
c_{TT'T''} \in \uC(T\times _X T' \times_X T'');  \; g_{TT'} \in \uG (T\times _X T')
\end{equation}
satisfying
\begin{equation}\label{eq:tw cocs 1}
c_{ TT'T''}c_{TT''T'''}=c_{TT'T'''}c_{T'T''T'''}
\end{equation}
and
\begin{equation}\label{eq:tw coc 4}
g_{TT'}g_{T'T''}=c_{TT'T''}g_{TT''}
\end{equation}
in such a way that, if $T_i$ are a good open cover, then $c_{T_iT_jT_k}$ is cohomologous to $c_{ijk}$ and $g_{T_iT_j}$ is equivalent to $g_{ij}.$ The choice is unique up to equivaklence in the following sense:
\begin{equation}\label {tw coc 2 }
c_{TT'T''}=c'_{TT'T''}b_{TT'}b_{T'T''}b_{TT''}^{-1}; \; g_{TT'}=h_Tg'_{TT'} h_{T'}^{-1} b_{TT'}
\end{equation}
for some $b_{TT'}\in \uC(T\times T')$ and $h_T\in \uG(T).$
\end{lemma}
\begin{proof} Consider inverse images on charts $T$ of an open cover $\{U_i\}$ of $X.$ Let
$$c_{ijk}=\alpha_{ij}(T)\alpha_{jk}(T)\alpha_{ik}(T)^{-1}$$
be a trivialization of $c$ on $T$. Choose trivializations
$$g_{ij}\alpha_{ij}(T)^{-1}=h_i(T)h_j(T)^{-1}$$
on $T$ and 
$$\alpha_{ij}(T)\alpha_{ij}(T')=\beta_i(T,T')\beta_j(T,T')^{-1}$$
where $\alpha_{ij}, \,\beta_{ij}$ are sections of $\uC$ and $h_i$ are sections of $\uG.$
Put
\begin{equation} \label{eq:tw coc 6}
c_{TT'T''}=\beta_i(T,T') \beta_i (T',T'') \beta_i (T,T'')^{-1}
\end{equation}
and 
\begin{equation} \label{eq:tw coc 7}
g_{TT'}=h_i(T)^{-1} h_i(T') \beta_i(T,T')
\end{equation}
The relations above show that these do not depend on $i.$
\end{proof}
\subsection{The groupoid of a twisted $G$-bundle}\label{ss:grpd of a twisted bdl} 

Let $G$ be a Lie group and $\uG$ the sheaf of smooth $G$-valued functions. Let $C$ be a central subgroup of $G$ and $\uC$ the sheaf of locally constant $C$-valued functions. Consider a $\uC$-valued two-cohomology class represented by a cocycle $c_{ijk}$ and twisted $G$-bundle represented by a $\uG$-valued cochain $g_{jk}.$

Define a groupoid on $X$ as follows.

For $x_0$ and $x_1$ in $X,$ define the set $\tG_{x_0,x_1}.$ Let $\gamma: [0,1]\to X$ be a smooth map. View it as a chart that we denote by $T.$ An element of $\tG_{x_0,x_1}$ is a class of an element $g_T\in G$ with respect to the following equivalence relation. Consider two charts $T$ and $T'$ representing two smooth maps $\gamma,\,\gamma': [0,1]\to X$ and a homotopy $\sigma: [0,1]^2\to X$ such that $\sigma(0, s)=x_0,\, \sigma(s,t)=x_1,\, \sigma(t,0)=\gamma(t),$ and $\sigma(t,1)=\gamma'(t).$ We will view $\sigma$ as a chart $S.$ We call $S$ {\em a homotopy between $S$ and $S'.$} Now generate the equivalence relation by the following.
\begin{equation}\label{eq:equrela naiva}
g_T \sim (g_{TT'}c_{TT'S}^{-1} )(x_0) g_{T'} (g_{TT'}c_{TT'S}^{-1} )(x_1)^{-1} 
\end{equation}


\begin{lemma}\label{lemma:strong ekviv} Let $S$ be a homotopy between $T$ and $T'$, $S'$ a homotopy between $T'$ and $T'',$ and $S''$ a homotopy between $T$ and $T''.$ If we denote the right hand side of \eqref{eq:equrela naiva} by $a(S)g_T,$ then $a(S)a(S')=\langle c, [\Sigma] \rangle a(S'')$ where $\Sigma$ is the sphere formed by $S,$ $S',$ and $S''.$
\end{lemma}
\begin{proof} We have
$$a(S)a(S')g_T=$$
$$g_{TT'}g_{T'T''}c_{TT'S}^{-1}c_{T'T''S'}^{-1}(x_0)  g_{T''} (g_{TT'}g_{T'T''}c_{TT'S}^{-1}c_{T'T''S'}^{-1}(x_1))^{-1} $$
The right hand side is equal to
$$(g_{TT''}c_{TT'T''}c_{TT'S}^{-1}c_{T'T''S'}^{-1})(x_0) g_{T''} (g_{TT''}c_{TT'T''}c_{TT'S}^{-1}c_{T'T''S'}^{-1})(x_1)^{-1} ;$$
therefore
$$a(S)a(S')= \frac{c_{TT'T''}c_{TT''S''}}{c_{TT'S}c_{T'T''S'}}(x_0) (\frac{c_{TT'T''}c_{TT''S''}}{c_{TT'S}c_{T'T''S'}}(x_1))^{-1}a(S'')$$

Applying the acyclicity condition to the quadruple of charts $TT'T''S,$ we get 
$$\frac{c_{TT'T''}c_{TT''S''}}{c_{TT'S}c_{T'T''S'}}=\frac{c_{TT''S''}c_{T'T''S}}{c_{TT''S}c_{T'T''S'}}$$
Applying the same condition to $TT''SS''$ and then to $SS'S''T'',$ we replace the right hand side with 
$$\frac{c_{TSS''}c_{T''SS'}}{c_{T''S'S''}c_{T'SS'}}=\frac{c_{SS'S''}c_{TSS''}}{c_{T'SS'}c_{T''S'S''}}.$$
But $T\times_X S\times _X S''=T, \, T'\times_X S\times _X S'=T',$ and $T''\times_X S'\times _X S''=T''.$ Therefore the values of $c_{TSS''},$ {\em etc.} at $x_0$ and $x_1$ are the same. Therefore
$$a(S)a(S')= c_{SS'S''}(x_0)  c_{SS'S''}(x_1) ^{-1} a(S'')$$
But 
$$ c_{SS'S''}(x_0)  c_{SS'S''}(x_1) ^{-1} = \langle c, [\Sigma] \rangle $$
for any two-cocycle $c$. 
To see this, note that the left hand side is $1$ for any coboundary $c.$ On the other hand, if we enlarge $S,\,S',\, S''$ a little bit to make them an open cover of $\Sigma,$ take an element $a$ of $C,$ and define $c_{SS'S''}(x_0)=a, \,c_{SS'S''}(x_1)=1,$ the result will be $a=\langle c,[\Sigma] \rangle$.
\end{proof}
\begin{cor}\label{cor:sekv grupoids}
There is an epimorphism 
\begin{equation}\label{eq:short sekv of grupoids 1} 
\tG_{x_0,x_1}\to \pi_1(x_0,x_1)
\end{equation}
When $x_0=x_1=x,$ the kernel of this epimorphism is isomorphic to $G/\langle c, \pi_2(X)\rangle.$
\end{cor}
\subsubsection{Example: the holonomy groupoid of a vector bundle}\label{sss:hologruvekt} Let $E$ be a real oriented vector bundle of rank $N.$ Let $G={\rm{SO}}_N(\bR)$ and ${\widetilde G}={\rm{Spin}}_N (\bR)$ its universal cover. Reduce the structure group of $E$ to $G$ using a Riemannian metric. Let
$\tIsom (E)_{x_0, x_1}$ be the set of equivalence classes of data $(\gamma, u_t)$ where $\gamma: [0,1]\to X$ is a smooth map, $\gamma(0)=x_0,$ $\gamma(1)=x_1,$ and $u_t : E_{\gamma(t)}\isomoto E_{\gamma(0)}$ a metric-preserving linear isomorphism smoothly depending on $t$ and satisfying $u_0=\id.$ An equivalence between $(\gamma, u_t)$ and $(\gamma', u_t')$ is a smooth map $\sigma: [0,1]\times [0,1]\to X$ such that $\sigma(0,s)=x_0,\,\sigma(1,s)=x_1,$ $\sigma(t,0)=\gamma(t),\,\sigma(t,1)=\gamma'(t),$ and a linear metric-preserving isomorphism $v_{t,s}: E_{\sigma(t,s)}\isomoto E_{x_0}$ smooth in $(t,s),$ such that $v_{0,s}=\id,$ $v_{t,0}=u_t,$ $v_{t,1}=u'_t,$ and $v_{1,s}=u_1=u'_1.$

Lift the transition isomorphisms $g^E _{ij}$ of $E$ to some ${\widetilde{g}}_{ij}.$ Put $c_{ijk}={\widetilde{g}}_{ij} {\widetilde{g}}_{jk} {\widetilde{g}}_{ik}^{-1}.$ This cocycle represents the second Stiefel-Whitney class $w_2(E).$ Note that the groupoid $\tIsom(E)$ is isomorphic to the groupoid $\tG'$ constructed from the twisted bundle defined by ${\widetilde{g}}_{ij}, c_{ijk}.$ In fact, note that for the charts $T$ and $S$ defined by maps $\gamma$ and $\sigma,$ there is a natural lifting $\tg _{TS}$ of $g_{TS}.$ Namely, $\tg_{TS} (\gamma(t))$ is the class of the path $g_{TS} (\gamma(\tau)),\,0\leq \tau\leq t.$ Similarly with $\tg_{ST'}.$ Identify with $\widetilde G$ the set of equivanence classes of $(\gamma, u_t)$ with fixed $\gamma$ (and with $\sigma(t,s)=\gamma(t)$ in the definition of the equivalence). Now, given an equivalence $\sigma,v$ between $\gamma, u$ and $\gamma',u',$ $g_T\in {\widetilde G}$ gets identified with $\tg_{TS} \tg_{ST'}=\tg_{TT'} c_{TST'}.$ 
\begin{cor}\label{cor:sekv grupoids igen}
There is an epimorphism 
\begin{equation}\label{eq:short sekv of grupoids 1 igen} 
\tIsom(E)_{x_0,x_1}\to \pi_1(x_0,x_1)
\end{equation}
and every preimage is a homogeneous space ${\rm{Spin}}(N,\bR)/\langle w_2 (E), \pi_2(X)\rangle.$
(We identify ${\mathbb Z}/2$ with the center of ${\rm{Spin}}(N,\bR)$).
\end{cor}

\subsubsection{Connections on twisted bundles}\label{sss:conns on tw bdles} As in \ref{ss:grpd of a twisted bdl}, let $G$ be a simply-connected (pro) Lie group and $\uG$ the sheaf of smooth $G$-valued functions. Let $C$ be a central subgroup of $G$ and $\uC$ the sheaf of smooth $C$-valued functions. In addition, fix some algebra $\cA$ on which $G$ acts by automorphisms. Consider a twisted bundle defined by the data $(g_{ij}, c_{ijk}).$ A connection in this twisted bundle is a collection of $\cA$-valued forms on $U_i$ such that 
$$\Ad_{g_{ij}}(d+A_j)=d+A_i$$ 
on every $U_{ij}.$ Here $\Ad_g(d)=-dg\cdot g^{-1}.$ Note that, because $c_{ijk}$ are locally constant and central, $\Ad_{g_{ij}}\Ad_{g_{jk}}(d+A_k)=\Ad_{g_{ik}}(d+A_k),$ so the conditions above are consistent on $U_{ijk}.$ The corvature $R=dA_i+A_i^2$ is a well-defined $\cA$-valued two-form.
\subsubsection{The flat connection up to inner derivations}\label{sss:conn on tG}  Here we will construct a flat connection up to inner derivations on the associated bundle of algebras $\cA$ compatible with the action of the groupoid $\tG$ of a twisted bundle ({\em cf.} \ref{ss:grpd of a twisted bdl}) . We will start from a flat connection on the twisted bundle itself. 

First define special coordinate charts on $\tG$ as follows. Fix: 
\begin{itemize}
\item two open charts $U_0$ and $U_1$ of $X;$ 
\item two points $x_0^*\in U_0$ and $x_1^*\in U_1;$ 
\item a path $\gamma$ from $x_0^*$ to $x_1^*$ in $X$; 
\item smooth maps $\tau _0: [0,1]\times U_0\to U_0$  and $\tau_1: [0,1] \times U_1\to U_1,$ $\tau_0(0,x_0)=x_0,$ $\tau_0(1,x_0)=x^*_0,$ $\tau_1(0,x_1)=x_1,$ $\tau_1(1,x_1)=x^*_1.$
\end{itemize}
For every $x_0\in U_0$ and $x_1\in U_1,$ we will denote the path $t\mapsto \tau_0(t,x_0)$ by $\tau_{x_0}$ and the path $t\mapsto \tau_1(t,x_1)$ by $\tau_{x_1}.$ For the data as above, we construct a chart $T$ in $\tG$ as a map
$$U_0\times U_1\to \tG;\; (x_0,x_1)\mapsto \tau_{x_0} \circ \gamma \circ \tau_{x_1}: x_0\to x_1$$
(the composition of paths). 

Now consider a flat connection in our twisted bundle. In a local trivialization, on any open chart $W,$ we write $\nabla_\cV=d+A_W.$ We can identify a local section of $\tG$ on $T$ with a ${\widetilde{G}}$-valued function $g_T(x_0,x_1)$ on $U_0\times U_1.$
\begin{definition}\label{dfn:conn grpd tG}
$$\alpha(g_T)=-d g_T\cdot g_T^{-1}-A_0+\Ad_{g_T}(A_1)$$

where $A_0=\pi_0 ^*(A_{U_0})$ and $A_1=\pi_1 ^*(A_{U_1});$ 
$$R=dA_0+A_0^2.$$
\end{definition}
\begin{lemma}\label{lemma:conn on tG transition}
The above formulas define a flat connection up to inner derivations on the associated bundle of algebras $\cA$ compatible with the action of $\tG.$
\end{lemma}
\section{Appendix. Modules associated to Lagrangian submanifolds and Lagrangian distributions}\label{s:LagrDistr} For any Lagrangian submanifold $L$ of a symplectic manifold $M$ with a given $\Sp^4$ structure we constructed a bundle of modules ${\widehat{\widehat{\bV}}}_L$ with a flat connection $\nabla_\bV$ ({\em cf.} \ref{sss:sheaf BL}). This is a bundle of $\fbA_M$-modules, and the connections $\nabla_\bV$ and $\nabla_\bA$ are compatible. In particular, denote by $\bA_M$ the sheaf of algebras of horizontal sections of $\nabla_\bA$ and by $\bV_L$ the sheaf of horizontal sections of $\nabla_\bV.$ Then $\bV_L$ is a sheaf of $\bA_M$-modules. 

Now apply the same construction to $L$ but instead of $M$ take a tubular neighborhood of $L$ and identify it with the tubular neighborhood of $L$ in $T^*L$ by Darboux-Weinstein theorem. Use the $\Sp^4$ structure provided by the Lagrangian polarization by fibers of $T^*L$ ({\em cf.} Lemma \ref{lemma:coho cocy sp4}). We get another $\bA_M$-module that we denote by $\bV^{(0)}_L.$
\begin{lemma} $\bV^{(0)}_L\isomoto |\Omega|_L^{\frac{1}{2}}$ where $|\Omega|^{\frac{1}{2}}$ stands for the bundle of half-densities. $\bV_L$ is isomorphic to $\bV^{(0)}_L$ twisted by the $\{\pm 1,\pm i\}$-valued Maslov class of $L.$ 
\end{lemma}
We denote this class by $\exp(\frac{\pi i}{2} \mu(L)).$ Note that $\mu(L)$ can be chosen as a $\bZ$-valued cocycle only if $2c_1(M)=0.$ 
\subsection{The asymptotic construction of H\"{o}rmander and Maslov}\label{ss:asym Hor Maslov} As we have seen in \ref{sss:sheaf BL}, the oscillatory module $\cV^\bullet_L$ is induced from the module of forms with coefficients in ${\widehat{\widehat{\bV}}}.$ But it is the twisted version of the latter module that serves as an asymptotic version of the classical construction of Lagrangian distributions with wave front $L.$

Put
\begin{equation}\label{eq:curW 00}
\bV_{L, \bK} =\bK\fotimes \bV_L= \{\sum _{k=0}^\infty  e^{\frac{1}{i\hbar}c_k} v_k |v_k \in  \bV_L ;\, c_k\in \bR;\, c_k \to \infty\}  
\end{equation} 
\begin{definition}\label{dfn:VLKalfa}
Assume $M=T^*X.$ Let $\bV_{L,\bK}^{\eta}$ be the twist of the sheaf $ \bV_{L, \bK}$ by the \v{C}ech cohomology class $\exp(-\frac{1}{i\hbar}\eta) \in H^1(L,\exp(\frac{1}{i\hbar}\bR))$ where $\eta$ is the class of the form $\xi dx|L.$
\end{definition}
Let $X=\cucup U_\alpha$ is a small open cover. Let $L=\cucup W_\gamma$ be a refinement of the cover $L=\cucup (T^*U_\alpha\cap L).$ In particular, a choice is made of $\gamma\mapsto\alpha= \alpha(\gamma)$ such that $W_\gamma\subset T^*U_\alpha\cap L.$
\subsubsection{Quantization procedure}\label{sss:quantization proced} First let us review our deformation quantization picture in the case $M=T^*X.$ First, we have the sheaf of algebras $\bA_{T^*X}.$ It can be described by products $*_\alpha$ on $C^\infty (T^*U_\alpha)[[\hbar]]$  
\begin{equation}\label{eq:star fun A}
a*_\alpha b=\sum_{k=0}^\infty (i\hbar)^k P_{\alpha,k} (a,b)
\end{equation}
and by transition functions
\begin{equation}\label{eq:trans fun A}
G_{\alpha\beta}(a)=\sum_{k=0}^\infty (i\hbar)^k T_{\alpha \beta, k}(a)
\end{equation}
where $P_{\alpha,k}$ are bilinear bidifferential expressions, $T_{\alpha\beta,k}$ are differential operators, $P_{\alpha,0}(f,g)=fg,$ $ P_{\alpha,1}(f,g)=\frac{1}{2}\{f,g\},$ and $T_{\alpha\beta,0}(f)=f.$ One has $G_{\alpha\beta}(a*_\beta b)=G_{\alpha\beta}(a)*_\alpha G_{\alpha\beta}(b).$ Actually $G_{\alpha\beta}$ can be made the identity automorphisms, but this is not necessarily the most natural choice.

The sheaf of modules $\bV^\eta_L$ is described by the action 
\begin{equation}\label{eq:star fun V}
a*_\gamma f=\sum_{k=0}^\infty (i\hbar)^k Q_{\gamma,k} (a,f)
\end{equation}
where $f\in |\Omega|^{\frac{1}{2}}(W_\gamma)$ and $a\in C^\infty (U_{\alpha(\gamma)}),$ and by the transition functions
\begin{equation}\label{eq:trans fun V}
H_{\gamma\delta}(f)=\exp (-\frac{1}{i\hbar}\eta_{\gamma\delta}) \sum_{k=0}^\infty (i\hbar)^k S_{\gamma\delta, k}(f)
\end{equation}
where $Q_{\gamma,k}$ are bidifferential and $S_{\gamma\delta,k}$ are differential. Moreover, $Q_{\gamma\delta,0}(a,f)=af$ and 
\begin{equation}\label{eq:TWIST}
S_{\gamma\delta,0}(f)=\exp(\frac{\pi i}{2}\mu_{\gamma\delta}(L)) f.
\end{equation} 
One has
$$a*_\gamma(b*_\gamma f)=(a*_{\alpha(\gamma)} b)*_\gamma f$$
and
$$S_{\gamma\delta}(a*_\delta f)=T_{\alpha(\gamma)\alpha(\delta)}(a)*_\gamma S_{\gamma\delta}(f)$$
Again, all higher $S_{\gamma\delta, k}$ can be made zero, but this is not the most natural choice.

Let $C^\infty_{\rm{poly}}$ denote functions on $T^*X$ that are polynomial on fibers. {\em A quantization procedure} is the following.

1) For any $\alpha,$ a map 
\begin{equation}\label{eq:quantiza 0}
{\rm{Op}}^\alpha_\hbar: C^\infty_{\rm{poly}} (T^*(U_\alpha) )\to \cD(U_\alpha, |\Omega|_X^{\frac{1}{2}}) 
\end{equation}
such that
$${\rm{Op}}^\alpha_\hbar (a) {\rm{Op}}^\alpha_\hbar (b)= {\rm{Op}}^\alpha_\hbar (a*_\alpha b)$$
and
$${\rm{Op}}^\alpha_\hbar (G_{\alpha\beta }(a))={\rm{Op}}^\beta (a)$$
on $U_\alpha\cap U_\beta.$ 
(We can ask for exact equalities, not for asymptotic equalities like we use below, when $a$ and $b$ are polynomial).

2) A map 
\begin{equation}\label{eq:quantiza 1}
u^\gamma _\hbar: |\Omega|^{\frac{1}{2}}_{c} (W_\gamma) \to |\Omega|^{\frac{1}{2}}_{c} ( U_{\alpha(\gamma)})
\end{equation}
for all $\hbar>0,$ such that
$$\Op^{\alpha(\gamma)}_\hbar (a) u^\gamma_\hbar (f) - \sum_{k=0}^N  (i\hbar)^k u^\gamma_\hbar (Q_{\gamma,k} (a,f))=O(h^{N+1})$$
and
$$u^\gamma (f)-\sum _{k=0}^N (i\hbar)^k u^\delta_\hbar (S_{\gamma,\delta,k}(f)) = O(h^{N+1})$$
for all $N.$

Let us recall how a quantization procedure is carried out. For every $\gamma$ choose a {\em phase function} for $L\cap W_\gamma$ as follows. Let $\theta=(\theta_1,\ldots,\theta_k$ be a coordinate system on $\bR^k.$ Choose a coordinate system $x=(x_1,\ldots,x_n)$ on $U_{\alpha(\gamma)}$. Choose {\em a phase function} for $L\cap W_\gamma$, {\em i.e.} a function $\varphi(x,\theta)$ such that 
\begin{equation}\label{eq:Hormander}
L\cap W_\gamma =\{(\xi,x)|\exists \theta\; {\rm{ such}}\; {\rm{that}} \;\xi=\varphi_x (x,\theta)\; {\rm{and}}\; \varphi_\theta(x,\theta)=0\}
\end{equation}
Here $\varphi_x$ and $\varphi_\theta$ stand for partial derivatives. We assume that the $n\times (n+k)$ matrix $(\varphi_{xx}, \varphi_{x\theta})$ is nondegenerate. 
\begin{example}\label{ex:phase fn dim 1}
Let $n=1.$ Assume that $L=\{\xi=\varphi '(x)\}.$ Then we can choose $k=0$ and $\varphi=\varphi(x).$ Now let $L=\{x=-\psi'(\xi)\}.$ Then we can take $k=1$ and $\varphi(x,\theta)=x\theta+\psi(\theta).$ 
\end{example}
\begin{example}\label{ex:phase fn dim n}
More generally, one can always subdivide the coordinates into two groups and write $x=(x_1,x_2);$ $\xi=(\xi_1,\xi_2)$ so that $L\cap W_\gamma$ will be of the form
\begin{equation}\label{eq: Lagr gen form loc}
\xi_1=F_{x_1}(x_1,\xi_2);\; x_2=-F_{\xi_1}(x_1,\xi_2)
\end{equation}
In this case one can take $\varphi(x_1,x_2,\theta)=x_2\theta+F(x_1,\theta).$
\end{example} 
Note that the condition that the matrix of second derivatives is nondegenerate means that $\theta$ in \eqref{eq:Hormander} is uniqueand therefore $L\cap W_\gamma$ can be identified with $\{(x,\theta)|\varphi_\theta (x,\theta)=0\}.$  (To do that, one may need to pass to a finer open cover). Moreover, we can choose $n$ out of $n+k$ coordinates $x,\theta$ so that they will be coordinates on $\{\varphi_\theta =0.$ Namely, we can take any $n$ coordinates such that the corresponding square submatrix of $(\varphi_{xx}, \varphi_{x\theta})$ is nondegenerate. Denote these coordinates by $z$ and the other $k$ coordinates by $\zeta.$ Choose a procedure for extending functions $f(z)$ to functions on $\{(x,\theta)\}.$ Namely, extend $f(z)$ to $f(z)\rho(z')$ where $\rho$ is a function with small support near zero and $\rho(z')=0.$ 

Given a phase function and a compactly supported half-form $f=f(z)|dz|^{\frac{1}{2}},$ define $u^\gamma_\hbar(f)$ as follows. Denote by $f(x,\theta)$ the extension of $f(z)$ as above. Then define
\begin{equation}\label{eq:Hormander formula}
u_\hbar (f) = \frac{e^{\frac{-\pi i k}{4}}}{(2\pi\hbar)^{\frac{k}{2}}} \int e^{-\frac{\varphi(x,\theta)}{i\hbar}} f(x,\theta) d\theta |dx|^{\frac{1}{2}}
\end{equation} 
For the sake of completeness let us outline the proof of the fact that this is indeed a quantization procedure as described above (it is contained essentially in \cite{GS} and \cite{H}, as well as in \cite{NT3}).

First, as proven in \cite{H}, any two local phase functions differ by a coordinate change
$$
\varphi(x,\theta)\mapsto \varphi(g(x), h(x,\theta))
$$
followed by 
$$
\varphi(x,\theta)\mapsto \varphi(x,\theta)+\theta_1^2
$$
Here $\theta_1$ is an extra variable, possibly mutlidimensional, and $\theta_1^2$ stands for the sum of squares of variables. So we can assume that our local phase functions are as in Example \ref{ex:phase fn dim n}, possibly with $\theta_1^2$ added. We have two choices of subdivision $x=(x_1,x_2).$ Namely, for $W_\gamma$ we will have 
$$x^\gamma _1=(x_1,x_2);\;x^\gamma_2=(x_3,x_4);$$ 
for 
$W_\delta,$
$$\;x^\delta_1=(x_1,x_3);\;x^\delta_2=(x_2,x_4).$$
Let $F_\gamma(x_1,x_2,\xi_3,\xi_4)$ and $F_\delta(x_1,x_3,\xi_2,\xi_4)$ be functions as in Example \ref{ex:phase fn dim n}. Let us look for functions $f_\gamma$ and $f_\delta$ such that \eqref{eq:Hormander formula} will give the same answer for the charts $W_\gamma$ and $W_\delta.$ 
\begin{equation}\label{eq:Hormander formula compar}
\frac{e^{-\frac{\pi i}{4} (k_3+k_4)}}{(2\pi\hbar)^{\frac{k_3+k_4}{2}}} \int e^{-\frac {1}{i\hbar}(x_3\xi_3+x_4\xi_4 +F_\gamma (x_1,x_2,\xi_3,\xi_4)} f_\gamma (x_1,x_2,\xi_3,\xi_4)d\xi_3d\xi_4=
\end{equation} 
$$
=\frac{e^{-\frac{\pi i}{4} (k_2+k_4)}}{(2\pi\hbar)^{\frac{k_2+k_4}{2}}} \int e^{-\frac {1}{i\hbar}(x_2\xi_2+x_4\xi_4 +F_\delta (x_1,x_3,\xi_2,\xi_4)} f_\delta (x_1,x_3,\xi_2,\xi_4)d\xi_2d\xi_4
$$
Applying the inverse Fourier thansform we get 
\begin{equation}\label{eq:compare gamma delta}
e^{-F_\gamma}f_\gamma = \frac{e^{-\frac{\pi i}{4} (k_2-k_3)}}{(2\pi\hbar)^{\frac{k_2+k_3}{2}}} \int e^{\frac {1}{i\hbar}(-x_2\xi_2+x_3\xi_3-F_\delta})f_\delta d\xi_2 dx_3
\end{equation}
Compute the right hand side by the stationary phase method. The critical points satisfy
\begin{equation}\label{eq:crit pts}
x_2=-\frac{\partial F_\delta}{\partial \xi_2};\; \xi_3=\frac{\partial F_{\delta}}{\partial x_3}
\end{equation}
In other words, the critical point $(\xi_2,x_3)$ is such that $(x_1,x_2,\xi_1,\xi_2)$ is in $L.$ 
\begin{equation}\label{eq:compare again}
f_\gamma =  \epsilon_{\gamma\delta} \exp(\frac{1}{i\hbar} ((x_3\xi_3-F_\delta)-(x_2\xi_2-F_\gamma))) {\rm{mod}}\hbar
\end{equation}
or
\begin{equation}\label{eq:compare again 2}
f_\gamma =  \epsilon_{\gamma\delta} \exp(\frac{1}{i\hbar} (\varphi_\delta-\varphi_\gamma)) {\rm{mod}}\hbar
\end{equation}
Here
\begin{equation}\label{eq:epsilon g d}
\epsilon_{\gamma\delta} = e^{-\frac{\pi i}{4} (k_2-k_3)} e ^{-\frac{\pi i}{4} (n_-(\gamma,\delta)-n_+(\gamma,\delta))}
\end{equation}
where $n_-(\gamma,\delta),$ resp. $n_+(\gamma,\delta),$ is the number of negative, resp. positive, eigenvalues of the matrix of second derivatives of $F_\delta$ with respect to variables $\xi_2$ and $x_3.$ We can re-write \eqref{eq:epsilon g d} as
\begin{equation}\label{eq:eps correct}
\epsilon_{\gamma\delta}=\exp{\frac{\pi i}{2}(n_+-k_2)}
\end{equation}
where, as above, $n_+$ is the number of positive eigenvalues of the matrix of second derivatives of $F_\delta$ in variables $x_2,\xi_3.$
\begin{example}\label{ex:eps gade } 
Let $F_\gamma(x)=\varphi(x)$ and $F_\delta (x,\theta)=x\theta-\psi(\theta)$ as in Example \ref{ex:phase fn dim 1}. Let us compute $\epsilon_{\gamma_\delta}.$ One has $k_2=1.$ If $\varphi_{xx}>0$ then $n_2=0$. If $\varphi_{xx}<0$ then $n_2=1.$ Therefore
$$\epsilon_{\gamma_\delta}=-1\;{\rm for}\; \varphi_{xx}>0;\; \epsilon_{\gamma_\delta}=0 \;{\rm{for}}\; \varphi_{xx}<0.$$

Now compute $\epsilon_{\delta\gamma}.$ One has $k_2=0.$ If $\varphi_{xx}>0$ then $n_2=1$. If $\varphi_{xx}<0$ then $n_2=0.$ Therefore
$$\epsilon_{\delta\gamma}=1\;{\rm{for}}\;\varphi_{xx}>0;\;\epsilon_{\delta\gamma}=0\;{\rm{for}}\;\varphi_{xx}<0.$$
\end{example}
Now note that $d\varphi_\gamma = \xi dx|L$ on $L\cap W_\gamma$ and $d\varphi_\delta = \xi dx|L$ on $L\cap W_\delta.$ Therefore, if $\eta_{\gamma\delta}=\varphi_\gamma-\varphi_\delta$ on  $L\cap W_\gamma\cap W_\delta,$ then $(\eta_{\gamma\delta})$ represents the cohomology class $\eta$ corresponding to the De Rham class of $\xi dx|L.$

On the other hand, a choice of a local presentation \eqref{eq: Lagr gen form loc} of $L$ determines a choice of lifting of transition isomorphisms as in \eqref{eq:lifting p t ij}. Indeed, in a tangent space $T_{(x,\xi)} L$ to a point of $L\cap W_\gamma,$ let $\fx,$ $\fxi$ be formal Darboux coordinates coming from some local coordinate system. Choose a presentation 
\begin{equation}\label{eq:TL in coords}
\fxi_1=A\fx_1+B\fxi_2;\; \fx_2=-C\fx_1-D\fxi_2
\end{equation}
Construct a symplectic matrix sending $L_0=\{\fxi_1=\fxi_2=0\}$ to $T_{(x,\xi)} L$ as follows. Let
\begin{equation}\label{eq:pick a pj}
p(A,B,C,D): (\fx_1, \fx_2, \fxi_1, \fxi_2)\mapsto  (\fx_1,\fx_2,A\fx_1+B\fx_2,C\fx_1+D\fx_2)
\end{equation}
and
\begin{equation}\label{eq:pick a pj 1}
F_{\fx_2} : (\fx_1, \fx_2, \fxi_1, \fxi_2)\mapsto  (\fx_1, -\fxi_2, \fxi_1, \fx_2)
\end{equation}
One has
\begin{equation}\label{eq:pick a pj 2}
T_{(x,\xi)} L=F_{\fx_2} p(A,B,C,D) L_0
\end{equation}
Note also that both factors of the right hand side extend automatically to elements in $\Sp^4.$ Indeed, one can replace $p(A,B,C,D)$ by the homotopy class of the path $p(tA,tB,tC,tD),$ $0\leq t\leq 1,$ and $F_{\fx_2}$ by the homotopy class of the path 
$$(\fx_1, \fx_2, \fxi_1, \fxi_2)\mapsto  (\fx_1,  \fx_2\cos t -\fxi_2 \sin t, \fxi_1, \fx_2 \sin t + \fxi_2 \cos t),\, 0\leq t\leq \frac{\pi}{2}$$
It is easy to see that the Maslov class $\mu$ corresponding to the lifted transition functions thus defined is inverse to the one defined by \eqref{eq:eps correct}.
\section{Appendix. Twisted $A_\infty$ modules and $A_\infty$ functors}\label{appendix:infty and twisted}
\subsection{Differential graded categories of $A_\infty$ functors}\label{ss:A infty categ of A infty functs}
Let $A$ and $B$ be two differential graded (DG) categories. For two maps
$$\bof,\bog:\Ob(A)\to \Ob(B)$$
define
$${\overline{C}}^\bullet_{\bof,\bog} (A,B)= \prod _{{n\geq 1};x_0,\ldots,x_n} \Hom^\bullet (A(x_0,x_1)\otimes\ldots\otimes A(x_{n-1},x_n)[n], B(f(x_0), g(x_n)))$$
where the product is taken over all $x_0,\ldots, x_n\in \Ob(A).$ Put
\begin{equation}\label{eq:full Hoch Ainf}
C^\bullet _{\bof,\bog}(A,B)=\prod_{x_0\in \Ob(A)} B(f(x_0), g(x_0)) \times {\overline{C}}^\bullet_{\bof,\bog} (A,B)
\end{equation}
Define the differential $d$ by
\begin{equation}\label{eq:diffl inside}
(d_1\varphi)(a_1,\ldots,a_{n+1})=\sum_{j=1}^n (-1)^{\sum_{p\leq j} (|a_p|+1)} \varphi (a_1,\ldots, a_ja_{j+1},\ldots, a_{n+1})
\end{equation}
($d_1=0$ on the first factor of \eqref{eq:full Hoch Ainf}); 
\begin{equation}\label{eq:diffl inside 1}
(d_2\varphi)(a_1,\ldots,a_{n})=\sum_{j=1}^n (-1)^{\sum_{p< j} (|a_p|+1)} \varphi (a_1,\ldots, d_Aa_j,\ldots, a_{n+1})+d_B\varphi(a_1,\ldots,a_n)
\end{equation}
Define
$$d=d_1+d_2$$
Also define the product 
$${\overline{C}}^\bullet_{\bof,\bog} (A,B) \otimes {\overline{C}}^\bullet_{\bog,\boh} (A,B) \to {\overline{C}}^\bullet_{\bof,\boh} (A,B)$$
by
\begin{equation}\label{dfn:A infty cup}
(\varphi\smile\psi)(a_1,\ldots,a_{m+n})=(-1)^{|\psi| \sum_{j=1}^m (|a_j|+1)} \varphi(a_1,\ldots,a_m)\psi (a_{m+1},\ldots, a_{m+n})
\end{equation}
(Note that here $m$ or $n$ can be zero, which corresponds to the case of one or both factors lying in the fisrt factor of \eqref{eq:full Hoch Ainf}).
\begin{definition}\label{dfn:A ifty fuct}
An $A_\infty$ functor $f:A\to B$ is a map $f:\Ob(A)\to\Ob(B)$ together with an element $f$ of degree $1$ in ${\overline{C}}^\bullet_{\bof,\bof} (A,B) $ such that 
$$df+f\smile f=0$$
A {\em curved} $A_\infty$ functor is defined the same way but now the cochain $f$ is allowed to be in ${C}^\bullet_{\bof,\bof} (A,B). $
\end{definition}
\begin{definition}\label{dfn:Cat of Ainfty frs}
Define the DG category ${\bf C}(A,B)$ as follows. Let objects be $A_\infty$ functors $\bof:A\to B;$ set
$${\bf C}^\bullet (A,B)(f,g) = C^\bullet_{\bof,\bog}(A,B)$$
with the differential 
$$\delta\varphi= d\varphi +f\smile \varphi - (-1)^{|\varphi|} \varphi \smile f$$
We define the composition to be the cup product.

Also, define the DG category ${\bf C}_+(A,B)$ the same way as above but with objects being curved $A_\infty$ functors.
\end{definition}
\subsubsection{Equivalence of objects in a DG category}\label{sss:ttt} Let ${\bfC}_1$ be the category with two objects $0$ and $1$ and two mutually inverse morphisms $g:0\to 1$ and $g^{-1}:1\to 0.$  
\begin{definition}\label{dfn:equiv of objs} Two objects $\bfx,\,\bfy$ of a DG category $C$ are equivalent if there is an $A_\infty$ functor $\bfC_1\to C$ sending $0$ to $\bfx$ and $1$ to $\bfy.$
\end{definition}
\begin{lemma}\label{lemma:equiv taki equiv} The relation defined above is an equivalence relation.
\end{lemma}
\begin{proof} Let $\bfC_2$ be the category with three objects $0,$ $1,$ $2$ and with unique morphism between any two objects. There are functors $i_{pq}:{\mathbf C}_1\to {\mathbf C}_2$ that send $0$ to $p$ and $1$ to $q$, $0\leq p<q\leq 2.$ If we have one equivalence between $\bfx$ and $\bfy$ and another between $\bfy$ and ${\mathbf z},$ then we have a functor 
\begin{equation}\label{eq:qe}
\Cbr\Br k[i_{01}\bfC_1] *_{k[1]} \Cbr\Br k[i_{12}\bfC_1] \to C
\end{equation}
that sends $0$ to $\bfx,$ $1$ to $\bfy,$ and $2$ to ${\mathbf z}.$ Here $*$ stands for free product of categories; for any category $\bfC,$ $k[\bfC]$ is its linearization, and $k[1]$ is the category with one object $1$ whose ring of endomorphisms is $k.$ But the left hand side of \eqref{eq:qe} is quasi-isomorphic to $ k[i_{01}\bfC_1] *_{k[1]}  k[i_{12}\bfC_1]\isomoto {\bfC}_2.$ By the standard transfer of structure \cite{KellerAinfty}, \cite{KoSoAinfty}, \cite{Merkulov}, we get an $A_\infty$ morphism ${\mathbf C}_2\to C$ that sends $0$ to $\bfx,$ $1$ to $\bfy,$ and $2$ to ${\mathbf z}.$ Composing it with $i_{02},$ we get an equivalence between $\bfx$ and ${\mathbf z}.$
\end{proof}
\begin{definition}\label{dfn:ekviv of A infty finct}
Two $A_\infty$ functors $A\to B$ are equivalent if they are equivalent as objects in ${\mathbf C}(A,B).$
\end{definition}
\subsubsection{The bar construction}\label{sss:bar cons}
The bar construction of a DG category $A$  is a DG cocategory $\Br (A)$ with the same objects where
$${\Br}(A)(x,y)=\bigoplus_{n\geq0}\bigoplus _{x_1,\ldots, x_n}A(x,x_1)[1]\otimes A(x_1,x_2)[1]\otimes \ldots \otimes A(x_n,x)[1]$$
with the differential
$$d=d_1+d_2;$$
$$d_1(a_1|\ldots | a_{n+1})=\sum_{i=1}^{n+1}\pm (a_1|\ldots|da_i|\ldots |a_{n+1});$$
$$d_2(a_1|\ldots | a_{n+1})=\sum_{i=1}^{n}\pm (a_1|\ldots|a_ia_{i+1}|\ldots |a_{n+1})$$
The second sum is taken over $n$-tuples $x_1,\ldots,x_n$ of objects of $A.$
The signs are $(-1)^{\sum_{j<i}(|a_i|+1)+1}$ for the first sum and $(-1)^{\sum_{j\leq i}(|a_i|+1)}$ for the second. The comultiplication is given by
$$\Delta(a_1|\ldots|a_{n})=\sum_{i=1}^{n-1}(a_1|\ldots|a_i)\otimes (a_{i+1}|\ldots |a_{n})$$
Dually, for a DG cocategory $B$ one defines the DG category $\Cbr(B)$. The DG category $\Cbr\Br(A)$ is a cofibrant resolution of $A$.

It is convenient for us to work with DG (co)categories without (co)units. For example, this is the case for $\Br(A)$ and $\Cbr(B)$ (we sum, by definition, over all tensor products with at least one factor). Let $A^+$ be the (co)category $A$ with the (co)units added, i.e. $A^+(x,y)=A(x,y)$ for $x\neq y$ and $A^+(x,x)=A(x,x)\oplus k\id _x.$ If $A$ is a DG category then $A^+$ is an augmented DG category with units, i.e. there is a DG functor $\epsilon: A^+\to k_{\Ob (A)}$. (For a set $I,$ $k_I$ is the DG category with the set of objects $I$ and with $k_I(x,y)=0$ for $x\neq y,$ $k_I(x,x)=k$). Dually, one defines the DG cocategory $k^{\Ob(B)}$ and the DG functor $\eta: k^{\Ob(B)} \to B^+$ for a DG cocategory $B$.

For DG (co)categories with (co)units, define $A\otimes B$ as follows: ${\rm{Ob}}(A\otimes B)={\rm{Ob}}(A)\times {\rm{Ob}}(B);$ $(A\otimes B)((x_1,y_1), (x_2, y_2))=A(x_1,y_1)\otimes B(x_2,y_2);$ the product is defined as $(a_1\otimes b_1)(a_2\otimes b_2)=(-1)^{|a_2||b_1|}a_1a_2\otimes b_1b_2$, and the coproduct in the dual way. This tensor product, when applied to two (co)augmented DG (co)categories with (co)units, is again a (co)augmented DG (co)category with (co)units: the (co)augmentation is given by $\epsilon\otimes \epsilon$, resp. $\eta\otimes \eta.$
\begin{definition}\label{dfn:tensor product without (co)units}
For DG categories $A$ and $B$ without units, put
$$A\otimes B=\Ker(\epsilon\otimes \epsilon: A^+\otimes B^+\to k_{\Ob (A)} \otimes k_{\Ob (B)}).$$
Dually, for DG cocategories $A$ and $B$ without counits, put
$$A\otimes B=\Coker(\eta\otimes \eta:k^{\Ob (A)}\otimes k^{\Ob (B)} \to A^+\otimes B^+).$$
\end{definition}
The following is standard (and straightforward).
\begin{lemma}\label{lemma:bar vs ainf}
There are natural bijections
$$\Ob {\mathbf C}(A,B)\isomoto \Hom  (\Cbr \Br (A), B);$$
$$\Ob {\mathbf C}_+(A,B)\isomoto \Hom (\Cbr \Br^+ (A), B)$$
\end{lemma}
In other words, an $A_\infty$ functor $A\to B$ is the same as a DG functor $\Cbr \Br (A) \to B.$ A curved $A_\infty$ functor $A\to B$ is the same as a DG functor $\Cbr \Br^+ (A) \to B.$
\subsubsection{The adjunction formula}\label{sss:The adjunction}
\begin{lemma}\label{lemma:adjunct} There are natural bijections
$$\Ob {\mathbf C}(A, {\mathbf C}(B,C))\isomoto \Hom_{{\rm {DGcat}}} (\Cbr (\Br^+(A)\otimes \Cbr(B)), C)$$
$$\Ob {\mathbf C}_+(A, {\mathbf C}_+(B,C))\isomoto \Hom_{{\rm {DGcat}}} (\Cbr (\Br^+(A)\otimes \Cbr^+(B)), C)$$
\end{lemma}
This (as well as Lemma \ref{lemma:bar vs ainf}) follows from Lemmas \ref{lemma:adj conv 0}, \ref{lemma:adj conv},  \ref{lemma:adj conv 1} below.
\subsubsection{Convolution categories}\label{sss:convocats} Let $\bB$ be a DG cocategory and $C$ a DG category. For 
$$\bof,\bog: \Ob(\bB)\to \Ob(C),$$
put
$${\overline{\Conv}}_{\bof,\bog} (\bB,C)=\prod_{x,y\in \Ob(\cB)} \Hom^\bullet(\bB(x,y), C(fx,gy))$$ 
$${{\Conv}}_{\bof,\bog} (\bB,C)=\prod_{x,y\in \Ob(\cB)} \Hom^\bullet(\bB^+(x,y), C(fx,gy))$$ 
The differential $d$ is the usual one (induced by the differentials on $\bB$ and $C$). Define the product
$$\Conv_{\bof,\bog}(\bB,C)\otimes \Conv_{\bog,\boh}(\bB,C)\to\Conv_{\bof,\boh}(\bB,C)$$ 
$$\oConv_{\bof,\bog}(\bB,C)\otimes \oConv_{\bog,\boh}(\bB,C)\to\oConv_{\bof,\boh}(\bB,C)$$ 
as follows. If 
$$\Delta b=\sum b^{(1)}\otimes b^{(2)}$$
then
\begin{equation}\label{eq:mult conv}
(\varphi\smile \psi)(b)=\sum (-1)^{|\psi||{b^{(1)}|}}  \varphi(b^{(1)}) \psi(b^{(2)} )
\end{equation}
\begin{definition}\label{dfn:convo cat}
Define DG categories $\bConv(\bB,C)$ and $\bConv_+(\bB,C)$ as follows. Their objects are maps $\bof:\Ob(\bB)\to \Ob(C)$ together with elements $f$ of degree one in $\oConv_{\bof,\bof}(\bB,C)$ (resp. in $\Conv_{\bof,\bof}(\bB,C)$)
satisfying
$$df+f\smile f=0.$$
The complex of morphisms between $f$ and $g$ is $\Conv_{\bof,\bog}(\bB,C)$ with the differential 
$$\delta\varphi=d\varphi+f\smile\varphi-(-1)^{|\varphi|} \varphi \smile f$$
The composition is the cup product \eqref{eq:mult conv}.
\end{definition}
\begin{lemma}\label{lemma:adj conv 0} 
There are natural isomorphisms of DG categories
$$\bC(A,B)\isomoto \bConv(\Br(A),B)$$
$$\bC_+(A,B)\isomoto \bConv_+(\Br(A),B)$$
\end{lemma}
\begin{lemma}\label{lemma:adj conv}
There is a natural bijection
$$\Hom_{\rm{DGcat}}(\Cbr(\bB), C)\isomoto \Ob(\bConv(\bB, C))$$
\end{lemma}
\begin{lemma}\label{lemma:adj conv 1}
There is a natural isomorphism of DG categories 
$$\bConv(\bB_1,\bConv(\bB_2, C))\isomoto\bConv(\bB_1\otimes\bB_2,C)$$
\end{lemma}
This is a reformulation of a result in \cite{Keller}.
\subsubsection{An $A_\infty$ functor to $A_\infty$ modules}\label{sss: a functor into mos} Let $k$ be a field. By $\dgmod(k)$ we denote the differential graded category of complex of modules over $k.$ Let $R$ be an associative algebra over $k$. 
\begin{definition}\label{def:modinft}
We denote the DG category $\bfC(\Br(R), \dgmod(k))$ by $\Modinf (R)$ and call it the DG category of $A_\infty$ modules over $R.$ 
\end{definition}
Let ${\mathfrak X}(R)$ be the category whose objects are pairs $(\cB\stackrel{\pi}{\longrightarrow}R, \cM)$ where $\cB$ is a differential graded algebra, $\pi$ a quasi-isomorphism of DGA, and $\cM$ a DG module over $\cB.$ A morphism $(\cB\stackrel{\pi}{\longrightarrow}R, \cM)\to (\cB'\stackrel{\pi'}{\longrightarrow}R, \cM')$ is a morphism $\cB\to\cB'$ of DGA over $R$ together with a compatible morphism $\cM\to\cM'.$

We will construct an $A_\infty$ functor
\begin{equation}\label{eq:CR to Modinf}
{\mathfrak X}(R)\to \Modinf(R)
\end{equation}
\begin{remark}\label{rmk:A inf func from cat}
An $A_\infty$ functor from a category ${\mathfrak X}$ to a DG category $\cA$ is by definition an $A_\infty$ functor from the linearization of ${\mathfrak X}$ (viewed as a DG category with zero differential) to $\cA.$
\end{remark} 
Define the DG category ${\mathfrak B}$  as follows. Its objects are the same as objects of ${\mathfrak{X}}(R)$ but repeated counably many times, {\em i.e.} an object of ${\mathfrak B}$ is a pair $(\bfx, n)$ where $\bfx$ is an object of ${\mathfrak X}(R)$ and $n\in {\mathbb Z}.$ The spaces of morphisms are as follows. 
\begin{equation}\label{eq:asdfa1}
{\mathfrak B}((\bfx, m), (\bfy, n))=0
\end{equation} 
if $m<n$ or $m=n$ but $\bfx\neq \bfy.$ If $m>n$ and 
\begin{equation}\label{eq:asdfa2}
\bfx=(\cB\stackrel{\pi}{\longrightarrow}R, \cM),\; \bfy=(\cB'\stackrel{\pi'}{\longrightarrow}R, \cM'),
\end{equation} 
then
\begin{equation}\label{eq:asdfa3}
{\mathfrak B}((\bfx, m), (\bfy, n))=\cB'\times {\mathfrak X}(R)(\bfx,\bfy)
\end{equation} 
By $a'{\mathbf b} $ we denote the pair $(a', {\mathbf b} )$ where $a'\in \cB'$ and ${\mathbf b}: \bfx\to \bfy$  is a morphism in ${\mathfrak X}(R).$ We denote the underlying morphism $\cB\to \cB'$ also by ${\mathbf b}.$ Put
\begin{equation}\label{eq:asdfa4}
{\mathfrak B}((\bfx, n), (\bfx, n))=\cB
\end{equation} 
We also denote the right hand side by $\cB \id_{\bfx}.$ The composition is given by
\begin{equation}\label{eq:asdfa5}
(a'' {\mathbf b}')(a'{\mathbf b})=(a''{\mathbf b}'(a')) {\mathbf b}'{\mathbf b}
\end{equation} 
Consider the following right DG module ${\bfM}$ over ${\mathfrak B}.$ Define
$$\bfM(\bfx, n)=\cM$$
where $\bfx=(\cB\to R, \cM).$ Defuine the action
$$\bfM(\bfx, m) \otimes \geB((\bfx, m), (\bfy, n))\to \bfM(\bfy, n)$$
by 
$$v\otimes (a'{\mathbf b})=a' {\mathbf b}(v)$$
Here we denote by ${\mathbf b}$ the underlying action of the morphism ${\mathbf b}:\bfx\to\bfy$ on the module, as well as on the algebra.

Define another DG category $\cR$ exactly like $\geB$ above with the only difference that we put
\begin{equation}\label{eq:asdfa8}
{\cR}((\bfx, m), (\bfy, n))=R\times {\mathfrak X}(R)(\bfx,\bfy)
\end{equation} 
instead of \eqref{eq:asdfa3} and 
\begin{equation}\label{eq:asdfa44}
{\mathfrak B}((\bfx, n), (\bfx, n))=R
\end{equation} 
instead of \eqref{eq:asdfa4}. We also denote the right hand side by $\id_{\bfx} R.$ Instead of \eqref{eq:asdfa5}, the composition is given by
he composition is given by
\begin{equation}\label{eq:asdfa55}
(a'' {\mathbf b}')(a'{\mathbf b})=(a''a') {\mathbf b}'{\mathbf b}
\end{equation} 
The morphisms $\pi: \cB\to R$ induce a quasi-isomorphism of DG categories $\geB\stackrel{\pi}{\longrightarrow}\cR.$ The transfer of structure argument makes $\bfM$ a right $A_\infty$ module over $\cR$ as follows. Fix a linear map $\cR{\stackrel{i}{\longrightarrow}} \cB$ that is inverse to $\pi$ at the level of cohomology. (This is where we use the assumption that $k$ is a field). Fix also homotopies for $\id_\geB-i\pi$ and for $\id_\cR-\pi i.$ (By this we mean collections of maps $\cR(\bfx,\bfy)\to \geB(\bfx,\bfy),$ {\em etc.}, for any objects $\bfx$ and $\bfy$).  From his data one constructs an $A_\infty$ functor $\cR\to \geB$ which is inverse to $\pi$ up to equivalence ({\em cf.}  \cite{KellerAinfty}, \cite{KoSoAinfty}, \cite{Merkulov}). Furthermore, the map $i$ and the homotpopies can be chosen to be invariant under the action of $\bZ$ on $\geB$ and on $\cR.$ Therefore the $A_\infty$ functor is also $\bZ$-invariant. We denote it by ${\mathbb T},$ and the corresponding twisting cochain $\rho$ by $\rho_{\mathbb T}.$

This, in turn, defines the desired $A_\infty$ functor \eqref{eq:CR to Modinf}. In fact, for any object $\bfx=(\cB\to R, \cM),$ the value of this $A_\infty$ functor on $\bfx$ is the underlying complex $\cM.$ For $g_1,\ldots,g_p\in R,$ put 
\begin{equation}\label{eq:rho as ainf}
\rho(g_1,\ldots,g_p)=\rho_{\mathbb T}(g_1\id_\bfx, \ldots,g_p\id_\bfx)
\end{equation}
where we view $\rho_j\id_\bfx$ as morphisms $(\bfx, 0)\to(\bfx, 0)$ in $ \geB.$ This makes each $\cM$ an $A_\infty$ module over $R.$ Now consider morphisms
\begin{equation}
\bfx_0\stackrel{{\mathbf b}_1}{\longleftarrow}\bfx_1\stackrel{{\mathbf b}_2}{\longleftarrow} \ldots \stackrel{{\mathbf b}_n}{\longleftarrow}\bfx_n
\end{equation}
in $\geX (R),$ as well as corresponding morphisms
\begin{equation}
(\bfx_0, 0)\stackrel{{\mathbf b}_1}{\longleftarrow}(\bfx_1, 1) \stackrel{{\mathbf b}_2}{\longleftarrow} \ldots \stackrel{{\mathbf b}_n}{\longleftarrow}(\bfx_n, n)
\end{equation}
in $\cR.$
Now put
$$\rho(g_1,\ldots,g_p)=
\sum\pm \rho_{\mathbb T}(g_1\id_{\bfx_0}, \ldots,g_{p_1}\id_{\bfx_0}, {\mathbf b}_1, 
$$
$$g_{p_1+1}\id_{\bfx_1}, \ldots,g_{p_2}\id_{\bfx_1}, \ldots, {\mathbf b}_n, g_{p_n+1}\id_{\bfx_n}, \ldots,g_{p}\id_{\bfx_n })
$$
where the sum is taken over all $0\leq p_1\leq \ldots \leq p_{n}\leq n.$ The sign rule: both $g_j\id_{\bfx_k}$ and ${\mathbf b}_i$ are treated as odd (the former has degree $(-1)^{|g_j|+1}$ if $R$ is graded). 

It is straightforward to check that thus defined $\rho,$ when viewed as a cochain
$$\rho({\mathbf b}_1,\ldots,{\mathbf b_n}) \in {\rm{Mod}}_\infty (R)(\cM_0,\cM_n),$$
is an $A_\infty$ functor $\geX(R)\to  {\rm{Mod}}_\infty (R).$ (Here $\cM_j$ is the underlying DG module of $\bfx_j,$ viewed as a complex).
\subsection{Twisted $A_\infty$ modules on a space}\label{ss:twisted A inf mos}
Let $\cR$ be a sheaf of algebras on a topological space $X.$ Fix an open cover ${\mathfrak U}$ of $X.$ For two collections  $\bfM=\{\cM_U|U\in {\mathfrak U}\}$ and $\bfN=\{\cN_U|U\in {\mathfrak U}\}$ of sheaves of $\cR_U$-modules, define the complex
$C^\bullet _{\bfM, \bfN}({\mathfrak U})$ as follows. Put
\begin{equation}\label{eq:Twisted Cech complex 1}
C^\bullet _{\bfM, \bfN}({\mathfrak U})=\prod_{p,q=0}^\infty \prod_{U_0,\ldots,U_p\in \geU} \uHom^{\bullet -p-q}(\cR^{\otimes q}, \uHom^{\bullet } (\cN_{U_p}, \cM_{U_0}))(U_0\cap \ldots \cap U_p)
\end{equation}
Define the differentials
\begin{equation}\label{eq:Cech twisted diff}
({\check{\partial}} \varphi)_{U_0\ldots U_{p+1} }=\sum _{j=1}^{p} (-1)^j \varphi_{U_0\ldots {\widehat {U_j}}\ldots U_{p+1}};
\end{equation}
\begin{equation}\label{eq:Cech twisted diff 1}
(\partial \varphi)(g_1,\ldots,g_{q+1})=(-1)^{p|\varphi|} \sum_{j=1}^q\varphi (g_1,\ldots,g_jg_{j+1},\ldots,g_{q+1})
\end{equation}
for local sections $g_1,\ldots$ of $\cR;$
\begin{equation}\label{eq:diffl tw tot 2}
d\varphi={\check{\partial}}\varphi+\partial\varphi+d_\cM\varphi-(-1)^{|\varphi|} \varphi d_\cN
\end{equation}
Define also the product
\begin{equation}\label{eq:prod tw cech}
C^\bullet_{\bfM,\bfN} (\geU) \otimes C^\bullet_{\bfN,\bfP} (\geU) \to C^\bullet_{\bfM,\bfP} (\geU)
\end{equation}
by
$$
(\varphi\smile\psi)_{U_0\ldots U_{p_1+p_2}}(g_1,\ldots ,g_{q_1+q_2})=
$$
$$
(-1)^{{|\varphi|p_2+(|\psi|+p_2)q_1}} \varphi_{U_0\ldots U_{p_1}}(g_1,\ldots,g_{q_1}) \psi_{U_{p_1},\ldots,U_{p_1+p_2}}(g_{q_1+1},\ldots,g_{q_1+q_2})
$$
Set
\begin{equation}\label{eq:Cech tw X}
C^\bullet _{\bfM,\bfN} (X)=\varinjlim _{\geU} C^\bullet _{\bfM,\bfN} (\geU)
\end{equation}
The differential and the cup product are well defined on the above complexes.
\begin{definition}\label{dfn:A inf twi}
A twisted $A_\infty$ module $\cM$ over $\cR$ is a collection $\bfM=\{\cM_U|U\in {\mathfrak U}\},$ of sheaves of $\cR_U$-modules together with a cochain $\rho$ of degree one in $C^\bullet_{\bfM,\bfM}(X)$ such that 
$$d\rho+\rho\smile \rho=0.$$
The DG category ${\rm{Tw}}\Modinf(\cR)$ has twisted $A_\infty$ modules as objects. The complex of morphisms between $\cM=(\bfM,\rho)$ and $\cN=(\bfN, \sigma)$ is the complex $C^\bullet _{\bfM,\bfN}(X)$ with the differential $\delta\varphi=d\varphi+\rho\smile \varphi-(-1)^{|\varphi|} \varphi \smile\sigma.$
\end{definition}
The above definition is an extension of the definition of twisted cochains from \cite{TT}. {\em Cf.} also \cite{OTT}, \cite{BGNT1}, and \cite{Wei}.
\begin{remark}\label{rmk:tw coch vs conv} The DG category of twisted $A_\infty$ modules is obtained almost {\em verbatim} as a partial case of the left hand side of Lemma \ref{lemma:adjunct} . Formally, one could choose $B$ to be the category with one object whose complex of morphisms is $\cR,$ and $A={\rm{Op}}_X$ to be the category of open subsets of $X.$ More precisely, we perform all the computations as if $A$ were the category whose objects are open subsets $U_\alpha,$ and there is one morphism $U_\alpha\to U_\beta$ for any two intersecting open subsets. This is not literally true (there may be nonempty intersections $U_\alpha\cap U_\beta$ and $U_\beta\cap U_\gamma$ but not $U_\alpha\cap U_\gamma$), but all the formulas work. The above motivation may be given rigorous meaning using the language of *** as in \cite{Gillette} or in \cite{BGNT1}.
\end{remark}
\subsection{Twisted $A_\infty$ modules over groupoids}\label{ss:twisted A inf mos gros}
For $q\geq 0,$ we use notation ${\mathbb U}=(U^{(0)}, \ldots, U^{(q)}).$ We denote by $\geU_q$ the set of all such ${\mathbb U}$ where $U_j$  is in a given open cover $\geU.$ For $p+1$ such $q$-tuples ${\mathbb U}_{j_0}, \ldots, {\mathbb U}_{j_p},$  denote
\begin{equation}\label{eq:gababgs}
U^{(k)}_{j_0\ldots j_p}=U^{(k)}_{j_0}\cap\ldots\cap U^{(k)}_{j_p}
\end{equation}
for all $0\leq k \leq q.$ Denote also
\begin{equation}\label{eq:ptup q-tup}
{\mathbb U}_{j_0\ldots j_p}=(U^{(0)}_{j_0\ldots j_p},\ldots,U^{(q)}_{j_0\ldots j_p}).
\end{equation}
Here $\cM_{U^{(0)}_0}$ stands for its inverse image under the map 
$$\prod_k \cap_j U_j^{(k)} \to \prod_k U_0^{(k)}\to U^{(0)}_0$$ 
Let $\Gamma$ be an \'{e}tale groupoid on a manifold $X$ (in our applications, $\Gamma=\pi_1(X)$). For $\bfM=\{\cM_U|U\in \geU\}$ and $\bfN=\{\cN_U|U\in \geU\}$ as in the beginning of \ref{ss:twisted A inf mos}, put
$$C^\bullet _{\bfM,\bfN}(\geU,\Gamma)=\prod_{p,q\geq 0} \prod _{{\mathbb U}_0,\ldots,{\mathbb U}_p \in \geU_q} \uHom^{\bullet-p-q}({\underline{\Gamma}}^{(q)}, \uHom^\bullet(\cN_{U_p^{(q)}},\cM_{U^{(0)}_0})(\prod _{k=0}^q U^{(k)}_{01\ldots p}) $$
The differential and the cup product are defined exactly as in \eqref{eq:prod tw cech}, \eqref{eq:Cech twisted diff 1}, \eqref{eq:Cech twisted diff} (with $U_j$ replaced by ${\mathbb U}_j$). Define 
\begin{equation}\label{eq:ind lim gpd}
C^\bullet _{\bfM,\bfN}(X,\Gamma) =\varinjlim_{\geU} C^\bullet _{\bfM,\bfN}(\geU,\Gamma)
\end{equation}
\begin{definition}\label{dfn:Tw mods groupd}
a) Define the DG category ${\rm Tw}\Modinf (\Gamma)$ exactly as in Definition \ref{dfn:A inf twi} usinng complexes $C^\bullet _{\bfM,\bfN}(X,\Gamma).$ 

b) The DG category 
$${\rm Tw}\Modinf (\Gamma, \Omega^\bullet _{\bK, X})$$ 
is defined the same way but with $\cM_U$ being $\Omega^\bullet _{\bK, U}$-modules as in \ref{ss:Inverse Images}. 
\end{definition} 
\begin{remark}\label{rmk:tw loc sy via} 
By
$${\rm{Loc}}_{\infty,\bK}(X)$$ 
we denote the DG category of $A_\infty$ representations of the fundamental groupoid $\pi_1(X).$ This is the partial case of the above Definition \ref{dfn:Tw mods groupd}, a) when $\Gamma=\pi_1(X),$ the topology on $X$ is {\em discrete}, and the ground ring is $\bK.$ Objects of this DG category are infinity local systems as in \ref{ss:Infty loc systs Kmods}. 
\end{remark}
\subsubsection{From $\cA_M^\bullet$-modules with an action of $\pi_1(M)$ up to inner automorphisms to twisted $(\Omega_{\bK,M}^\bullet,\pi_1(M))$-modules}\label{sss:from up to to tw} Given two $\cA_M^\bullet$-modules $\cV^\bullet$ and $\cW^\bullet$ with an action of $\pi_1(M)$ up to inner automorphisms, consider the standard complex 
$$\cM=\cC^\bullet(\cV^\bullet, \cA^\bullet, \cW^\bullet).$$ As it is shown in \ref{eq:como and Yo}, $\cM$ has the following structure.

For a number of open subsets $U^{(j)}$ indexed by $j\in J,$ write ${\mathbb U}_{ij}=(U^{(i)},U^{(j)}). $ We have constructed:

a) For every $U^{(0)}$ and $U^{(1)},$ an $\Omega^\bullet_{\bK, U^{(0)}\times U^{(1)}}$-module $\cB_{\mathbb U_{01}}$ together with a quasi-isomorphism 
\begin{equation}\label{eq:quiz B to pi1}
\cB_{\mathbb U_{01}}\to \bK{\underline {\pi_1}}(M)|(U^{(0)}\times U^{(1)});
\end{equation}
b) a morphism 
\begin{equation}\label{eq:comp on B curve}
p_{01}^*\cB_{{\mathbb U}_{01}}\otimes p_{12}^*\cB_{{\mathbb U}_{12}} \to p_{02}^*\cB_{{\mathbb U}_{02}}
\end{equation}
which commutes with the composition on ${\underline {\pi_1}}(M)$ under \eqref{eq:quiz B to pi1};

c) for any $U_0^{(j)}$ and $U^{(j)}_1,$ an isomorphism
\begin{equation}\label{eq:trans for B cu}
{\mathbf b}_{01}: \cB_{{\mathbb U}_0}\otomosi \cB_{{\mathbb U}_1}
\end{equation}
that commutes with \eqref{eq:quiz B to pi1} and \eqref{eq:comp on B curve} and satisfies 
$${\mathbf b}_{01}  {\mathbf b}_{12} = {\mathbf b}_{02} $$
on the intersections.

Now repeat the procedure from \ref{sss: a functor into mos}, together with Remark \ref{rmk:tw coch vs conv}, in the above context. First note that the constructions of \ref{sss: a functor into mos} can be carried out in the case when $R$ is a category (and all $\cB$ are DG categories with the same objects). Now act as if $R$ were the category with objects $U^{(j)}$, with 
$$R(i,j)={\underline {\pi_1}}(M)|(U^{(i)}\times U^{(j)})$$
and the composition being the one on ${\underline{\pi_1}}.$ Now, let ${\rm{Op}}_M$ be the category whose objects are open subsets $U_j$, exactly as discussed in Remark \ref{rmk:tw coch vs conv}. View the data a), b), c) above as a DG functor ${\rm{Op}}_M \to {\mathfrak X}(R).$ Applying formulas from \ref{sss: a functor into mos}, we get an $A_\infty$ functor ${\rm{Op}}_M\to\bfC(R,{\rm{dgmod}}(\bK)),$ which is the same as an $\Omega^\bullet_{\bK,M}$-module with a twisted action of $\pi_1(M).$
\subsubsection{From twisted $( \Omega^\bullet _{\bK, X}, \pi_1(X))$ modules to infinity local systems}\label{sss:From twisted mods to inf loc sys} 
Here we extend the construction from \ref{sss:from ommods to locsyss}. Consider all open covers of the type $\geU=\{U_x|x\in X\}.$ For an object $\cM$ of ${\rm Tw}\Modinf (\pi_1(X), \Omega^\bullet _{\bK, X})$ choose a cover $\geU$ as above and define
\begin{equation}\label{eq:indlim for twisted mods}
{\cM_x}=\varinjlim_{U\subset U_x} C^\bullet (U, \cM_{U_x})
\end{equation}
The $A_\infty$ operators $T(g_1,\ldots,g_n)$ are by definition $\rho_{\mathbb U}(g_1,\ldots,g_n)$ where $g_j\in \pi_1(X)_{x_{j-1},x_{j}}$ and ${\mathbb U}=(U_{x_0},\ldots,U_{x_n})$. Let us show that different choices of $\geU$ lead to equivalent infinity local systems (in the sense of Definition \ref{dfn:ekviv of A infty finct}). Choose two covers $\geU'$ and $\geU ''.$ Apply \eqref{eq:indlim for twisted mods} to all covers of the form $\geU=\{U_x| x\in X\}$ where for any $x$ either $U_x=U'_x$ or $U_x=U_x''.$ This data defines an $A_\infty$ functor $\bK\bfC_1\otimes \bK\pi_1(X)\to {\rm{dgmod}}(\bK)$ ({\em cf.} \ref{sss:ttt}). Let $\bK(0),$ resp. $\bK(1),$ be the full subcategory of $\bfC_1$ with one object $0,$ resp. $1.$ When restricted to $\bK(0)$, resp. to $\bK(1),$ our $A_\infty$ functor coincides with the infinity local system obtained from $\geU',$ resp. from $\geU''.$ By the adjunction formula (Lemma \ref{lemma:adjunct} ), the two infinity local systems are equivalent.
\begin{remark}\label{rmk:A inf func promis} It is easy to modify the above construction and obtain an $A_\infty$ functor
$${\rm{TwMod}}(\Omega^\bullet_{\bK,X}, \pi_1(X))\to {\rm{Loc}}_{\infty,\bK}(X).$$
Moreover, the right hand side is a monoidal category up to homotopy, and the assignment $\cM, \cN\mapsto \uRHOM(\cM,\cN)$ turns oscillatory modules, as well as $\Omega^\bullet_{\bK,M}$-modules with an action of $\pi_1(M)$ up to inner automorphisms, into a category enriched over it. The main reason for this is Lemma \ref{lemma:coproduct Bc}. We will provide the details in a subsequent work.
\end{remark} 
\section{Appendix. Jets and twisted modules}\label{s:app d} Here we will describe the deformation quantization and the twisted bundle $\cH_M$ in terms of bundles of jets. 
\subsection{Jet bundles}\label{ss:jet bdles} Let $M$ be any manifold and let $\cE$ be a complex vector bundle of rank $N$ on $M$. Here we recall the construction of the bundle whose fiber at a point $x$ is the space of jets of sections of $\cE$ at $x.$ This bundle has the canonical connection; its horizontal sections are determined by sections $s$ of $\cE$. The value of such a section at any $x$ is the jet of $s$ at $x.$

Let $\{U_\alpha\}$ is an open cover and $x_\alpha=(x_{\alpha, 1}, \ldots, x_{\alpha, n})$ a local coordinate system on $U_\alpha.$ For $x\in U_{\alpha}\cap U_\beta,$ we denote by $x_\alpha$, resp. $x_\beta,$ its coordinates in the corresponding coordinate system and write
\begin{equation}\label{eq:gaba}
x_\alpha=g_{\alpha\beta}(x_\beta)
\end{equation}
Let $h_{\alpha\beta}: U_{\alpha}\cap U_\beta \to \GL_N$ be the transition isomorphisms of $\cE.$ We identify a local section of $\cE$ on $U_{\alpha}\cap U_{\beta}$ with a $\bC^N$-valued function in the coordinates $x_\beta.$

Let $\C^N[[\fx]]=\C^N[[\fx_1,\ldots,\fx_n]].$ For $x\in U_\alpha$ define $G_{\beta\alpha}(x): \C^N[[\fx]]\to \C^N[[\fx]]$ by $G_{\beta\alpha}(x):f_\alpha\mapsto f_\beta$
where
\begin{equation}\label{eq:transijets}
f_\beta(\fx)=h_{\alpha\beta}(x_\beta+\fx) f_\alpha (g_{\alpha\beta}(x_\beta+\fx)-x_\alpha)
\end{equation}
It is easy to see that different choices of covers and of local trivializations lead to isomorphic bundles. We denote the bundle defined by \eqref{eq:transijets} by $\Jets(\Gamma(\cE)).$

The canonical flat connection is given in any local coordinate system by
\begin{equation}\label{eq:canco jets gen}
\nabla_{\rm{can}}=(\frac{\partial}{\partial x_\alpha}-\ddfx)dx_\alpha
\end{equation}
If a local section of $\cE$ is represented by a vector-valued function $f(x_\alpha),$ it defines a horizontal section which is given in local coordinates by $f(x_\alpha+\fx).$
\subsection{Real polarization}\label{ss:real pol} 

Recall that a real polarization is an integrable distribution of Lagrangian subspaces. Let $\caP$ be a real polarization on $M$. In this case, automatically $2c_1(TM)=0$ modulo $4$ (cf. \cite{S}). 
\subsubsection{The line bundle $\cL$}\label{sss:the line bdle L} Assume that $\omega$ admits a real polarization $\cP$ (i.e. a foliation by Lagrangian submanifolds). By $T_\cP$ we denote the quotient of $TM$ by the subbundle of vectors tangent to the leaves. Choose local Darboux coordinates $\xi_\alpha,x_\alpha$ such that $x_{j,\alpha}$ are constant along the leaves. Then the transition coordinate changes are of the form 
\begin{equation}\label{eq:coord change polariz case}
x_\alpha=g_{\alpha\beta}(x_\beta);\; \xi_\alpha=({{g'_{\alpha\beta}}(x_\beta)}^{t})^{-1}(\xi_\beta+\varphi_{\alpha\beta}(x_\beta))
\end{equation}
Assume that $-i\omega$ is a $2\pi i\mathbb Z$-valued cohomology class. Construct explicitly the line bundle $\cL$ such that $c_1(\cL)=i\omega.$ Adding some constants to $\varphi_{\alpha\beta},$ we may assume that $i\varphi_{\alpha\beta}-i \varphi_{\alpha\gamma}+i\varphi_{\beta\gamma}\in 2\pi i{\mathbb Z};$ define $\cL$ to be the line bundle with transition isomorphisms $\exp(i\varphi_{\alpha\beta}).$ Formulas
\begin{equation}\label{eq:xi dx conn}
A_\alpha=-i\xi_\alpha dx_\alpha
\end{equation}
define a connection in this bundle, since 
$$\xi_\alpha dx_\alpha=\xi_\beta dx_\beta+d\varphi_{\alpha\beta};$$
the curvature of this connection is $-i\omega.$
\subsubsection{The jet bundle $\Jets(\Gamma_{\hor}(\Omhalf \otimes \cL^k))$}\label{sss:jets one h} Define for $x\in U_\alpha \cap U_\beta$ 
$$G_{\beta\alpha}(x): \bC[[\fx]]\to \bC[[\fx,\hbar]]$$
by $(G_{\beta\alpha} f_\alpha)(\fx)=f_\beta(\fx)$ where
\begin{equation}\label{eq: jets Lk re}
f_\beta(\fx)=\det g'_{\alpha\beta}(x_\beta+\fx)^{\frac{1}{2}}  e^{ik\varphi_{\alpha\beta}(x_\beta+\fx)} f_\alpha (g_{\alpha\beta}(x_\beta+\fx)-x_\alpha)
\end{equation}
The square root of the determinant comes from the metalinear structure. The above formula defines the transition functions for the bundle of jets of $\cP$-horizontal sections of the bundle $(\wedge^{\max}T^*_\cP)^{\frac{1}{2}} \otimes \cL^k.$
\subsubsection{The jet bundle ${\rm{Rees}} \Jets \,D(\Gamma_{\hor}(\Omhalf \otimes \cL^{\frac{1}{\hbar}}))$}\label{sss:jets RD one h} Recall the construction of the Rees ring and the Rees module \cite{Borel} of a filtered ring and a filtered module. If $A$ is a ring with an increasing filtration $F_p A,$ $p\geq 0,$ and $V$ an $A$-module with a compatible filtration $F_p,$ $p\geq 0,$ we put
\begin{equation}\label{eq:Rees dfn}
\Rees A=\oplus_{p\geq 0} \hbar ^p F_p A ; \;\Rees V=\oplus_{p\geq 0} \hbar ^p  F_p V .
\end{equation}
\begin{equation}\label{eq:Rees dfn f}
\Rees_f A=\prod_{p\geq 0} \hbar ^p F_p A ; \;\Rees_f V=\prod_{p\geq 0} \hbar ^p  F_p V .
\end{equation}
When applied to the ring of formal differential operators with its filtration by order, \eqref{eq:Rees dfn} produces the ring $\C[[\fx]][\fxi,\hbar]$ with the usual Heisenberg relations ($\fxi_j=i\hbar \frac{\partial}{\partial \fx_j}$). When applied to the module of formal functions $V=\C[[\fx]]$ whose filtration is given by $F_0 V=V,$ it gives $\C[[\fx]][ \hbar].$ The completed version \eqref{eq:Rees dfn f} produces the complete Weyl algebra $\C[[\fx,\fxi,\hbar]]$ and the complete module $\C[[\fx, \hbar]].$

Observe that in the expression $G_{\beta\alpha} (i\hbar \frac{\partial}{\partial \fx}) G_{\alpha\beta}$ one can substitute $\frac{1}{i\hbar}$ for $k$. The result will be given (in vector/matrix notation) by the following:
$$\frac{1}{2}( i\hbar \frac{\partial}{\partial \fx}) ({{g'_{\alpha\beta}}(x_\beta+\fx)}^{t})+({{g'_{\alpha\beta}}(x_\beta+\fx)}^{t})(i\hbar \frac{\partial}{\partial \fx})-\varphi'_{\alpha\beta}(x_\beta+\fx)$$
Define the bundle of algebras ${\rm{Rees}}\, \Jets \,D(\Gamma_{\hor}(\Omhalf \otimes \cL^{\frac{1}{\hbar}}))$ whose fiber is $\C[[\fx, \hbar]][\fxi]$ and whose transition isomorphisms are
\begin{equation} \label{eq:trafu RD}
G_{\beta\alpha}(\fx)=g_{\beta\alpha}(x_\alpha+\fx)-x_\beta;
\end{equation}
\begin{equation} \label{eq:trafu RD 1}
G_{\beta\alpha}(\fxi)= {{g'_{\alpha\beta}}(x_\beta+\fx)}^t * \fxi - \varphi'_{\alpha\beta}(x_\beta+\fx)
\end{equation}
(the multiplcation in the left hand side is the (matrix) Moyal-Weyl multiplication). We see that our bundle is the result of formally substituting $\frac{1}{\hbar} $ for $k$ in the bundle of jets of Rees rings of $\cP$-horizontal differential operators on $(\wedge^{\max}T^*_\cP)^{\frac{1}{2}} \otimes \cL^k.$

The above formula is the result of formally substituting $k$ by $\frac{1}{\hbar}$ into the transition functions for the bundle 
$${\rm{Rees}}\,\Jets \,D(\Gamma_{\hor} ((\wedge^{\max}T^*_\cP)^{\frac{1}{2}} \otimes \cL^k)).$$ 
\subsubsection{The bundle of algebras ${\widehat \bA}_M$ and the twisted bundle of modules $\cH_M$}\label{sss:gauge trareal} Now apply to the bundle above the gauge transformation \cite{NT3}
\begin{equation}\label{eq:gauge transf Ad}
\Ad \exp(\frac{1}{i\hbar}\xi_\alpha \fx)
\end{equation} 
We get transition isomorphisms
\begin{equation} \label{eq:trafu RD 2}
G_{\beta\alpha}(\fx)=g_{\beta\alpha}(x_\alpha+\fx)-x_\beta;
\end{equation}
\begin{equation} \label{eq:trafu RD 12}
G_{\beta\alpha}(\fxi)= {{g'_{\alpha\beta}}(x_\beta+\fx)}^t *( \fxi+\xi_\alpha) - \varphi'_{\alpha\beta}(x_\beta+\fx)-\xi_\beta
\end{equation}
Unlike in \eqref{eq:trafu RD} and \eqref{eq:trafu RD 1}, these transition isomorphisms preserve the maximal ideal $\langle \fx,\,\fxi, \hbar\rangle$ and therefore extend to the complete Weyl algebra ${\widehat \bA}=\C[[\fx,\fxi,\hbar]],$ {\em cf.} \ref{ss:Deformation quantization of a formal neighborhood} . We use them to construct a bundle of algebras $\bA_M$ whose fiber is the Weyl algebra ${\widehat \bA}.$ We see immediately that the bundle of algebras ${\widehat \bA}_M$ is a deformation of the bundle of jets of functions on $M.$

Moreover, after we apply the gauge transformation \eqref{eq:gauge transf Ad}, the formula \eqref{eq: jets Loh re} allows to replace $k$ by $\frac{1}{\hbar}$. We get new transition isomorphisms
\begin{equation}\label{eq: jets Loh re}
f_\beta(\fx)=\det g'_{\alpha\beta}(x_\beta+\fx)^{\frac{1}{2}}  e^{-\frac{1}{i\hbar}(\varphi_{\alpha\beta}(x_\beta+\fx)-\varphi'_{\alpha\beta}(x_\beta)\fx)} f_\alpha (g_{\alpha\beta}(x_\beta+\fx)-x_\alpha)
\end{equation}
that define a twisted bundle of modules $\cH_M$ whose fiber is the space $\cH$ of the formal metaplectic representation ({\em cf} \eqref{eq:alg Weil}). The cocycle $c$ from the definition of a twisted module (\eqref{eq:tw coc}) is $\exp(\frac{1}{i\hbar}(\varphi_{\alpha\beta}-\varphi_{\alpha\gamma}+\varphi_{\beta\gamma}))$. (The summand $-\varphi'_{\alpha\beta}(x_\beta)\fx$ in the exponent comes from the difference of $\xi_\alpha\fx$ and $\xi_\beta\fx$ that figure in the gauge transformation).

In other words, the bundle of algebras $\fbA_M$ can be formally described as 
\begin{equation}\label{eq:Algra jets}
\fbA_M={\rm{Rees}}_f\, {\rm{Jets}}\, D _{\rm{hor}}((\wedge^{\frac{1}{2} }T_\cP^*)^{\frac{1}{2}}\otimes \cL^{\frac{1}{\hbar}})
\end{equation}
\begin{equation}\label{eq:Modu jets}
\cH_M={\rm{Rees}}_f\, {\rm{Jets}}\, \Gamma_{\rm{hor}}((\wedge^{\frac{1}{2} }T_\cP^*)^{\frac{1}{2}}\otimes \cL^{\frac{1}{\hbar}})
\end{equation}
({\em cf.} \eqref{eq:Rees dfn f} for the meaning of ${\bf{Rees}}_f$). The latter is only a twisted bundle because the transition functions of $\cL$ stop being a one-cocycle when elevated to the power $\frac{1}{\hbar}.$
\subsubsection{The canonical connections}\label{sss:can con re} The bundle of horizontal sections of $\Omhalf\otimes\cL^k$ has a canonical connection that is given by the formula
$$\nabla=(\ddx-\ddfx)dx+\ddxi d\xi$$
in all local coordinate systems.

This connection induces a connection in $\Jets {\rm{Rees}}\,D (\Gamma_{\hor}(\Omhalf\otimes \cL^{\frac{1}{i\hbar}}))$ that is given by the same formula. After the gauge transformation from \ref{sss:gauge trareal} we get flat connections
\begin{equation}\label{eq:fla con A re}
\nabla_{\bA}=(\ddx-\ddfx)dx+(\ddxi-\ddfxi) d\xi
\end{equation}
in $\cA_M$ and 
\begin{equation}\label{eq:fla con V re}
\nabla_{\cH}=-\oih \xi dx+(\ddx-\ddfx)dx+(\ddxi+\oih \fx) d\xi
\end{equation}
\subsection{Complex polarization}\label{ss:complex pol} The following is largely based on the approach to deformation quantization from \cite{Kara}.
\subsubsection{K\"{a}hler potentials}\label{sss:Kahler pote} Let $M$ be  a K\"{a}hler manifold. We can locally choose a K\"{a}hler potential, {\em i.e.} a real-valued function $\Phi$ such that the symplectic form is given by 
$$\omega=-i \partial{\overline{\partial}}\Phi$$
A K\"{a}hler potential is unique up to a change $\Phi\mapsto \Phi+\varphi+{\overline{\varphi}}$ where $\varphi$ is holomorphic.
\begin{lemma}\label{lemma:cplex Darboux coords}
Put $\zeta_j=i\frac{\partial \Phi}{\partial z_j}.$ Then
$$\{z_j,z_k\}=0;\; \{\zeta_k, z_j\}=\delta_{jk};\; \{\zeta_j,\zeta_k\}=0.$$
\end{lemma}
\begin{proof} Choose local holomorphic coordinates and put
$$A_{jk}=\frac{\partial}{\partial z_j}\frac{\partial}{\partial \bz_k}\Phi(z,\bz)$$
We have
$$\{z_j,\bz_k\}=i (A^{-1})_{kj};$$
$$\{z_j, \zeta_k\}=i\sum\frac{\partial\zeta_k}{\partial \bz_l} \{z_j,\bz_l\}=\sum A_{kl}(A^{-1})_{lj}= \delta_{jk};$$
$$-\{\zeta_j,\zeta_k\}=\sum ( \frac{\partial\zeta_j}{\partial z_p} \frac{\partial \zeta_k}{\partial \bz_q}-\frac{\partial\zeta_k}{\partial z_p} \frac{\partial \zeta_j}{\partial \bz_q})\{z_p,\bz_q\}=$$
$$i\sum ( \frac{\partial^2\Phi}{\partial z_j \partial z_p} A_{kq}- \frac{\partial^2\Phi}{\partial z_k \partial z_p} A_{jq})(A^{-1})_{qp}=i(\frac{\partial^2\Phi}{\partial z_j \partial z_k}-\frac{\partial^2\Phi}{\partial z_k\partial z_j})=0$$
\end{proof}
\subsubsection{The line bundle $\cL$}\label{sss:line L comp} Choose an open cover $ \{U_{\alpha}\}$ of $M$ and a holomorphic coordinate system $z_\alpha=(z_{\alpha, 1},\ldots,z_{\alpha,n})$ on every $U_{\alpha}.$ We write 
\begin{equation}\label{eq:z transi}
z_\alpha=g_{\alpha\beta}(z_\beta).
\end{equation}
Choose local K\"{a}hler potentials $\Phi_\alpha.$ We have 
\begin{equation}\label{eq: Phi phi}
i\Phi_\alpha -i \Phi_\beta = \varphi_{\alpha\beta}+\overline{\varphi_{\alpha\beta}}
\end{equation}
where $\varphi_{\alpha\beta}$ are holomorphic.

Let us start with rewriting the transition isomorphisms in terms of the new complex Darboux coordinates $z,\zeta.$ We have
$$i\Phi_\alpha (z_\alpha)-i \Phi_\beta((z_\beta) = \varphi_{\alpha\beta}+\overline{\varphi_{\alpha\beta}(z_\beta)}$$
Applying $\frac{\partial}{\partial z_\beta},$ we get
$$\frac{\partial z_\alpha}{\partial z_\beta}i\frac{\partial\Phi}{\partial z_\alpha}(z_\alpha)-i\frac{\partial \Phi}{\partial z_\beta}(z_\beta)=\frac{\partial \varphi_{\alpha\beta}}{\partial z_\beta}(z_\beta)$$
or
\begin{equation}\label{eq:zeta transi}
\zeta_\alpha=(g'_{\alpha\beta}(z_\beta)^{-1})^t(\zeta_\beta+\frac{\partial \varphi_{\alpha\beta}}{\partial z_\beta}(z_\beta))
\end{equation}
Together with \eqref{eq:z transi}, this describes the rule for the change of new variables.

Assume that $i(\varphi_{\alpha\beta}+\varphi_{\beta\gamma}-\varphi_{\alpha\gamma})$ is a $2\pi i{\mathbb Z}$-valued two-cocycle. We get the line bundle $\cL$ with transition functions $\exp(\varphi_{\alpha\beta}).$ The curvature of this connection is $-i\omega.$ 
\subsubsection{The jet bundles}\label{sss:jet bundle holom} Assume that the canonical sheaf has a square root $\Omega^{\frac{1}{2}}.$ We call this line bundle the bundle of holomorphic half-forms on $M.$ The transition isomorphisms of this line bundle are denoted by $\det {g_{\alpha\beta}'}^\half.$ For any integer $k,$ consider the bundle ${\rm{Jets}}(\Gamma_{\rm{hol}}(\cL^k\otimes \Omega^{\frac{1}{2}}))$ of jets of holomorphic sections of $\cL^k\otimes \Omega^{\frac{1}{2}}.$ The fiber of this bundle is $\bC[[\fz]]$ where $\fz=(\fz_1,\ldots,\fz_n).$
The transition isomorphisms of the jet bundle take a power series $f_\alpha(\fz)$ to a power series $f_\beta(\fz)$ according to the following formula.
\begin{equation}\label{eq:jet bdle k holo}
f_\beta (\fz)=f_\alpha (g_{\alpha\beta}(z_\beta+\fz)-z_\alpha) \det g'_{\alpha\beta}(z_\beta+\fz)^{\frac{1}{2}} \exp(k\varphi _{\alpha\beta}(z_\beta+\fz))
\end{equation}

Exactly as in \ref{sss:jets RD one h}, 
we can define the bundle of algebras whose fiber is $\C[[\fz, \hbar]][\fzeta]$ by transition isomorphisms
\begin{equation} \label{eq:trafu RD hol}
G_{\beta\alpha}(\fz)=g_{\beta\alpha}(z_\alpha+\fz)-z_\beta;
\end{equation}
\begin{equation} \label{eq:trafu RD 1 hol}
G_{\beta\alpha}(\fzeta)= {{g'_{\alpha\beta}}(z_\beta+\fz)}^t * \fzeta- \partial_{z_\beta} \varphi_{\alpha\beta}(z_\beta+\fzeta)
\end{equation}
We see that our bundle is the result of formally substituting $\frac{1}{\hbar} $ for $k$ in the bundle of jets of Rees rings of holomorphic differential operators on $\Omega^{\frac{1}{2}} \otimes \cL^k$ (if we map $\fzeta_i$ to $i\hbar \partial_{\fz}$). On the other hand, because of \eqref{eq:zeta transi}, this bundle of algebras is a deformation of the bundle of jets of $C^\infty$ functions on $M.$ The gauge transformation
\begin{equation}\label{eq:gauge transf Ad hol}
\Ad \exp(\frac{1}{i\hbar}\fzeta_\alpha \fz)
\end{equation} 
produces new transition functions
\begin{equation} \label{eq:trafu RD holga}
G_{\beta\alpha}(\fz)=g_{\beta\alpha}(z_\alpha+\fz)-z_\beta;
\end{equation}
\begin{equation} \label{eq:trafu RD 1 holga}
G_{\beta\alpha}(\fzeta)= {{g'_{\alpha\beta}}(z_\beta+\fz)}^t * (\fzeta+\zeta_\alpha)- \partial_{z_\beta} \varphi_{\alpha\beta}(z_\beta+\fzeta)-\zeta_\beta
\end{equation}
that extend to $\fbA_M=\C[[\fz, \fzeta, \hbar]].$ The transition isomorphisms for the module of jets \eqref{eq:jet bdle k holo} become (when we replace $k$ by $\frac{1}{\hbar}$) which now define only a twisted module that we denote by $\cH_M.$
$$
f_\beta (\fz)=f_\alpha (g_{\alpha\beta}(z_\beta+\fz)-z_\alpha) \det g'_{\alpha\beta}(z_\beta+\fz)^{\frac{1}{2}} \exp(\frac{1}{i\hbar}\varphi _{\alpha\beta}(z_\beta+\fz)-\partial_{z_\beta} \varphi _{\alpha\beta}(z_\beta)\fz)
$$ 
As in the case of a real polarization, the canonical connections become
\begin{equation}\label{eq:fla con A re 1}
\nabla_{\bA}=(\ddz-\ddfz)dz+(\ddzeta-\ddfzeta) d\zeta
\end{equation}
in $\cA_M$ and 
\begin{equation}\label{eq:fla con V re 3}
\nabla_{\cH}=-\oih \zeta dz+(\ddz-\ddfz)dz+(\ddzeta+\oih \fz) d\zeta
\end{equation}
on $\cH_M.$  We conclude that
\begin{equation}\label{eq:fffdddggg}
\fbA_M\isomoto{\rm{Rees}}_f\, {\rm{Jets}}\,D_{\rm{hol}} (\Omega ^{\frac{1}{2}}\otimes \cL^{\frac{1}{\hbar}})
\end{equation}
\begin{equation}\label{eq:fffdddggg1}
\cH_M\isomoto {\rm{Rees}}_f\,{\rm{Jets}}\,\Gamma_{\rm{hol}} (\Omega ^{\frac{1}{2}}\otimes \cL^{\frac{1}{\hbar}})
\end{equation}
\end{document}